\documentclass{article}

\usepackage[T2A, T1]{fontenc}
\usepackage[utf8]{inputenc}
\usepackage[english]{babel}

\usepackage[letterpaper,top=2cm,bottom=2cm,left=3cm,right=3cm,marginparwidth=1.75cm]{geometry}

\usepackage{amsmath, amsfonts, amssymb}
\usepackage{graphicx}
\usepackage{enumitem}
\usepackage[dvipsnames]{xcolor}
\usepackage{tikz,tikz-cd}
\usepackage{pgfkeys}
\usepackage[colorlinks=true, allcolors=blue]{hyperref}
\usepackage{adjustbox}

\usetikzlibrary{calc}

\usepackage{amsthm}

\usepackage{lmodern}
\DeclareMathSizes{5}{5}{3.5}{2.5}
\setlist{nosep}

\DeclareTextSymbolDefault{\CYRD}{T2A}
\newcommand{\Z}{\mathbb{Z}}
\newcommand{\N}{\mathbb{N}}
\newcommand{\De}{\text{\CYRD}}
\newcommand{\sln}{\mathfrak{sl}(n)}
\newcommand{\mulg}{\mathfrak{gr}}
\newcommand{\id}{\mathrm{id}}
\newcommand{\Cobthree}{\mathsf{Cob}^3}
\newcommand{\Cobquot}{\mathsf{Cob}_{\ell}}
\newcommand{\deloop}{\mathbf{Del}}
\newcommand{\bnbrack}[1]{{[\![ #1 ]\!]}}
\newcommand{\capdot}{\mathsf{D}}
\newcommand{\cupdot}{\rotatebox[origin=c]{180}{$\mathsf{D}$}}
\newcommand{\figeq}[1]{\vcenter{\hbox{\includegraphics[scale=0.25]{figures/#1}}}}

\DeclareMathOperator{\homom}{hom}
\DeclareMathOperator{\End}{End}
\DeclareMathOperator{\Res}{Res}
\DeclareMathOperator{\Quo}{Quo}
\DeclareMathOperator{\ev}{ev}

\NewDocumentCommand{\sdr}{ m m m m m O{large} }{%
	\begin{tikzcd}[column sep=#6,ampersand replacement=\&]
		{#1} \arrow[r, "{#3}", shift left=0.5ex,] \arrow[r, leftarrow, "{#4}", shift right=0.5ex, swap] \&
		#2 \arrow[loop right, distance=1em, in=15, out=-15, "{#5}",swap]
	\end{tikzcd}%
}

\NewDocumentCommand{\vsdr}{ m m m m m O{large} }{%
	\begin{tikzcd}[column sep=#6, ampersand replacement=\&]
		{#1}
		\arrow[r, "{#3}", shift left=0.5ex, color=violet]
		\arrow[r, leftarrow, "{#4}", shift right=0.5ex, swap, color=violet]
		\& #2 \arrow[loop right, distance=1em, in=15, out=-15, "{#5}",swap, color=violet]
	\end{tikzcd}%
}

\NewDocumentCommand{\gsdr}{ m m m m m O{large} }{%
	\begin{tikzcd}[column sep=#6, ampersand replacement=\&]
		{#1}
		\arrow[r, "{#3}", shift left=0.5ex, color=ForestGreen]
		\arrow[r, leftarrow, "{#4}", shift right=0.5ex, swap, color=ForestGreen]
		\& #2 \arrow[loop right, distance=1em, in=15, out=-15, "{#5}",swap, color=ForestGreen]
	\end{tikzcd}%
}

\newcommand{\mfshort}[4]{%
	\begin{tikzcd}[ampersand replacement=\&]
		{#1} \arrow[r, "#3", bend left=5,start anchor=east, end anchor=west, shift left=2pt] \&
		{#2} \arrow[l,"#4", bend left=5, start anchor=west, end anchor=east, shift left=2pt]
	\end{tikzcd}%
}

\newcommand{\kmf}[2]{
	K_{#1}{\begin{pmatrix}
			#2
		\end{pmatrix}
	}
}

\newcommand{\skmf}[2]{
	K_{#1}{\left(\begin{smallmatrix}
			#2
		\end{smallmatrix}\right)
	}
}

\NewDocumentCommand{\ThickEdgeAndCircle}{m m m O{} O{} O{}}{
	\begin{tikzpicture}[thick,baseline=(current bounding box.center)]
		\draw[dotted,thick] (0,0) circle (0.8cm);

		\draw[line width=2.5pt] (-0.2,-0.2) -- (-0.2,0.2);

		\draw[line width=1pt,]
			(-0.5,-0.65) .. controls (-0.4,-0.55) and (-0.25,-0.35) .. (-0.225,-0.18);
		\draw[line width=1pt, -Stealth]
			(-0.225,0.18) .. controls (-0.25,0.35) and (-0.4,0.55) .. (-0.5,0.65);
		\draw[line width=1pt] ([shift=(-145:0.32cm)]0.07,0) arc (-145:145:0.32cm);

		\node at (0,-0.55) {$#1$};
		\node at (-0.5,0) {$#2$};
		\node at (0.1,0) {$#3$};

		\node[anchor=center] at (0.7,0.7) {#4};
		\node[anchor=center] at (0.7,-0.7) {#5};
		\node[anchor=center] at (-0.7,-0.7) {#6};
	\end{tikzpicture}
}

\NewDocumentCommand{\RER}{m m O{} O{} O{} O{0}}{
	\begin{tikzpicture}[baseline=(current bounding box.center)]
		\coordinate (O) at (0,0);
		\coordinate (A) at (0.2,0.5);
		\coordinate (C) at (0.2,-0.5);

		\draw[dotted,thick] (O) circle (1cm);

		\draw[thick]
			($(C)+(45:0.3)$)
			.. controls +({45}:0.2) and +(-45:0.2)
			.. ($(A)+(-45:0.3)$);

		\draw[thick]
			($(C)+(-45:0.3)$)
			.. controls +({-45}:0.1) and +(110:0.1)
			.. ($(O)+(-60:1)$);

		\draw[thick]
			($(C)+(135:0.3)$)
			.. controls +({135}:0.4) and +(45:0.4)
			.. ($(O)+(-150:1)$);

		\draw[thick]
			($(C)+(-135:0.3)$)
			.. controls +({-135}:0.1) and +(90:0.1)
			.. ($(O)+(260:1)$);

		\draw[thick]
			($(A)+(-135:0.3)$)
			.. controls +({-135}:0.4) and +(-45:0.4)
			.. ($(O)+(150:1)$);

		\draw[thick]
			($(A)+(45:0.3)$)
			.. controls +({45}:0.1) and +(-110:0.1)
			.. ($(O)+(60:1)$);

		\draw[thick]
			($(A)+(135:0.3)$)
			.. controls +({135}:0.1) and +(-90:0.1)
			.. ($(O)+(100:1)$);

		\begin{scope}[shift={(A)}]
			\csname thi#1e\endcsname{0.3}
		\end{scope}
		\begin{scope}[shift={(C)}]
			\csname thi#2e\endcsname{0.3}
		\end{scope}

		\node[anchor=center] at (0.85,0.85) {#3};
		\node[anchor=center] at (0.85,-0.85) {#4};
		\node[anchor=center] at (-0.85,-0.85) {#5};

		\ifnum #6=1
		\node[anchor=center] at (0.2,0) {$x_1$};
		\node[anchor=center] at (-0.35,-0.5) {$x_2$};
		\node[anchor=center] at (-0.45,0.55) {$x_0$};
		\node[anchor=center] at (0.2,-0.8) {$x_3$};
		\node[anchor=center] at (0.75,0) {$x_4$};
		\node[anchor=center] at (0.2,0.81) {$x_5$};
		\fi

		\ifnum #6=2
		\node[anchor=center] at (0.2,0) {$u$};
		\node[anchor=center] at (-0.35,-0.5) {$x_2$};
		\node[anchor=center] at (-0.45,0.55) {$x_0$};
		\node[anchor=center] at (0.2,-0.8) {$x_3$};
		\node[anchor=center] at (0.75,0) {$x_4$};
		\node[anchor=center] at (0.2,0.81) {$x_5$};
		\fi
	\end{tikzpicture}
}

\pgfkeys{
	/diagramarcs/.is family, /diagramarcs,
	q/.estore in = \qdeg,
	h/.estore in = \hdeg,
	hslash/.estore in = \hslashdeg,
	dir1/.estore in = \arconedir, lab1/.store in = \arconelabel, sha1/.estore in =\arconeshade,
	dir2/.estore in = \arctwodir, lab2/.store in = \arctwolabel, sha2/.estore in =\arctwoshade,
	dir3/.estore in = \arcthreedir, lab3/.store in = \arcthreelabel, sha3/.estore in =\arcthreeshade,
	extra/.store in = \extralines,
	cr1/.estore in = \iscrossingone, cr2/.estore in = \iscrossingtwo, cr3/.estore in = \iscrossingthree,
	emp1/.estore in = \isemptyone, emp2/.estore in = \isemptytwo, emp3/.estore in = \isemptythree,
	thin1/.estore in = \isemptyone, thin2/.estore in = \isemptytwo, thin3/.estore in = \isemptythree,
	thick1/.estore in = \isemptyone, thick2/.estore in = \isemptytwo, thick3/.estore in = \isemptythree,
	q={},
	h={},
	hslash={},
	hslash/.default={$\hslash$},
	emp1=0, emp1/.default=2, thin1/.default=-3, thick1/.default=-4,
	emp2=0, emp2/.default=2, thin2/.default=-3, thick2/.default=-4,
	emp3=0, emp3/.default=2, thin3/.default=-3, thick3/.default=-4,
	cr1=0, cr1/.default=1,
	cr2=0, cr2/.default=1,
	cr3=0, cr3/.default=1,
	dir1=0, lab1= {}, sha1=0,
	dir2=0, lab2= {}, sha2=0,
	dir3=0, lab3= {}, sha3=0,
	extra={},
}

\NewDocumentCommand{\thinput}{ m O{0} O{0} O{a} O{0}}{
	\begin{scope}[rotate=#3]

		\pgfmathsetmacro{\AAA}{int(abs((1-(#5)*(#5))*(100-(#5)*(#5))))}
		\pgfmathsetmacro{\BBB}{int(mod((#3)/90,2))}
		\pgfmathsetmacro{\CCC}{int(mod((#3)/90,4))}

		\ifnum \AAA>0
		\ifnum #2=0
		\ifnum #5>-3
		\fill[gray!15]
			($(135:#1)$)
			.. controls +(-45:#1*0.7) and +(45:#1*0.7) .. ($(-135:#1)$)
			arc (-135:-225:#1);

		\fill[gray!15]
			($(-45:#1)$)
			.. controls +(135:#1*0.7) and +(-135:#1*0.7) .. ($(45:#1)$)
			arc (45:-45:#1);
		\fi
		\fi

		\ifnum #2=1
		\ifnum #5>-3
		\fill[gray!15]
			($(135:#1)$)
			.. controls +(-45:#1*0.7) and +(45:#1*0.7) .. ($(-135:#1)$)
			arc (-135:-45:#1)
			.. controls +(135:#1*0.7) and +(-135:#1*0.7) .. ($(45:#1)$)
			arc (45:135:#1);
		\fi
		\fi

		\ifnum #5=0
		\draw[dotted, dash pattern=on 1pt off 1pt, -{>[scale=0.8]}]
			($(-180:#1/3)$)
			.. controls +({0}:0.2) and +(180:0.2)
			.. ($(0:#1/3)$);
		\fi

		\ifnum #5=-4
		\draw[ultra thick]
			($(-180:#1/3)$)
			.. controls +({0}:0.18) and +(180:0.2)
			.. ($(0:#1/3.2)$);
		\fi

		\ifnum #5=-3
		\draw[dotted, dash pattern=on 1pt off 1pt]
			($(-180:#1/3)$)
			.. controls +({0}:0.2) and +(180:0.2)
			.. ($(0:#1/3.2)$);
		\fi

		\draw[thick, {\ifnum #5<-2 \ifnum #5>-4 Stealth-{} \fi\fi}]
			($(135:#1)$)
			.. controls +(-45:#1*0.7) and +(45:#1*0.7)
			.. ($(-135:#1)$);

		\draw[thick, {\ifnum #5<-2 {\ifnum #5<-3 Stealth \fi}-Stealth \fi}]
			($(-45:#1)$)
			.. controls +(135:#1*0.7) and +(-135:#1*0.7)
			.. ($(45:#1)$);

		\ifnum \BBB=0
		\node[anchor=center] at (0,#1/1.5) {#4};
		\fi

		\ifnum \CCC=1
		\node[anchor=base] at (-#1/6,#1/1.5) {#4};
		\fi

		\ifnum \CCC=3
		\node[anchor=base] at (#1/6,#1/1.5) {#4};
		\fi

		\fi

		\ifnum \AAA=0
		\ifnum \numexpr #5*#5 = 1
		\begin{scope}[rotate=90*(1-#2)]
			\fill[gray!15] (0,0) -- (45:#1) arc (45:135:#1) -- cycle;
			\fill[gray!15] (0,0) -- (-135:#1) arc (-135:-45:#1) -- cycle;
		\end{scope}
		\fi

		\ifnum #5<0

		\draw[thick, stealth-]
			($(-45:#1)$) -- ($(135:#1)$);

		\draw[line width=0.91pt, color=white]
			($(-45:0.2*#1)$) -- ($(135:0.2*#1)$);

		\draw[thick, -stealth]
			($(-135:#1)$) -- ($(45:#1)$);

		\fi

		\ifnum #5>0
		\draw[thick, -stealth]
			($(-135:#1)$) -- ($(45:#1)$);

		\draw[thick, line width=0.91pt, color=white]
			($(-135:0.2*#1)$) -- ($(45:0.2*#1)$);

		\draw[thick, stealth-]
			($(-45:#1)$) -- ($(135:#1)$);

		\fi
		\fi

	\end{scope}
}

\newcommand{\Elementary}[1][]{
	\begin{tikzpicture}[thick,baseline=(current bounding box.center)]
		\pgfkeys{/diagramarcs, #1}

		\pgfmathsetmacro{\Y}{int(mod(\arconedir+\arconeshade,2))}

		\thinput{0.8}[\Y][90*\arconedir][\arconelabel][\the\numexpr \iscrossingone + \isemptyone \relax]

		\node[anchor=center] at (0.7,0.7) {\qdeg};
		\node[anchor=center] at (0.7,-0.7) {\hdeg};
		\node[anchor=center] at (-0.7,-0.7) {\hslashdeg};

		\draw[dotted,thick] (0,0) circle (0.8cm);

		\extralines
	\end{tikzpicture}
}

\newcommand{\SingleEdge}[1][]{
	\begin{tikzpicture}[thick,baseline=(current bounding box.center)]
		\pgfkeys{/diagramarcs, #1}

		\begin{pgfinterruptboundingbox}
			\begin{scope}
				\clip (0,0) circle (0.8cm);

				\fill[gray!15]
					(110:0.8)
					.. controls +(-45:0.6) and +(45:0.6) .. (-110:0.8)
					-- (-1.5, -1.2) -- (-1.5, 1.2) -- cycle;
			\end{scope}
		\end{pgfinterruptboundingbox}

		\draw[dotted,thick] (0,0) circle (0.8cm);

		\draw[thick]
			(110:0.8)
			.. controls +(-45:0.6) and +(45:0.6)
			.. (-110:0.8);

		\node[anchor=center] at (0.7,0.7) {\qdeg};
		\node[anchor=center] at (0.7,-0.7) {\hdeg};
		\node[anchor=center] at (-0.7,-0.7) {\hslashdeg};

		\extralines
	\end{tikzpicture}
}

\newcommand{\ThinEdgeAndCircle}[1][]{
	\begin{tikzpicture}[baseline=(current bounding box.center)]
		\pgfkeys{/diagramarcs, #1}
		\begin{scope}
			\clip (0,0) circle (0.8cm);

			\fill[gray!15]
				(-0.5,0.65) .. controls (-0.4,0.5) and (135:0.6) .. (135:0.35)
				arc (135:225:0.35)
				.. controls (-135:0.6) and (-0.4,-0.5) .. (-0.5,-0.65)
				-- (-1, -1) -- (-1, 1) -- cycle;

			\fill[gray!15]
				($(45:0.35)$) .. controls +(45:0.6) and +(-45:0.6) .. ($(-45:0.35)$)
				arc (-45:45:0.35);
		\end{scope}

		\draw[dotted,thick] (0,0) circle (0.8cm);

		\draw[dotted, dash pattern=on 1pt off 1pt, line width=0.5pt] (-0.18,0) -- (0.0,0);

		\draw[thick]
			(-0.5,0.65) .. controls (-0.4,0.5) and (135:0.6) .. (135:0.35);
		\draw[thick]
			(-0.5,-0.65) .. controls (-0.4,-0.5) and (-135:0.6) .. (-135:0.35);
		\draw[thick]
			($(45:0.35)$)
			.. controls +(45:0.6) and +(-45:0.6)
			.. ($(-45:0.35)$);

		\pgfmathsetmacro{\Y}{int(mod(\arconedir+\arconeshade,2))}

		\thinput{0.35}[\Y][90*\arconedir][\arconelabel][\the\numexpr \iscrossingone + \isemptyone \relax]

		\node[anchor=center] at (0.7,0.7) {\qdeg};
		\node[anchor=center] at (0.7,-0.7) {\hdeg};
		\node[anchor=center] at (-0.7,-0.7) {\hslashdeg};

		\extralines
	\end{tikzpicture}
}

\newcommand{\Square}[1][]{
	\begin{tikzpicture}[baseline=(current bounding box.center)]
		\pgfkeys{/diagramarcs, #1}
		\fill[gray!15]
			($(-0.4,0)+(45:0.35)$)
			.. controls +(45:0.1) and +(135:0.1) .. ($(0.4,0)+(135:0.35)$)
			arc (135:225:0.35)
			.. controls +(-135:0.1) and +(-45:0.1) .. ($(-0.4,0)+(-45:0.35)$)
			arc (-45:45:0.35)
			-- cycle;

		\fill[gray!15]
			($(0,0)+(140:1)$)
			arc (140:220:1)
			.. controls +(50:0.1) and +(-135:0.1) .. ($(-0.4,0)+(-135:0.35)$)
			arc (225:135:0.35)
			.. controls +(135:0.1) and +(-50:0.1) .. ($(0,0)+(140:1)$)
			-- cycle;

		\fill[gray!15]
			($(0,0)+(-40:1)$)
			arc (-40:40:1)
			.. controls +(220:0.1) and +(45:0.1) .. ($(0.4,0)+(45:0.35)$)
			arc (45:-45:0.35)
			.. controls +(-45:0.1) and +(-220:0.1) .. ($(0,0)+(-40:1)$)
			-- cycle;

		\draw[dotted,thick] (0,0) circle (1cm);

		\draw[thick]
			($(-0.4,0)+(-45:0.35)$)
			.. controls +(-45:0.1) and +(-135:0.1)
			.. ($(0.4,0)+(-135:0.35)$);
		\draw[thick]
			($(-0.4,0)+(45:0.35)$)
			.. controls +(45:0.1) and +(135:0.1)
			.. ($(0.4,0)+(135:0.35)$);
		\draw[thick]
			($(-0.4,0)+(135:0.35)$)
			.. controls +(135:0.1) and +(-50:0.1)
			.. ($(0,0)+(140:1)$);
		\draw[thick]
			($(-0.4,0)+(-135:0.35)$)
			.. controls +(-135:0.1) and +(50:0.1)
			.. ($(0,0)+(220:1)$);

		\draw[thick]
			($(0.4,0)+(45:0.35)$)
			.. controls +(45:0.1) and +(220:0.1)
			.. ($(0,0)+(40:1)$);
		\draw[thick]
			($(0.4,0)+(-45:0.35)$)
			.. controls +(-45:0.1) and +(-220:0.1)
			.. ($(0,0)+(-40:1)$);

		\pgfmathsetmacro{\Y}{int(mod(\arconedir+\arconeshade,2))}
		\pgfmathsetmacro{\Z}{int(mod(\arctwodir+\arctwoshade,2))}

		\begin{scope}[shift={(-0.4,0)}]
			\thinput{0.35}[\Y][90*\arconedir][\arconelabel][\the\numexpr \iscrossingone + \isemptyone \relax]
		\end{scope}
		\begin{scope}[shift={(0.4,0)}]
			\thinput{0.35}[\Z][90*\arctwodir][\arctwolabel][\the\numexpr \iscrossingtwo + \isemptytwo \relax]
		\end{scope}

		\node[anchor=center] at (0.85,0.85) {\qdeg};
		\node[anchor=center] at (0.85,-0.85) {\hdeg};
		\node[anchor=center] at (-0.85,-0.85) {\hslashdeg};

		\extralines
	\end{tikzpicture}
}

\newcommand{\Bigon}[1][]{
	\begin{tikzpicture}[baseline=(current bounding box.center)]
		\pgfkeys{/diagramarcs, #1}
		\begin{pgfinterruptboundingbox}
			\begin{scope}
				\clip (0,0) circle (1cm);

				\fill[gray!15]
					($(0,0.4)+(135:0.35)$)
					.. controls +(135:0.1) and +(-50:0.1) .. ($(0,0)+(130:1)$)
					-- ($(0,0)+(130:1.5)$) arc (130:230:1.5) -- ($(0,0)+(-130:1)$)
					.. controls +(40:0.1) and +(-135:0.1) .. ($(0,-0.4)+(-135:0.35)$)
					arc (225:135:0.35)
					.. controls +(135:0.1) and +(-135:0.1) .. ($(0,0.4)+(-135:0.35)$)
					arc (225:135:0.35)
					-- cycle;

				\fill[gray!15]
					($(0,0.4)+(-45:0.35)$)
					.. controls +(-45:0.1) and +(45:0.1) .. ($(0,-0.4)+(45:0.35)$)
					arc (45:-45:0.35)
					.. controls +(-45:0.1) and +(140:0.1) .. ($(0,0)+(-50:1)$)
					-- ($(0,0)+(-50:1.5)$) arc (-50:50:1.5) -- ($(0,0)+(50:1)$)
					.. controls +(-130:0.1) and +(45:0.1) .. ($(0,0.4)+(45:0.35)$)
					arc (45:-45:0.35)
					-- cycle;
			\end{scope}
		\end{pgfinterruptboundingbox}

		\draw[dotted,thick] (0,0) circle (1cm);

		\draw[thick]
			($(0,-0.4)+(45:0.35)$)
			.. controls +(45:0.1) and +(-45:0.1)
			.. ($(0,0.4)+(-45:0.35)$);

		\draw[thick]
			($(0,-0.4)+(135:0.35)$)
			.. controls +(135:0.1) and +(-135:0.1)
			.. ($(0,0.4)+(-135:0.35)$);

		\draw[thick]
			($(0,-0.4)+(-135:0.35)$)
			.. controls +(-135:0.1) and +(40:0.1)
			.. ($(0,0)+(-130:1)$);

		\draw[thick]
			($(0,-0.4)+(-45:0.35)$)
			.. controls +(-45:0.1) and +(140:0.1)
			.. ($(0,0)+(-50:1)$);

		\draw[thick]
			($(0,0.4)+(135:0.35)$)
			.. controls +(135:0.1) and +(-50:0.1)
			.. ($(0,0)+(130:1)$);

		\draw[thick]
			($(0,0.4)+(45:0.35)$)
			.. controls +(45:0.1) and +(-130:0.1)
			.. ($(0,0)+(50:1)$);

		\pgfmathsetmacro{\Y}{int(mod(\arconedir+\arconeshade,2))}
		\pgfmathsetmacro{\Z}{int(mod(\arctwodir+\arctwoshade,2))}

		\begin{scope}[shift={(0,0.4)}]
			\thinput{0.35}[\Y][90*\arconedir][\arconelabel][\the\numexpr \iscrossingone + \isemptyone \relax]
		\end{scope}
		\begin{scope}[shift={(0,-0.4)}]
			\thinput{0.35}[\Z][90*\arctwodir][\arctwolabel][\the\numexpr \iscrossingtwo + \isemptytwo \relax]
		\end{scope}

		\node[anchor=center] at (0.85,0.85) {\qdeg};
		\node[anchor=center] at (0.85,-0.85) {\hdeg};
		\node[anchor=center] at (-0.85,-0.85) {\hslashdeg};

		\extralines
	\end{tikzpicture}
}

\newcommand{\RL}[1][]{
	\begin{tikzpicture}[baseline={(current bounding box.center)}]
		\pgfkeys{/diagramarcs,#1}

		\coordinate (O) at (0,0);
		\coordinate (A) at (0.3,0.3);
		\coordinate (B) at (-0.3,-0.3);

		\begin{scope}
			\clip (O) circle (1cm);

			\fill[gray!15]
				($(A)+(135:0.35)$)
				.. controls +({135}:0.1) and +(-90:0.1) .. ($(O)+(100:1)$)
				-- ($(O)+(100:1.5)$) -- ($(O)+(45:1.5)$) -- ($(O)+(45:1)$)
				.. controls +(-135:0.1) and +({45}:0.1) .. ($(A)+(45:0.35)$)
				arc (45:135:0.35) -- cycle;

			\fill[gray!15]
				($(B)+(-45:0.35)$)
				.. controls +(-45:0.1) and +(90:0.1) .. ($(O)+(-80:1)$)
				-- ($(O)+(-80:1.5)$) -- ($(O)+(-30:1.5)$) -- ($(O)+(-30:1)$)
				.. controls +({150}:0.2) and +(-45:0.2) .. ($(A)+(-45:0.35)$)
				arc (-45:-135:0.35)
				.. controls +(-135:0.1) and +({45}:0.1) .. ($(B)+(45:0.35)$)
				arc (45:-45:0.35) -- cycle;

			\fill[gray!15]
				($(B)+(135:0.35)$)
				.. controls +({135}:0.2) and +(-30:0.2) .. ($(O)+(150:1)$)
				-- ($(O)+(150:1.5)$) -- ($(O)+(-135:1.5)$) -- ($(O)+(-135:1)$)
				.. controls +(45:0.1) and +({-135}:0.1) .. ($(B)+(-135:0.35)$)
				arc (-135:-225:0.35) -- cycle;
		\end{scope}

		\draw[dotted,thick] (O) circle (1cm);

		\draw[thick]
			($(O)+(-30:1)$)
			.. controls +({150}:0.2) and +(-45:0.2)
			.. ($(A)+(-45:0.35)$);

		\draw[thick]
			($(O)+(-80:1)$)
			.. controls +(90:0.1) and +(-45:0.1)
			.. ($(B)+(-45:0.35)$);

		\draw[thick]
			($(B)+(135:0.35)$)
			.. controls +({135}:0.2) and +(-30:0.2)
			.. ($(O)+(150:1)$);

		\draw[thick]
			($(B)+(-135:0.35)$)
			.. controls +({-135}:0.1) and +(45:0.1)
			.. ($(O)+(-135:1)$);

		\draw[thick]
			($(B)+(45:0.35)$)
			.. controls +({45}:0.1) and +(-135:0.1)
			.. ($(A)+(-135:0.35)$);

		\draw[thick]
			($(A)+(45:0.35)$)
			.. controls +({45}:0.1) and +(-135:0.1)
			.. ($(O)+(45:1)$);

		\draw[thick]
			($(A)+(135:0.35)$)
			.. controls +({135}:0.1) and +(-90:0.1)
			.. ($(O)+(100:1)$);

		\pgfmathsetmacro{\Y}{int(1-mod(\arconedir+\arconeshade,2))}
		\pgfmathsetmacro{\Z}{int(mod(\arctwodir+\arctwoshade,2))}

		\begin{scope}[shift={(A)}]
			\thinput{0.35}[\Y][90*\arconedir][\arconelabel][\the\numexpr \iscrossingone + \isemptyone \relax]
		\end{scope}
		\begin{scope}[shift={(B)}]
			\thinput{0.35}[\Z][90*\arctwodir][\arctwolabel][\the\numexpr \iscrossingtwo + \isemptytwo \relax]
		\end{scope}

		\node[anchor=center] at (0.85,0.85) {\qdeg};
		\node[anchor=center] at (0.85,-0.85) {\hdeg};
		\node[anchor=center] at (-0.85,-0.85) {\hslashdeg};

		\extralines
	\end{tikzpicture}
}

\newcommand{\ThreeHills}[1][]{
	\begin{tikzpicture}[baseline={(current bounding box.center)}]
		\def\myextracode{}

		\pgfkeys{/diagramarcs, #1}

		\coordinate (O) at (0,0);
		\coordinate (A) at (30:0.43);
		\coordinate (B) at (150:0.43);
		\coordinate (C) at (270:0.43);

		\begin{pgfinterruptboundingbox}
			\begin{scope}
				\clip (O) circle (1cm);

				\fill[gray!15]
					($(O)+(60:1)$)
					.. controls +(240:0.2) and +(75:0.2) .. ($(A)+(75:0.35)$)
					arc (75:165:0.35)
					.. controls +(165:0.2) and +(15:0.2) .. ($(B)+(15:0.35)$)
					arc (15:105:0.35)
					.. controls +(105:0.2) and +(300:0.2) .. ($(O)+(120:1)$)
					-- ($(O)+(120:1.5)$) arc (120:60:1.5) -- cycle;

				\fill[gray!15]
					($(O)+(180:1)$)
					.. controls +(0:0.2) and +(195:0.2) .. ($(B)+(195:0.35)$)
					arc (195:285:0.35)
					.. controls +(285:0.2) and +(135:0.2) .. ($(C)+(135:0.35)$)
					arc (135:225:0.35)
					.. controls +(225:0.2) and +(60:0.2) .. ($(O)+(240:1)$)
					-- ($(O)+(240:1.5)$) arc (240:180:1.5) -- cycle;

				\fill[gray!15]
					($(O)+(300:1)$)
					.. controls +(120:0.2) and +(315:0.2) .. ($(C)+(315:0.35)$)
					arc (-45:45:0.35)
					.. controls +(45:0.2) and +(-105:0.2) .. ($(A)+(-105:0.35)$)
					arc (-105:-15:0.35)
					.. controls +(-15:0.2) and +(180:0.2) .. ($(O)+(0:1)$)
					-- ($(O)+(0:1.5)$) arc (360:300:1.5) -- cycle;
			\end{scope}
		\end{pgfinterruptboundingbox}

		\draw[dotted,thick] (O) circle (1cm);

		\draw[thick]
			($(A)+(-15:0.35)$) .. controls +(-15:0.2) and +(180:0.2) .. ($(O)+(0:1)$);
		\draw[thick]
			($(A)+(75:0.35)$) .. controls +(75:0.2) and +(240:0.2) .. ($(O)+(60:1)$);

		\draw[thick]
			($(B)+(105:0.35)$) .. controls +(105:0.2) and +(300:0.2) .. ($(O)+(120:1)$);
		\draw[thick]
			($(B)+(195:0.35)$) .. controls +(195:0.2) and +(0:0.2) .. ($(O)+(180:1)$);

		\draw[thick]
			($(C)+(225:0.35)$) .. controls +(225:0.2) and +(60:0.2) .. ($(O)+(240:1)$);
		\draw[thick]
			($(C)+(315:0.35)$) .. controls +(315:0.2) and +(120:0.2) .. ($(O)+(300:1)$);

		\draw[thick]
			($(A)+(165:0.35)$) .. controls +(165:0.02) and +(15:0.02) .. ($(B)+(15:0.35)$);
		\draw[thick]
			($(B)+(285:0.35)$) .. controls +(285:0.02) and +(135:0.02) .. ($(C)+(135:0.35)$);
		\draw[thick]
			($(C)+(45:0.35)$) .. controls +(45:0.02) and +(-105:0.02) .. ($(A)+(-105:0.35)$);

		\pgfmathsetmacro{\XX}{int(1-mod(\arconedir+\arconeshade,2))}
		\pgfmathsetmacro{\YY}{int(mod(\arctwodir+\arctwoshade,2))}
		\pgfmathsetmacro{\ZZ}{int(mod(\arcthreedir+\arcthreeshade,2))}

		\begin{scope}[shift={(A)}, rotate=30]
			\thinput{0.35}[\XX][90*\arconedir][\arconelabel][\the\numexpr \iscrossingone + \isemptyone \relax]
		\end{scope}

		\begin{scope}[shift={(B)}, rotate=60]
			\thinput{0.35}[\YY][90*\arctwodir][\arctwolabel][\the\numexpr \iscrossingtwo + \isemptytwo \relax]
		\end{scope}

		\begin{scope}[shift={(C)}]
			\thinput{0.35}[\ZZ][90*\arcthreedir][\arcthreelabel][\the\numexpr \iscrossingthree + \isemptythree \relax]
		\end{scope}

		\node[anchor=center] at (0.85,0.85) {\qdeg};
		\node[anchor=center] at (0.85,-0.85) {\hdeg};
		\node[anchor=center] at (-0.85,-0.85) {\hslashdeg};

		\extralines
	\end{tikzpicture}
}

\newcommand{\LR}[1][]{
	\begin{tikzpicture}[baseline=(current bounding box.center)]
		\pgfkeys{/diagramarcs, #1}

		\coordinate (O) at (0,0);
		\coordinate (A) at (-0.3,0.3);
		\coordinate (B) at (0.3,-0.3);

		\begin{scope}
			\clip (O) circle (1cm);

			\fill[gray!15]
				($(O)+(-150:1.5)$) -- ($(O)+(-150:1)$)
				.. controls +({30}:0.2) and +(-135:0.2) .. ($(A)+(-135:0.35)$)
				arc (-135:-225:0.35)
				.. controls +({135}:0.1) and +(-45:0.1) .. ($(O)+(135:1)$)
				-- ($(O)+(135:1.5)$) -- cycle;

			\fill[gray!15]
				($(O)+(-100:1.5)$) -- ($(O)+(-100:1)$)
				.. controls +(90:0.1) and +(-135:0.1) .. ($(B)+(-135:0.35)$)
				arc (-135:-45:0.35)
				.. controls +({-45}:0.1) and +(135:0.1) .. ($(O)+(-45:1)$)
				-- ($(O)+(-45:1.5)$) -- cycle;

			\fill[gray!15]
				($(O)+(80:1.5)$) -- ($(O)+(80:1)$)
				.. controls +(-90:0.1) and +({45}:0.1) .. ($(A)+(45:0.35)$)
				arc (45:-45:0.35)
				.. controls +(-45:0.1) and +({135}:0.1) .. ($(B)+(135:0.35)$)
				arc (135:45:0.35)
				.. controls +({45}:0.2) and +(-150:0.2) .. ($(O)+(30:1)$)
				-- ($(O)+(30:1.5)$) -- cycle;
		\end{scope}

		\draw[dotted,thick] (O) circle (1cm);

		\draw[thick]
			($(O)+(-150:1)$)
			.. controls +({30}:0.2) and +(-135:0.2)
			.. ($(A)+(-135:0.35)$);

		\draw[thick]
			($(O)+(-100:1)$)
			.. controls +(90:0.1) and +(-135:0.1)
			.. ($(B)+(-135:0.35)$);

		\draw[thick]
			($(B)+(45:0.35)$)
			.. controls +({45}:0.2) and +(-150:0.2)
			.. ($(O)+(30:1)$);

		\draw[thick]
			($(B)+(-45:0.35)$)
			.. controls +({-45}:0.1) and +(135:0.1)
			.. ($(O)+(-45:1)$);

		\draw[thick]
			($(B)+(135:0.35)$)
			.. controls +({135}:0.1) and +(-45:0.1)
			.. ($(A)+(-45:0.35)$);

		\draw[thick]
			($(A)+(135:0.35)$)
			.. controls +({135}:0.1) and +(-45:0.1)
			.. ($(O)+(135:1)$);

		\draw[thick]
			($(A)+(45:0.35)$)
			.. controls +({45}:0.1) and +(-90:0.1)
			.. ($(O)+(80:1)$);

		\pgfmathsetmacro{\Y}{int(mod(\arconedir+\arconeshade,2))}
		\pgfmathsetmacro{\Z}{int(1-mod(\arctwodir+\arctwoshade,2))}

		\begin{scope}[shift={(A)}]
			\thinput{0.35}[\Y][90*\arconedir][\arconelabel][\the\numexpr \iscrossingone + \isemptyone \relax]
		\end{scope}
		\begin{scope}[shift={(B)}]
			\thinput{0.35}[\Z][90*\arctwodir][\arctwolabel][\the\numexpr \iscrossingtwo + \isemptytwo \relax]
		\end{scope}

		\node[anchor=center] at (0.85,0.85) {\qdeg};
		\node[anchor=center] at (0.85,-0.85) {\hdeg};
		\node[anchor=center] at (-0.85,-0.85) {\hslashdeg};

		\extralines
	\end{tikzpicture}
}

\newcommand{\LRL}[1][]{
	\begin{tikzpicture}[baseline=(current bounding box.center)]
		\pgfkeys{/diagramarcs,#1}

		\coordinate (O) at (0,0);
		\coordinate (A) at (-0.2,0.5);
		\coordinate (B) at (0.5,0);
		\coordinate (C) at (-0.2,-0.5);

		\fill[gray!15]
			($(O)+(120:1)$)
			.. controls +(-70:0.1) and +(135:0.1) .. ($(A)+(135:0.3)$)
			arc (135:225:0.3)
			.. controls +(135:0.2) and +(-135:0.2) .. ($(C)+(135:0.3)$)
			arc (135:225:0.3)
			.. controls +(-135:0.1) and +(70:0.1) .. ($(O)+(-120:1)$)
			arc (-120:-240:1)
			-- cycle;

		\fill[gray!15]
			($(O)+(280:1)$)
			.. controls +(90:0.1) and +(-45:0.1) .. ($(C)+(-45:0.3)$)
			arc (-45:45:0.3)
			.. controls +(45:0.1) and +(-135:0.1) .. ($(B)+(-135:0.3)$)
			arc (-135:-45:0.3)
			.. controls +(-45:0.1) and +(135:0.1) .. ($(O)+(-30:1)$)
			arc (-30:-80:1)
			-- cycle;

		\fill[gray!15]
			($(O)+(30:1)$)
			.. controls +(-135:0.1) and +(45:0.1) .. ($(B)+(45:0.3)$)
			arc (45:135:0.3)
			.. controls +(135:0.1) and +(-45:0.1) .. ($(A)+(-45:0.3)$)
			arc (-45:45:0.3)
			.. controls +(45:0.1) and +(-90:0.1) .. ($(O)+(80:1)$)
			arc (80:30:1)
			-- cycle;

		\draw[dotted,thick] (O) circle (1cm);

		\draw[thick]
			($(C)+(135:0.3)$)
			.. controls +(135:0.2) and +(-135:0.2)
			.. ($(A)+(-135:0.3)$);

		\draw[thick]
			($(C)+(-45:0.3)$)
			.. controls +(-45:0.1) and +(90:0.1)
			.. ($(O)+(280:1)$);

		\draw[thick]
			($(C)+(45:0.3)$)
			.. controls +(45:0.1) and +(-135:0.1)
			.. ($(B)+(-135:0.3)$);

		\draw[thick]
			($(C)+(-135:0.3)$)
			.. controls +(-135:0.1) and +(70:0.1)
			.. ($(O)+(-120:1)$);

		\draw[thick]
			($(B)+(45:0.3)$)
			.. controls +(45:0.1) and +(-135:0.1)
			.. ($(O)+(30:1)$);

		\draw[thick]
			($(B)+(-45:0.3)$)
			.. controls +(-45:0.1) and +(135:0.1)
			.. ($(O)+(-30:1)$);

		\draw[thick]
			($(B)+(135:0.3)$)
			.. controls +(135:0.1) and +(-45:0.1)
			.. ($(A)+(-45:0.3)$);

		\draw[thick]
			($(A)+(135:0.3)$)
			.. controls +(135:0.1) and +(-70:0.1)
			.. ($(O)+(120:1)$);

		\draw[thick]
			($(A)+(45:0.3)$)
			.. controls +(45:0.1) and +(-90:0.1)
			.. ($(O)+(80:1)$);

		\pgfmathsetmacro{\X}{int(mod(\arconedir+\arconeshade,2))}
		\pgfmathsetmacro{\Y}{int(1-mod(\arctwodir+\arctwoshade,2))}
		\pgfmathsetmacro{\Z}{int(mod(\arcthreedir+\arcthreeshade,2))}

		\begin{scope}[shift={(A)}]
			\thinput{0.3}[\X][90*\arconedir][\arconelabel][\the\numexpr \iscrossingone + \isemptyone \relax]
		\end{scope}
		\begin{scope}[shift={(B)}]
			\thinput{0.3}[\Y][90*\arctwodir][\arctwolabel][\the\numexpr \iscrossingtwo + \isemptytwo \relax]
		\end{scope}
		\begin{scope}[shift={(C)}]
			\thinput{0.3}[\Z][90*\arcthreedir][\arcthreelabel][\the\numexpr \iscrossingthree + \isemptythree \relax]
		\end{scope}

		\node[anchor=center] at (0.85,0.85) {\qdeg};
		\node[anchor=center] at (0.85,-0.85) {\hdeg};
		\node[anchor=center] at (-0.85,-0.85) {\hslashdeg};

		\extralines
	\end{tikzpicture}
}

\newcommand{\RLR}[1][]{
	\begin{tikzpicture}[baseline=(current bounding box.center)]
		\pgfkeys{/diagramarcs,#1}

		\coordinate (O) at (0,0);
		\coordinate (A) at (0.2,0.5);
		\coordinate (B) at (-0.45,0);
		\coordinate (C) at (0.2,-0.5);
		\fill[gray!15]
			($(A)+(-135:0.3)$)
			.. controls +(-135:0.1) and +({45}:0.1) .. ($(B)+(45:0.3)$)
			arc (45:-45:0.3)
			.. controls +(-45:0.1) and +({135}:0.1) .. ($(C)+(135:0.3)$)
			arc (135:45:0.3)
			.. controls +({45}:0.2) and +(-45:0.2) .. ($(A)+(-45:0.3)$)
			arc (-45:-135:0.3)
			-- cycle;

		\fill[gray!15]
			($(O)+(-150:1)$)
			arc (-150:-210:1)
			.. controls +(-45:0.1) and +({135}:0.1) .. ($(B)+(135:0.3)$)
			arc (135:225:0.3)
			.. controls +({-135}:0.1) and +(45:0.1) .. ($(O)+(-150:1)$)
			-- cycle;

		\fill[gray!15]
			($(O)+(-60:1)$)
			arc (-60:-100:1)
			.. controls +(90:0.1) and +({-135}:0.1) .. ($(C)+(-135:0.3)$)
			arc (-135:-45:0.3)
			.. controls +({-45}:0.1) and +(110:0.1) .. ($(O)+(-60:1)$)
			-- cycle;

		\fill[gray!15]
			($(O)+(100:1)$)
			arc (100:60:1)
			.. controls +(-110:0.1) and +({45}:0.1) .. ($(A)+(45:0.3)$)
			arc (45:135:0.3)
			.. controls +({135}:0.1) and +(-90:0.1) .. ($(O)+(100:1)$)
			-- cycle;

		\draw[dotted,thick] (O) circle (1cm);

		\draw[thick]
			($(C)+(45:0.3)$)
			.. controls +({45}:0.2) and +(-45:0.2)
			.. ($(A)+(-45:0.3)$);

		\draw[thick]
			($(C)+(-45:0.3)$)
			.. controls +({-45}:0.1) and +(110:0.1)
			.. ($(O)+(-60:1)$);

		\draw[thick]
			($(C)+(135:0.3)$)
			.. controls +({135}:0.1) and +(-45:0.1)
			.. ($(B)+(-45:0.3)$);

		\draw[thick]
			($(C)+(-135:0.3)$)
			.. controls +({-135}:0.1) and +(90:0.1)
			.. ($(O)+(260:1)$);

		\draw[thick]
			($(B)+(-135:0.3)$)
			.. controls +({-135}:0.1) and +(45:0.1)
			.. ($(O)+(-150:1)$);

		\draw[thick]
			($(B)+(135:0.3)$)
			.. controls +({135}:0.1) and +(-45:0.1)
			.. ($(O)+(150:1)$);

		\draw[thick]
			($(B)+(45:0.3)$)
			.. controls +({45}:0.1) and +(-135:0.1)
			.. ($(A)+(-135:0.3)$);

		\draw[thick]
			($(A)+(45:0.3)$)
			.. controls +({45}:0.1) and +(-110:0.1)
			.. ($(O)+(60:1)$);

		\draw[thick]
			($(A)+(135:0.3)$)
			.. controls +({135}:0.1) and +(-90:0.1)
			.. ($(O)+(100:1)$);

		\pgfmathsetmacro{\X}{int(1-mod(\arconedir+\arconeshade,2))}
		\pgfmathsetmacro{\Y}{int(mod(\arctwodir+\arctwoshade,2))}
		\pgfmathsetmacro{\Z}{int(1-mod(\arcthreedir+\arcthreeshade,2))}

		\begin{scope}[shift={(A)}]
			\thinput{0.3}[\X][90*\arconedir][\arconelabel][\the\numexpr \iscrossingone + \isemptyone \relax]
		\end{scope}
		\begin{scope}[shift={(B)}]
			\thinput{0.3}[\Y][90*\arctwodir][\arctwolabel][\the\numexpr \iscrossingtwo + \isemptytwo \relax]
		\end{scope}
		\begin{scope}[shift={(C)}]
			\thinput{0.3}[\Z][90*\arcthreedir][\arcthreelabel][\the\numexpr \iscrossingthree + \isemptythree \relax]
		\end{scope}

		\node[anchor=center] at (0.85,0.85) {\qdeg};
		\node[anchor=center] at (0.85,-0.85) {\hdeg};
		\node[anchor=center] at (-0.85,-0.85) {\hslashdeg};

		\extralines
	\end{tikzpicture}
}

\theoremstyle{plain}
\newtheorem{theorem}{Theorem}[section]
\newtheorem*{theorem*}{Theorem}
\newtheorem{lemma}[theorem]{Lemma}
\newtheorem{proposition}[theorem]{Proposition}
\newtheorem{corollary}[theorem]{Corollary}
\newtheorem*{corollary*}{Corollary}

\theoremstyle{definition}
\newtheorem{definition}[theorem]{Definition}
\newtheorem{example}[theorem]{Example}

\theoremstyle{remark}
\newtheorem{remark}[theorem]{Remark}
\newtheorem{question}[theorem]{Question}

\title{On the $A_\infty$-extension of Bar-Natan multicurves}
\author{Tomás Mejía-Gómez}
\date{}

\begin{document}

\maketitle

\begin{abstract}
	Kotelskiy, Watson and Zibrowius show that the Bar-Natan invariant $\De(T)$
	associated to a $4$-ended oriented tangle $T$ extends to a twisted complex
	$\De^\infty(T)$ over an $A_\infty$-enhancement of the Bar-Natan algebra. We
	prove that $\De^\infty(T)$ itself is a tangle invariant, settling a conjecture from their work. 
	We deduce this from an invariance statement for a
	\emph{matrix multifactorization} $\mathcal M(T)$, generalizing a construction
	of Ballinger to the coefficient ring $\Z[G]$. This multifactorization is well defined up to a
	notion of $1$-homotopy equivalence. The proof rests on explicit formulas for
	special deformation retracts of Koszul matrix factorizations, which realize
	delooping in this setting and may be of independent interest.
\end{abstract}

\section{Introduction}

Khovanov homology is an invariant of links in the 3-sphere which is very tractable by virtue of the combinatorial nature of its definition. In Bar-Natan's exposition in \cite{bn}, one starts with a link diagram $D$ drawn on a disk. Any possible total resolution of the crossings of $D$ yields an object in a cobordism category $\Cobthree$. The objects are embedded 1-submanifolds in the disk, and the morphisms are 2-dimensional embedded cobordisms in the solid cylinder. By taking an additive completion and quotienting out by a few local relations between cobordisms, one obtains a convenient category $\Cobquot$ where all resolutions of the diagram $D$ can be arranged into a chain complex $[\![D]\!]\in Ch(\Cobquot)$. This defines a link invariant up to chain homotopy.

Bar-Natan's local formulation allows for direct generalization to tangles. Moreover, different coefficient rings can be considered if the local relations are suitably adapted. In our case of interest, we work with a category denoted $\Cobquot(E^{2k}_*)$, with $k>0$, defined over the coefficient ring $\Z[G]$, where all objects have the same pointed set $E^{2k}_*$ of size $2k$ as their boundary. The basepoint is labeled by $*$. Of particular interest are crossingless tangles with no closed components, referred to as \emph{crossingless matchings}. In the case of 4 ends, we have two elementary crossingless matchings
\[
	\circ:=\scalebox{0.5}{\Elementary[sha1=-1, emp1, extra={\node at (-0.62,0.6) {*};}] }
	\qquad\qquad
	\bullet:=\scalebox{0.5}{\Elementary[dir1=1, emp1, sha1=-10, extra={\node at (-0.62,0.6) {*};} ]}_.
\]
Any object in $\Cobquot(E^{2k}_*)$ is isomorphic to a direct sum of shifted copies of crossingless matchings, a result known as \emph{delooping}. For this reason, one considers the linear subcategory $\mathcal B=\mathcal B(E^{4}_*)$ generated by the crossingless matchings $\bullet,\circ \in \Cobquot$. Thought of as an algebra with orthogonal idempotents, $\mathcal B$ is succinctly described as the quiver algebra over $\Z$
\[
	\mathcal B= \End_{\Cobquot(E_*^4)}(\circ\oplus\bullet)
	= \left[
	\begin{tikzcd}
		\circ
		\arrow[loop left, "D"]
		\arrow[r, bend left=20, "S"]
		&
		\bullet
		\arrow[loop right, "D"]
		\arrow[l, bend left=20, "S"]
	\end{tikzcd}
	\right]_
	{/(SD=0, DS=0)}.
\]
The chain complex $[\![T]\!]$ associated to a 4-ended oriented tangle $T$ can then be studied as a chain complex $\De(T)$ on the additive enlargement $\Sigma \mathcal B$. A benefit of this approach is that, over a field, the derived representation theory of $\mathcal B$ is well understood in terms of multicurves in punctured spheres. See \cite{hkk} and \cite{kwz}.

The algebra $\mathcal B$ admits another interpretation in the language of matrix factorizations with potential $W=xyz\in R= \Z[x,y,z]$. Roughly speaking, such a matrix factorization is a free $\Z/2$-graded $R$-module $M$ together with a degree-1 map $d_M$ such that $d_M^2=W\cdot\id_M$. Matrix factorizations over $R$ with potential $W$ are the objects of a $\Z/2$-graded dg-category $MF(R,W)$. Let us consider the following elementary objects in $MF(R,W)$:
\[
	M_\circ = \left[\mfshort{\hslash R}{R}{x}{yz} \right],
	\qquad
	M_\bullet = \left[\mfshort{\hslash R}{R}{y}{xz} \right],
\]
where the piece of degree 1 is indicated by $\hslash$. With suitable grading conventions, and working over a field $\mathbf k$, there is an identification
\begin{equation}\label{eq:homologyofmf}
	\mathcal B\otimes \mathbf k = H( \End_{MF(R\otimes \mathbf k,W)}(M_\circ\oplus M_\bullet))
\end{equation}
as algebras. As the homology of a dg-algebra, $\mathcal B\otimes \mathbf k$ inherits higher operations that enhance it into an $A_\infty$-algebra $\mathcal B^\infty_{\mathbf k}$, well defined up to quasi-isomorphism. By the homological mirror symmetry results of \cite{aaeko}, the $A_\infty$-algebra $\mathcal B^\infty_{\mathbf k}$ can also be understood in terms of the wrapped Fukaya category of the 3-punctured sphere. From this perspective, there is a concrete $A_\infty$-algebra $\mathcal B^\infty_{ds}$ defined over the integers by explicit combinatorial rules for its higher operations. It extends $\mathcal B$, and $\mathcal B^\infty_{ds}\otimes\mathbf k$ is quasi-isomorphic to $\mathcal B^\infty_{\mathbf k}$. This construction first appeared in \cite{hkk}. Once an extension of $\mathcal B$ into an $A_\infty$-algebra is fixed, an extension of $\Cobquot(E^{4}_*)$ into an $A_\infty$-category is determined by a delooping argument. We prove this more generally, for an arbitrary number of ends, in Theorem \ref{thm:barnatanainfinity}.

Matrix factorizations were first considered in the context of link homologies in \cite{krfacto}, where Khovanov and Rozansky produced the first categorifications of $\sln$-link polynomials for arbitrary $n>1$. Even though categories of matrix factorizations present a more computationally challenging point of view than Bar-Natan's, one could expect to gain new insight from the higher structures present there. An instance of this is found in work of Ballinger \cite{bal}, where matrix factorizations are used to define new (non-orientable) concordance invariants of knots. This construction can be seen as a lift of the Bar-Natan theory in the setting $G=0$.

In the papers \cite{kwz} and \cite{kwzmutation}, Kotelskiy, Watson and Zibrowius give a systematic treatment of multicurve invariants coming from the tangle invariant $\De(T)$.
These works prove several interesting structural results about multicurves associated to 4-ended tangles and about Khovanov homology, culminating in a proof of invariance of Khovanov homology over any field under component-preserving mutation. One crucial point is the fact that, working over a field, $\De(T)$ extends into a (bifiltered) twisted complex $\De^\infty(T)$ over the $A_\infty$-algebra $\mathcal B_{ds}^\infty$. The invariance of $\De^\infty(T)$ is left as a conjecture in \cite{kwzmutation}, which we settle in the affirmative.\\

The present work deals with a generalization of Ballinger's construction to the coefficient ring $\Z[G]$, agreeing with the one from which $\De^\infty(T)$ is obtained in \cite{kwzmutation}. To a tangle diagram $D$ with some decorations we associate a \emph{matrix multifactorization} $\mathcal M(D)$: a twisted complex in a suitable category of matrix factorizations, carrying a \emph{homological} $\Z$-grading $h$. The twisted differential is only required to be filtered of $h$-degree $\geq 1$, while homotopies are allowed to have filtered $h$-degree $\geq-1$. This yields the notion of \emph{1-homotopy equivalence}, which is used for our main invariance result.

\begin{theorem*}
	Let $D$ and $D'$ be decorated tangle diagrams with the same boundary data presenting the same oriented tangle $T$. Then $\mathcal M(D)$ and $\mathcal M(D')$ are 1-homotopy equivalent.
\end{theorem*}

\begin{corollary*}
	Let $\mathbf k$ be a field. The bifiltered twisted complex $\De^\infty(T)$ over $\mathcal B_{ds}^\infty\otimes\mathbf k$ is an invariant of $T$, up to 1-homotopy equivalence.
\end{corollary*}

The proof of invariance of $\mathcal M(D)$ rests on machinery we develop to get explicit formulas for special deformation retracts of Koszul matrix factorizations, of which the analysis of Reidemeister moves is only a special case. Indeed, we are able to obtain explicit data for delooping cubes of resolutions of arbitrary tangles in the context of matrix factorizations, as formulated more precisely in Proposition \ref{prop:deloopmulti}. We also expect this machinery to facilitate a generalization to $\sln$-link homologies for $n>2$.

It is natural to expect a generalization of \eqref{eq:homologyofmf} for the algebra $\mathcal B(E^{2k}_*)$ with arbitrary number of ends, although the author has not found it in the literature. This suggests, in turn, the existence of a preferred $A_\infty$-extension of $\Cobquot(E_*^{2k})$ up to quasi-isomorphism, for which an intrinsic description would be desirable.

\begin{question}
	Is there a tractable definition of an $A_\infty$-extension of $\Cobquot(E_*^{2k})$ in terms of cobordisms, such that it is quasi-isomorphic to one induced from an extension of $\mathcal B(E^{2k}_*)$ coming from matrix factorizations?
\end{question}
Once a desirable $A_\infty$-extension of $\Cobquot(E_*^{2k})$ is obtained, then we would also like to understand tangle invariants intrinsically.
\begin{question}
	Is there a formulation of the extended Bar-Natan invariant $\De^\infty(T)$ without appeal to matrix factorizations, by adapting the definition of the Bar-Natan bracket $\bnbrack{T}$ into one of a twisted complex over the distinguished extension of $\Cobquot(E_*^{2k})$?
\end{question}

In the case of 4-ended tangles, the explicit combinatorial description of the extension $\mathcal B^\infty_{ds}$ of $\mathcal B$ can be transferred to $\Cobquot(E_*^{4})$ to give a partial answer to the first question. The question of a description of $\De^\infty(T)$ intrinsic to the Bar-Natan category, however, remains elusive even for $4$-ended $T$.\\

\noindent\textbf{Structure of this paper.} Sections~\ref{sec:bn} and \ref{sec:ainfinity} deal with the generalities of the Bar-Natan category, the invariant $\De(T)$ and $A_\infty$-categories, including the transfer of $A_\infty$-structures from $\mathcal B(E^{2k}_*)$ to the Bar-Natan category $\Cobquot(E^{2k}_*)$. In Section~\ref{sec:mf}, we discuss the matrix factorization approach to crossingless tangles in the disk, going through the explicit formalism that produces convenient equivalences of matrix factorizations, including delooping. Next, we define the matrix multifactorization $\mathcal M(D)$ associated to a tangle diagram $D$ in Section~\ref{sec:multif}, and recall how it produces the $A_\infty$-extension $\De^\infty(T)$ of $\De(T)$. In the final Section~\ref{sec:inv}, we prove the invariance of $\mathcal M (-)$ and $\De^\infty(-)$ up to the adequate notion of homotopy equivalence.\\

\noindent\textbf{Acknowledgements.} The author was supported by Claudius Zibrowius' Starting Grant (project no. 101115822) from the European Research Council (ERC), and thanks him for suggesting this project and for many helpful conversations.

\section{The Bar-Natan construction}\label{sec:bn}

We begin with some definitions pertaining to Bar-Natan categories $\Cobquot(E)$. In general, objects are formal direct sums of 1-manifolds embedded in the disk $\mathbb D^2 $ with a fixed boundary $E\subset \partial\mathbb D^2$. The coefficient ring is $\Z[G]$, which is graded by quantum degree $q$, and such that $q(G)=-2$. We keep $E$ arbitrary in our definitions below but, for concreteness, we often restrict ourselves to the set $E^{2k}$ consisting of the points in the unit circle of the form $e^{\frac{(2j-1)\pi}{2k}i}$, for $j=1,\dots,2k$. It is also useful to consider the pointed set $E^{2k}_*$ with underlying set $E^{2k}$ and basepoint $*=e^{\frac{(2k-1)\pi}{2k}i}$.

\begin{definition}
	Let $E$ be a finite set of points on the boundary circle of a 2-dimensional disk $\mathbb D^2$. The linearized cobordism category $\Cobthree(E)$ is the $\Z$-graded $\Z[G]$-linear category defined as follows.
	\begin{itemize}
		\item Objects are embedded 1-manifolds (crossingless tangles) $\Gamma\subseteq \mathbb D^2$ such that $\partial \Gamma=\Gamma\cap\partial \mathbb D^2=E$.
		\item Morphisms $\Gamma_0\to \Gamma_1$ are formal $\Z[G]$-linear combinations of dotted cobordisms between $\Gamma_0$ and $\Gamma_1$, embedded in $\mathbb D^2\times [0,1]$, and considered up to isotopy rel. boundary. By dotted, we mean that each component may carry a finite number of dots that can freely move along the component.
		\item Identities are given by product cobordisms; composition is by gluing.
		\item The (quantum) grading of a cobordism $C$ is defined to be $q(C):= \chi(C)-\frac12 |E|-2\#\mbox{dots}$.

	\end{itemize}
\end{definition}

In the following definition, recall that the \emph{additive enlargement} or \emph{additive completion} $\Sigma\mathcal C$ of a graded linear category $\mathcal C$ has objects given by formal direct sums of shifts of objects in $\mathcal C$, and morphisms given by matrices of morphisms in $\mathcal C$.

\begin{definition}
	Let $E$ be a finite subset of the circle $\partial \mathbb D^2$. The \emph{Bar-Natan category} $\Cobquot(E)$ is the category $\Cobquot(E)=\Sigma\Cobthree(E)_{/\ell}$ obtained by quotienting with respect to the following local relations and taking additive enlargement.
	\begin{align*}
		\figeq{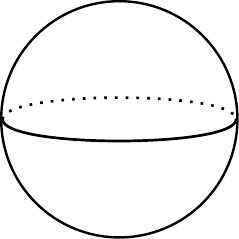}=0
		&& \figeq{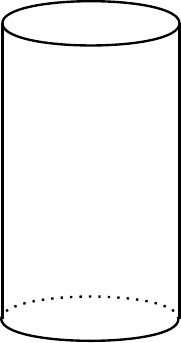} = \figeq{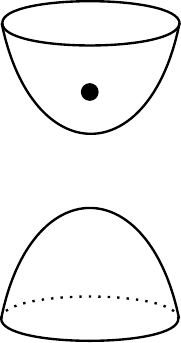}+\figeq{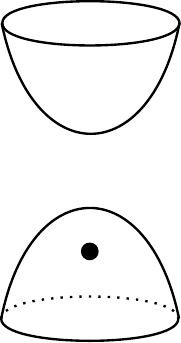}+G\figeq{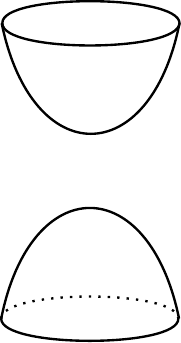}
		\\\\
		\figeq{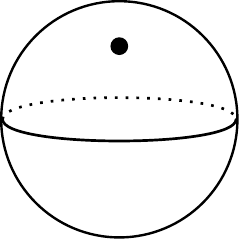}=1
		&& \figeq{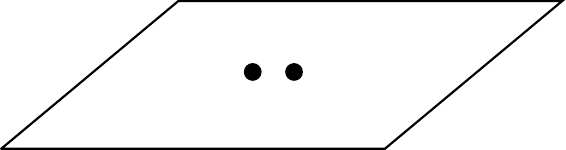} \,+\, G\figeq{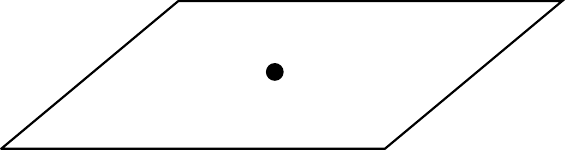}=0
	\end{align*}
	If $E_*$ is a pointed set with underlying set $E$ and basepoint $*\in E$, then we define $\Cobquot(E_*)$ as a further quotient with respect to the additional relation
	\[
		\figeq{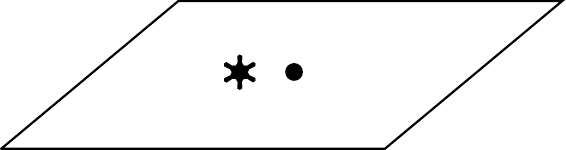}=0,
	\]
	which asserts the vanishing of cobordisms with dots on the based component.
\end{definition}

We introduce the red dot notation
\[
	\figeq{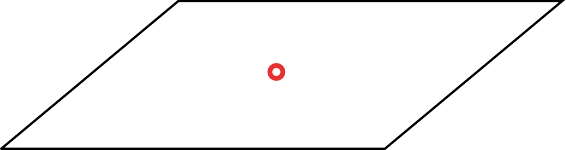}:= \figeq{dottedplane} + G\figeq{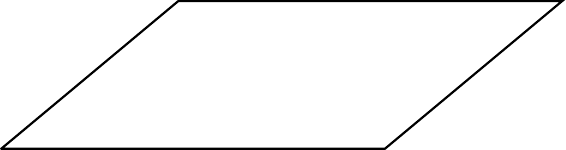}_.
\]
This allows us to express certain linear combinations of dotted cobordism in $\Cobquot(E)$ as a single \emph{decorated cobordism}, that is, a cobordism possibly with black and red dots. Observe that the presence of a black and a red dot in the same component automatically makes a cobordism vanish.

Given a crossingless tangle $\Gamma\subset \mathbb D^2$, write $\Gamma=c(\Gamma)\cup o(\Gamma)$, where $c(\Gamma)$ consists of the closed components, referred to as \emph{loops}, and $o(\Gamma)$ consists of components with boundary, referred to as \emph{open components}. The set of all loops is denoted by $\mathrm{circ}(\Gamma):= \pi_0(c(\Gamma))$.

As we will see, the local relations ensure that $\Cobquot(E)$ is additively generated by finitely many objects. These generating objects are \emph{crossingless matchings}, that is, crossingless tangles without loops. If $|E|=2k$, the number of crossingless matchings up to isotopy is the $k$-th Catalan number.

\begin{definition}
	Let $E$ be a possibly pointed finite subset of the circle $\partial \mathbb D^2$. We define $\mathcal B(E)$
	to be the full subcategory of $\Cobquot(E)$ with objects given by a full set of representatives of isotopy classes of crossingless matchings with boundary $E$.
\end{definition}
More often, we think of $\mathcal B(E)$ as a $\Z[G]$-algebra with idempotents corresponding to (the chosen) crossingless matchings. In the case $E=E^4_*$, we have the generators
\[
	\circ:=\scalebox{0.5}{\Elementary[sha1=-1, emp1, extra={\node at (-0.62,0.6) {*};}] }_,
	\qquad\qquad
	\bullet:=\scalebox{0.5}{\Elementary[dir1=1, emp1, sha1=-10, extra={\node at (-0.62,0.6) {*};} ]}_.
\]
The algebra $\mathcal B(E^4_*)$ is succinctly described as the quiver algebra
\[
	\mathcal B(E^4_*)=\End_{\Cobquot(E^{4}_*)}(\circ\oplus\bullet)
	= \left[
	\begin{tikzcd}
		\circ
		\arrow[loop left, "D"]
		\arrow[r, bend left=20, "S"]
		&
		\bullet
		\arrow[loop right, "D"]
		\arrow[l, bend left=20, "S"]
	\end{tikzcd}
	\right]_
	{/(SD=0, DS=0)},
\]
over $\Z$, where arrows labeled by $S$ are given by saddle cobordisms, and arrows labeled by $D$ are given by cylinder cobordisms with a dot on the non-distinguished component. The action of $G$ can be succinctly described by $G=S^2-D$.

The following lemma is a rephrasing of the \emph{Delooping Lemma} in \cite{bnfast}, which says that any object in $\Cobquot$ is isomorphic to a direct sum of shifted crossingless matchings.
\begin{lemma}[Delooping]
	The inclusion $\iota:\Sigma\mathcal B(E) \hookrightarrow \Cobquot(E)$ is an equivalence, and thus it induces an equivalence between $Ch(\Sigma\mathcal B(E))$ and $Ch (\Cobquot(E))$.
\end{lemma}
\begin{proof}
	Let us write an explicit inverse equivalence. The \emph{delooping functor} $\deloop:\Cobquot(E)\to \Sigma\mathcal B(E)$ is defined on objects by
	\[
		\deloop(\Gamma) := \bigoplus_{X\subseteq \mathrm{circ}(\Gamma)}q^{|\mathrm{circ}(\Gamma)|-2|X|}o(\Gamma).
	\]

	If $C:\Gamma_0\to \Gamma_1$ is a morphism between crossingless tangles, and $X\subseteq \mathrm{circ}(\Gamma_0)$, $Y\subseteq \mathrm{circ}(\Gamma_1)$, we define $\deloop(C)$ component-wise by
	\[
		\deloop(C)^{X,Y} := \cupdot_Y\circ C\circ \capdot_X.
	\]
	Here, $\capdot_X:o(\Gamma_0)\to \Gamma_0$ is a decorated cobordism consisting of a cylindrical cobordism in $o(\Gamma_0)$, caps bounding all loops in $\mathrm{circ}(\Gamma_0)$, and black dots on all caps bounding loops in $X$. Similarly, $\cupdot_Y:\Gamma_1\to o(\Gamma_1)$ is a decorated cobordism consisting of a cylinder in $o(\Gamma_1)$, cups bounding all loops in $\mathrm{circ}(\Gamma_1)$, and red dots on all cups bounding loops in $\mathrm{circ}(\Gamma_1)\setminus Y$. We have the identity $\deloop\circ\iota=\id_{\Sigma\mathcal B(E)}$ on the nose, while
	\begin{align*}
		\phi_\Gamma:=
		&
		(\capdot_X)_{X\subseteq \mathrm{circ}(\Gamma)} : \bigoplus_{X\subseteq \mathrm{circ}(\Gamma)} q^{|\mathrm{circ}(\Gamma)|-2|X|}o(\Gamma) \longrightarrow \Gamma,
		\\
		\phi^{-1}_\Gamma:=
		&
		(\cupdot_Y)_{Y \subseteq \mathrm{circ}(\Gamma)}:
		\Gamma \longrightarrow
		\bigoplus_{Y\subseteq \mathrm{circ}(\Gamma)} q^{|\mathrm{circ}(\Gamma)|-2|Y|}o(\Gamma)
	\end{align*} give explicit, mutually inverse, natural isomorphisms $\iota\circ\deloop\cong \id_{\Cobquot(E)}$.
\end{proof}

As a consequence of delooping, it is possible to produce a basis for morphism spaces in $\Cobquot(E)$ from a fixed basis of $\mathcal B (E)$. See \cite{naot} for more on bases in $\Cobquot$.

\begin{corollary}\label{cor:basis}
	Let $\Gamma_0, \Gamma_1$ be crossingless tangles in $\mathbb D^2$ with the same boundary $E$. Let $\mathbf B$ be a basis for $\homom_{\mathcal B(E)}(o(\Gamma_0),o(\Gamma_1))$ over $\Z$. Then
	\[
		\left\{\capdot_Y\circ C\circ \cupdot_X\mid C\in\mathbf B,\, X\subseteq \mathrm{circ}(\Gamma_0),\, Y\subseteq \mathrm{circ}(\Gamma_1) \right\}
	\]
	is a basis for $\homom_{\Cobquot(E)}(\Gamma_0,\Gamma_1)$ over $\Z$.
\end{corollary}

\subsection{The Bar-Natan invariant of a tangle} The Bar-Natan construction takes an oriented tangle diagram $D$ for an oriented tangle $T$ with boundary $E\subset \partial\mathbb D^2$ and associates to it a chain complex $[\![ D]\!]$ in $Ch(\Cobquot(E))$. For elementary crossings, we have a complex consisting of a saddle cobordism from the 0-resolution to the 1-resolution.
\[
	D_+=\Elementary[dir1=1, cr1=10]
	\quad\longmapsto\quad
	[\![D_+]\!]=\left[
	\begin{tikzcd}[column sep=1cm]
		\scalebox{0.8}{\Elementary[q=1, sha1=-1, emp1] }
		\arrow[r, "S"]
		&
		\scalebox{0.8}{\Elementary[dir1=1, q=2, h=1, emp1, sha1=-10 ]}
	\end{tikzcd}
	\right]
\]
\[
	D_-=\Elementary[dir1=1, cr1=-10]
	\quad\longmapsto\quad
	[\![D_-]\!]=\left[
	\begin{tikzcd}[column sep=1cm]
		\scalebox{0.8}{\Elementary[dir1=1, q=-2,h=-1, sha1=-4, emp1] }
		\arrow[r, "S"]
		&
		\scalebox{0.8}{\Elementary[q=-1, emp1, sha1=-5 ]}
	\end{tikzcd}
	\right]
\]

In our notation, homological degree shifts are indicated in the lower right, while quantum degree shifts are indicated in the upper right.
\begin{definition}
	For an arbitrary oriented tangle diagram $D$ with boundary $E$ and crossings labeled by $1,\dots,n$, the \emph{Bar-Natan bracket} is the complex over $\Cobquot(E)$ defined by
	\[
		[\![D]\!]:=\bigoplus_{s} q^{n_+-2n_-+|s|}h^{|s|-n_-}\Gamma_s,
	\]
	with indexing over all total resolutions $s\in\{0,1\}^n$. Here, given a total resolution $s$, we denote the number of 1-resolutions in $s$ by $|s|$, and the crossingless tangle obtained by applying all resolutions in $s$ to the diagram $D$ by $\Gamma_s$. The number of positive and negative crossings are denoted by $n_+$ and $n_-$, respectively. To define the differential, let us additionally denote the number of 1-resolutions among the crossings $1,\dots,i-1$ in $s$ by $|s|_{<i}$. The differential $d_{[\![D]\!]}$ has a non-zero component $s\to s'$ precisely when $s$ and $s'$ differ by a change from 0- to 1-resolution, in which case we write $s<_is'$. The precise expression is
	\[
		d_{[\![D]\!]} := \sum_{s<_is'} (-1)^{|s|_{<i}} S_{i},
	\]
	where $S_i$ is a cobordism given by a saddle at the location of the $i$-th crossing, extended by the identity everywhere else.
\end{definition}

Given an oriented tangle $T$ with boundary $E$ represented by a tangle diagram $D$, define
\[
	\De(D):= \deloop(\bnbrack{D})\in Ch(\Sigma\mathcal B(E)).
\]
Bar-Natan proves in \cite{bn} that the bracket $[\![D]\!]$, and therefore its delooping $\De(D)$, is independent of the choice of diagram, up to chain homotopy equivalence. For this reason, we may occasionally write $[\![T]\!]$ and $\De(T)$.

\section{\texorpdfstring{$A_\infty$}{A-infinity}-categories and twisted complexes}\label{sec:ainfinity}

In this section, we consider extensions of $\mathcal B(E)$ into a \emph{minimal} $A_\infty$-algebra, that is, higher operations $(\mu_n)_{n\geq 3}$ that, together with the usual multiplication $\mu_2$, endow $\mathcal B(E)$ with the structure of an $A_\infty$-algebra. Minimality here means $\mu_1\equiv 0$. The motivating example is an extension of $\mathcal B(E^4_*)$ denoted by $\mathcal B^\infty_{ds}$. In subsequent sections, we study the extension of the tangle invariant $\De(D)$ into a bifiltered twisted complex over $\mathcal B^\infty_{ds}$, making precise how this extension affords a tangle invariant.

\subsection{Preliminaries on \texorpdfstring{$A_\infty$}{A-infinity}-categories and twisted complexes}\label{sec:prelimainfinity}
We follow the standard formalism of \cite{seidel}, with some minor adaptations to account for the gradings that appear in the study of tangle invariants.
\begin{definition}
	Let $\mathbf K$ be a commutative ring. An $A_\infty$\emph{-category} $\mathcal C$ over $\mathbf K$ consists of a set of objects $Ob(\mathcal C)$, a $\Z$-graded $\mathbf K$-module $\mathcal C(X,Y)$ for each $X,Y\in Ob(\mathcal C)$, and linear structure maps
	\[
		\mu_n:\mathcal C(X_{n-1},X_n)\otimes\cdots\otimes\mathcal C(X_0,X_1)\to \mathcal C(X_0,X_n)
	\]
	of degree $2-n$ for $n\geq 1$. We refer to $\mu_1$ as the \emph{differential}, $\mu_2$ as the \emph{composition}, and $\mu_n$ as the \emph{higher compositions} for $n\geq 3$. There are \emph{units} $1_X\in\mathcal C(X,X)$ which satisfy $\mu_2(1,a)=a$ and $\mu_2(a,1)=(-1)^{|a|}a$. Higher compositions vanish if any input is a unit. The structure maps satisfy the $A_\infty$-relations
	\[
		0 = \sum_{i+j+k=n} (-1)^\odot \mu_{i+1+k}(a_n,\dots,a_{n-i+1},\mu_j(a_{n-i},\dots,a_{k+1}),a_k,\dots,a_1)
	\]
	for all composable sequences $a_1,\dots,a_n$, where $\odot=|a_k|+\cdots +|a_1|+k$, and $|-|$ denotes the alluded $\Z$-grading. We call $|-|$ the $A_\infty$\emph{-grading}. An $A_\infty$-category is said to be $q$-filtered if there is an additional degree $q$ on each $\mathcal C(X,Y)$ such that the operations $\mu_n$ are homogeneous of degree $q(\mu_n)=3(n-2)$.
\end{definition}

In this paper, an $A_\infty$\emph{-algebra} is an $A_\infty$-category with finitely many objects, which are referred to as \emph{generating orthogonal idempotents}. A dg-category $\mathcal C$ with differential $d$ and composition $\circ$ is regarded as an $A_\infty$-category via $\mu_1(a)=(-1)^{|a|}d(a)$, $\mu_2(b,a)=(-1)^{|a|}b\circ a$, and $\mu_n\equiv 0$ for $n\geq 3$.

An $A_\infty$-\emph{functor} between $A_\infty$-categories $\mathcal C$ and $\mathcal C'$ consists of a map $\mathcal F:Ob(\mathcal C)\to Ob(\mathcal C')$ and a sequence of linear maps of the form
\[
	\mathcal F_n: \mathcal C(X_{n-1},X_n)\otimes\cdots\otimes\mathcal C(X_0,X_1)\to \mathcal{ C'}(\mathcal F(X_0),\mathcal F(X_n))
\]
of $A_\infty$-degree $1-n$ for each $X_0,\dots, X_n\in Ob(\mathcal C)$, satisfying

\begin{align*}
	&\sum_{r} \sum_{i_1+\cdots+i_r=n} \mu_r^{\mathcal C'}\left(\mathcal F_{i_r}(a_n,\dots,a_{n-i_r+1}),\dots,\mathcal F_{i_1}(a_{i_1},\dots,a_1)\right)
	\\\quad&=
	\sum_{j,k} (-1)^\odot \mathcal F_{n-j+1}\left(a_n,\dots,a_{j+k+1},\mu^{\mathcal C}_{j}(a_{j+k},\dots,a_{k+1}),a_k,\dots,a_1\right).
\end{align*}
We call $\mathcal F$ \emph{unital} if $\mathcal F_1(1_X)=1_{\mathcal F(X)}$ and $\mathcal F_n(a_n,\dots,a_1)=0$ for $n\ge2$ whenever some $a_i$ is a unit. A functor $\mathcal F$ between $q$-filtered $A_\infty$-categories is $q$-filtered if $q(\mathcal F_n)=3(n-1)$.

\begin{example}\label{ex:diskseq}
	Consider the quiver algebra
	\[
		\mathcal B:= \mathcal B(E^4_*) =
		\left[
		\begin{tikzcd}
			\circ
			\arrow[loop left, "D"]
			\arrow[r, bend left=20, "S"]
			&
			\bullet
			\arrow[loop right, "D"]
			\arrow[l, bend left=20, "S"]
		\end{tikzcd}
		\right]_
		{/(SD=0, DS=0)}.
	\]
	Following \cite{hkk}, we enhance $\mathcal B$ into an $A_\infty$-algebra $\mathcal B^\infty_{ds}$ over $\Z$ by defining an $A_\infty$-grading and higher operations, as follows. For the grading, we declare
	\begin{align*}
		|\circ \xrightarrow{S}\bullet|=|\bullet \xrightarrow{S}\circ| = |\bullet \xrightarrow{D}\bullet| = 0, \qquad |\circ \xrightarrow{D}\circ|=2.
	\end{align*}
	In order to describe the higher operations, we begin by recursively defining what it means for an even-length sequence of paths in $\mathcal B$ to be a \emph{disk sequence}. A \emph{disk sequence} of length 4 is a composable sequence of arrows of the form $(S,D,S,D)$ or $(D,S,D,S)$. If $(a_{2k},\dots, a_{1})$ is a \emph{disk sequence} of length $2k$, then any sequence of composable paths of the form $(a_{2k},\dots,a_iS,D,S,Da_{i-1},\dots, a_{1})$ with $a_iS\neq 0$ is a disk sequence of length $2k+2$, and so is any sequence of composable paths of the form $(a_{2k},\dots,a_iD,S,D,Sa_{i-1},\dots, a_{1})$ with $a_iD\neq 0$. A disk sequence is defined to be any sequence obtained by iterated application of this procedure. Define $\mu_{2k}^{\mathcal B^\infty_{ds}}:\mathcal B^{\otimes 2k}\to \mathcal B$, for $2k\geq 4$, by declaring
	\begin{align*}
		\mu_{2k}^{\mathcal B^\infty_{ds}} (a_{2k},\dots,a_1b)
		&
		:= b ,
		\\
		\mu_{2k}^{\mathcal B^\infty_{ds}} (ba_{2k},\dots,a_1)
		&
		:= b ,
	\end{align*}
	where $(a_{2k},\dots,a_{1})$ is a disk sequence and $b$ is any path such that $a_1b\neq 0$ (resp. $ba_{2k}\neq 0$). The operation $\mu_{2k}^{{\mathcal B}^\infty_{ds}}$ is also defined to vanish on any input sequence of composable paths that does not have either of the indicated forms. All odd operations $\mu_{2k+1}^{{\mathcal B}^\infty_{ds}}$ vanish identically. Observe that $\mathcal B^\infty_{ds}$ is $q$-filtered with respect to the usual quantum degree in the Bar-Natan category.

\end{example}

Given an $A_\infty$-category $\mathcal C$ over a commutative ring $\mathbf K$, we define its \emph{additive enlargement} $\Sigma \mathcal C$ to be the $A_\infty$-category with objects given by formal expressions
\begin{equation}\label{eq:objectadditive}
	X=\bigoplus_{i\in I} V^i \otimes X^i,
\end{equation}
where $I$ is a finite indexing set, $V^i$ a free (quantum and $A_\infty$-) graded $\mathbf K$-module of finite rank, and $X^i\in Ob(\mathcal C)$. The morphism spaces are matrices with entries generated by elements of the form $\phi\otimes f$, with $\phi$ a homogeneous $\mathbf K$-linear map and $f$ a morphism in $\mathcal C$. More explicitly,
\[
	\left(\sum_k\phi^{jik}\otimes f^{jik}\right)_{i,j}
	\in
	\homom_{\Sigma\mathcal C}\left(\bigoplus_{i\in I} V_0^i \otimes X_0^i, \bigoplus_{j\in J} V_1^j \otimes X_1^j\right)
	=
	\bigoplus_{i,j}\homom_{\mathbf K}(V_0^i,V_1^j)\otimes \homom_{\mathcal C} (X_0^i,X_1^j).
\]
The structure maps are defined by
\[
	\mu^{\Sigma \mathcal C}_d(a_d,\dots,a_1)^{i_d,i_0}
	:=
	\sum_{i_1,\dots, i_{d-1}}
	(-1)^\vartriangleleft
	\phi_d^{i_d,i_{d-1}}\cdots \phi_1^{i_1,i_{0}}
	\otimes
	\mu^{\mathcal C}_d(f_d^{i_d,i_{d-1}},\dots,f_1^{i_1,i_0}),
\]
where
$\vartriangleleft
=
\sum_{p<q} |\phi_p^{i_p,i_{p-1}}|\cdot(|f_q^{i_q,i_{q-1}}|-1)$
and
$a^{ji}_p=\phi_p^{ji}\otimes f_p^{ji}.$ We also consider the $\hslash$\emph{-graded additive enlargement} $\Sigma_{\hslash}\mathcal C$, in which we consider the $V^i$ to be graded by the $\Z/2$-reduction of $|-|$, denoted $\hslash$. Notice that the structure maps are still well defined over $\Sigma_{\hslash}\mathcal C$.

The category $Tw\mathcal C$ of $\hslash$\emph{-graded twisted complexes over} $\mathcal C$ has objects given by pairs $(X,\delta)$ with $X\in Ob(\Sigma_\hslash\mathcal C)$ and $\delta\in\homom_{\Sigma_\hslash\mathcal C}(X,X)$ a morphism of $\hslash$-degree 1 satisfying the \emph{Maurer--Cartan equation}
\[
	\sum_{d=1}^\infty \mu^{\Sigma_{\hslash}\mathcal C}_d(\delta,\dots,\delta)=0.
\]
We call $\delta$ the \emph{twisted differential}. In this paper, we always work with an additional grading that ensures the finiteness of the sum in the Maurer--Cartan equation. A \emph{homological grading} is a function $h:I\to \Z$, where $I$ is the indexing set of $X$. Given an $\hslash$-graded twisted complex $(X,\delta)$ over a $q$-filtered $A_\infty$-category $\mathcal C$, we say that $(X,\delta,h)$ is a \emph{bifiltered twisted complex} if $\delta^{ji}=0$ whenever $h(j)\leq h(i)$, i.e. $\delta$ is strictly upper triangular with respect to the decreasing filtration on $X$ given by
\begin{equation}\label{eq:decreasingfiltration}
	\mathsf F_p X = \bigoplus_{h(i)\geq p} V^i\otimes X^i.
\end{equation}
Thus, the twisted differential $\delta$ has components $\delta_k$ of $h$-degree $h^k$ for $k>0$. We also require $\delta_k$ to have $(\hslash, q)$-degree $\hslash^1q^{3(k-1)}$. The category of bifiltered twisted complexes over $\mathcal C$ is denoted by $Tw^h\mathcal C$.

\begin{example}
	Given a bifiltered twisted complex $(\De^\infty,\delta)\in Tw^h \mathcal B^\infty_{ds}$, there is always an underlying chain complex $(\De,\delta_1)\in Ch(\Sigma\mathcal B)$ obtained by forgetting the $\hslash$-grading and all components $\delta_k$ for $k>1$. In $(\De,\delta_1)$, the homological grading is still $h$, and the differential $\delta_1$ strictly preserves $q$-grading. We say that $\De^\infty$ is an \emph{extension} of $\De$.
\end{example}

Given two bifiltered twisted complexes $(X,\delta)$ and $(X',\delta')$ over an $A_\infty$-category $\mathcal C$, the morphism space is given by $\homom_{Tw^h\mathcal C}((X,\delta),(X',\delta'))=\homom_{\Sigma_{\hslash}\mathcal C}(X,X')$.
There are structure maps given by
\[
	\mu_d^{Tw^h\mathcal C}(a_d,\dots,a_1)
	:=
	\sum_{i_0,\dots,i_d} \mu^{\Sigma_{\hslash}\mathcal C}_{d+i_0+\cdots+i_d}(\underbrace{\delta_{X_d},\dots,\delta_{X_d}}_{i_d}, a_d, \underbrace{\delta_{X_{d-1}},\dots,\delta_{X_{d-1}}}_{i_{d-1}}, a_{d-1}, \dots, a_1, \underbrace{\delta_{X_0},\dots,\delta_{X_0}}_{i_0}).
\]
A morphism $f$ is a \emph{homomorphism} if $\mu_1^{Tw^h\mathcal C}(f)=0$ and its homogeneous component $f_k$ in $h$-degree $h^k$ is also $(\hslash,q)$-homogeneous of degree $\hslash^0 q^{3k}$. A morphism $\kappa$ is a \emph{homotopy} between two homomorphisms $f$ and $g$ if $\mu_1^{Tw^h\mathcal C}(\kappa)=f-g$ and $\kappa_k$ is $(\hslash,q)$-homogeneous of degree $\hslash^1 q^{3(k+1)}$. We say that a homomorphism $f$ is a \emph{homotopy equivalence} if there exist homomorphisms $g_l$ and $g_r$ such that $\mu_2^{Tw^h\mathcal C}(g_l,f)$ and $\mu_2^{Tw^h\mathcal C}(f,g_r)$ are homotopic to identity morphisms.

Now, let $(X,\delta,h)$ and $(X',\delta',h')$ be bifiltered twisted complexes. Let $a:X'\to X$ be a morphism in $Tw^h\mathcal C$ and let $a_k:X'\to X$ be its homogeneous component in $h$-degree $h^k$. We say that $a$ is an $n$-morphism if $a_k\equiv 0$ for $k<-n$. Notice that $\delta$ is always a $(-1)$-morphism. If an $n$-morphism $a$ is also a homomorphism or a homotopy, we say that $a$ is an $n$-\emph{homomorphism} or an $n$-\emph{homotopy}, respectively. Observe that a $0$-morphism (and in consequence an $n$-morphism for any $n\leq 0$) preserves the decreasing filtration \eqref{eq:decreasingfiltration}.

A $0$-homomorphism $f$ is an $n$-\emph{homotopy equivalence} if it is a homotopy equivalence witnessed by homotopy inverses $g_l$ and $g_r$ which are 0-homomorphisms, and by homotopies $\kappa_l$ and $\kappa_r$ which are $n$-homotopies. This defines an equivalence relation. In the study of tangle invariants, we will be mostly interested in the notion of $1$-homotopy equivalence, which is natural in view of the following.

\begin{example}
	Let $f:(\De^\infty,\delta)\to(\tilde\De^\infty,\tilde\delta)$ be a 1-homotopy equivalence in $Tw^h\mathcal B^\infty_{ds}$, witnessed by $(g_l,g_r,\kappa_l,\kappa_r)$. Then the map $f_0:(\De,\delta_1)\to(\tilde\De,\tilde\delta_1)$ between underlying chain complexes is an ordinary chain homotopy equivalence, witnessed by $((g_l)_0,(g_r)_0,(\kappa_l)_{-1},(\kappa_r)_{-1})$.
\end{example}

\begin{remark}
	Given a unital ($q$-filtered) $A_\infty$-functor $\mathcal F:\mathcal C\to \mathcal C'$, there are induced $A_\infty$-functors at the level of additive enlargements, twisted complexes and bifiltered twisted complexes. These induced functors preserve the notions of ($n$-)homomorphism, ($n$-)homotopy and ($n$-)homotopy equivalence.
\end{remark}

\subsection{\texorpdfstring{$A_\infty$}{A-infinity}-enhancements of Bar-Natan categories}

Example \ref{ex:diskseq} provides a particular $A_\infty$-extension $\mathcal B^\infty_{ds}$ of $\mathcal B(E^4_*)$. Given a minimal extension of $\mathcal B (E)$, the following theorem endows all of $\Cobquot(E)$ with the structure of an $A_\infty$-category. The extension of $\Cobquot(E)$ is completely determined by the requirement that $\deloop$ trivially extends into an $A_\infty$-functor, that is, the family defined by $\mathcal F_1=\deloop$ and $\mathcal F_k=0$ for $k>1$ forms an $A_\infty$-functor.

\begin{theorem}\label{thm:barnatanainfinity}
	Let $E$ be a finite subset of the circle $\partial \mathbb D^2$. Assume there are higher operations $(\mu^{\mathcal B^\infty(E)}_n)_{n\geq 3}$ and a grading defining a minimal $A_\infty$-extension $\mathcal B^\infty(E)$ of $\mathcal B(E)$. In other words, $\mu^{\mathcal B^\infty(E)}_1$ vanishes and $\mu^{\mathcal B^\infty(E)}_2$ is the usual multiplication in $\mathcal B(E)$. Then, there is a unique minimal $A_\infty$-extension $\Cobquot^\infty(E)$ of $\Cobquot(E)$ so that $\deloop$ trivially extends into an $A_\infty$-functor $\Cobquot^\infty(E)\to \Sigma \mathcal B^\infty(E)$. In the basis of Corollary \ref{cor:basis}, these higher operations are explicitly given by
	\[
		\mu_n^{\Cobquot^\infty(E)}(\capdot_{Y_n}\circ C_n\circ \cupdot_{X_n},\dots,\capdot_{Y_1}\circ C_1\circ \cupdot_{X_1})
		=
		\begin{cases}
			\capdot_{Y_n}\circ\mu_n^{\mathcal B^\infty(E)}(C_n,\dots,C_1)\circ\cupdot_{X_1}, & \mbox{if every } Y_i=X_{i+1},
			\\
			0, & \mbox{otherwise.}
		\end{cases}
	\]

\end{theorem}

\begin{proof}
	Unraveling the desired functoriality condition, we need $\mu_n^{\Cobquot^\infty(E)}$ to satisfy that, given any sequence of composable cobordisms $C_1,\dots,C_n$ with $C_i:\Gamma_{i-1}\to\Gamma_i$,
	\[
		\deloop\left(\mu_n^{\Cobquot^\infty(E)}(C_n,\dots,C_1)\right)
		= \mu_n^{\Sigma\mathcal B^\infty(E)}\left(\deloop(C_n),\dots,\deloop(C_1)\right).
	\]
	This forces
	\[
		\mu_n^{\Cobquot^\infty(E)}(C_n,\dots,C_1)
		= \phi_{\Gamma_n}\circ\mu_n^{\Sigma\mathcal B^\infty(E)}\left(\phi^{-1}_{\Gamma_n}\circ C_n\circ \phi_{\Gamma_{n-1}},\dots,\phi^{-1}_{\Gamma_1}\circ C_1\circ \phi_{\Gamma_0}\right)\circ\phi^{-1}_{\Gamma_0}.
	\]
	Taking this equation as the definition of higher operations in $\Cobquot^\infty(E)$, one obtains the structure of an $A_\infty$-category, with the $A_\infty$-relations following directly from those in $\Sigma\mathcal B^\infty(E)$. The $A_\infty$-grading is also inherited from $\Sigma\mathcal B^\infty(E)$. The basis expression for these operations is readily verified, making use of the observation that
	\[
		(\phi^{-1}_{\Gamma'}\circ\capdot_{Y'}\circ C \circ \cupdot_{X'}\circ\phi_{\Gamma})^{X,Y} = \cupdot_Y\circ\capdot_{Y'}\circ C \circ \cupdot_{X'}\circ\capdot_X
		=\begin{cases}
			C, & \mbox{if } X'=X \mbox{ and } Y'=Y,
			\\
			0, & \mbox{otherwise}.
		\end{cases}
	\]
\end{proof}

\section{Matrix factorizations}\label{sec:mf}

\begin{definition}
	Let $R$ be a polynomial ring over $\Z$. A \emph{matrix factorization} over $R$ with potential $W\in R$ is a $\Z/2$-graded free $R$-module $M$ together with a degree 1 $R$-module homomorphism $d_M$ such that $d_M^2=W\cdot\id$. If $R$ and $M$ carry additional gradings, we require $d_M$ to be homogeneous with respect to these additional gradings.
\end{definition}

In this paper, we always denote the $\Z/2$-grading internal to matrix factorizations by $\hslash$. If $R$ is a polynomial ring over $\Z$, we allow it to carry various integer-valued gradings, such that the variables generating $R$ are always even. We succinctly write $\mulg(x)\in (2\Z)^{\times N}$ for $\mulg$-homogeneous $x\in R$. A degree shift by $\mulg_1^{k_1}\cdots \mulg_N^{k_N}$ is thus simply denoted by $\mulg^{v}$, where $v=(k_1,\dots,k_N)$. Among these gradings, there is often a \emph{quantum} degree $q$ and a \emph{homological} degree $h$. These are such that polynomial generators live in quantum degree $q^{-2}$ and homological degree $h^0$. The potential $W$ is always required to be $\mulg$-homogeneous.

\begin{definition}

	Let $M$ and $N$ be matrix factorizations over $R$ with potential $W$. A \emph{morphism} $f: M\to N$ is a map of $R$-modules $f:M\to N$ such that $ fd_M=(-1)^{\hslash(f)}d_N f$. If it has $\hslash$-degree 0, we also call it a \emph{homomorphism}. We keep referring to maps $f: M\to N$ of the underlying $R$-modules simply as \emph{maps}.

	A \emph{homotopy} $\kappa: M\to N$ between homomorphisms $f,g:M\to N$ is an $\hslash$-degree 1 map such that $\kappa d_M+d_N\kappa =f-g$. We write $f\sim g$ when there is a homotopy between $f$ and $g$. Note that we require homomorphisms to preserve additional gradings $\mulg$, and homotopies to be homogeneous of degree $\mulg(\kappa)=-\mulg(d_M)$.

	A homomorphism $f: M\to N$ is said to be a \emph{homotopy equivalence} if it is invertible up to homotopy on both sides, i.e. if there exist homomorphisms $g_l, g_r: N\to M$ such that $fg_r\sim 1_N$ and $g_lf\sim1_M$.
\end{definition}

\begin{definition}
	Let $M$ and $N$ be matrix factorizations over $R$ with potential $W$, let $\iota:M\to N$ and $\rho:N\to M$ be homomorphisms, and let $\kappa$ be a homotopy between $\iota\rho$ and $1_N$. We say that $(M,N,\iota,\rho,\kappa)$ is a \emph{special deformation retract} if the maps satisfy the additional conditions $\rho\iota=1_M$, $\rho\kappa=0$, $\kappa\iota=0$ and $\kappa^2=0$.
\end{definition}

Notice that, in particular, $\iota$ and $\rho$ are homotopy equivalences. We will mostly work with homotopy equivalences which are special deformation retracts. Their usefulness will become clear when proving equivalences of twisted complexes of matrix factorizations in subsequent sections. We often use the color violet in diagrams such as
\[
	\vsdr{M}{N}{\iota}{\rho}{\kappa}
\]
to indicate that the provided maps constitute a special deformation retract of matrix factorizations.

Matrix factorizations over $R$ with potential $W$ and maps between them form a linear category $MF(R,W)$. If one considers the differential $\delta(f:M\to N) = d_Nf + (-1)^{\hslash(f)+1}fd_M$ and the usual functional composition $\circ$, then $MF(R,W)$ is a $\Z/2$-graded dg-category. We refrain from referring to $MF(R,W)$ as an $A_\infty$-category because the $A_\infty$-categories of Subsection \ref{sec:prelimainfinity} are $\Z$-graded. Observe also that matrix factorizations over $R$ with potential $W$ and homomorphisms between them form a linear category, which we denote by $Z^0(MF(R,W))$.

There is a tensor product operation on matrix factorizations given by $\otimes_R$ at the level of graded $R$-modules. Our convention here is that if $f$ and $g$ are homogeneous maps and $a$ and $b$ are homogeneous elements, then
\[
	(f\otimes g)(a\otimes b)= (-1)^{\hslash(g)\hslash(a)}f(a)\otimes g(b).
\]
The differential in a tensor product is given by
\[
	d_{M\otimes N}(a\otimes b) = d_M(a)\otimes b + (-1)^{\hslash(a)}a\otimes d_N(b).
\]

We write $\mulg^v\hslash^\epsilon R$ for the matrix factorization consisting of a rank 1 free $R$-module concentrated in degree $\mulg^v\hslash^\epsilon $ and vanishing potential and differential. We can shift an arbitrary matrix factorization $M$ over $R$ by tensoring $\mulg^v\hslash^\epsilon R$ on the left or on the right, operations respectively denoted by
\begin{align*}
	\mulg^v\hslash^\epsilon M &:= \mulg^v\hslash^\epsilon R\otimes_R M, \\ M \mulg^v\hslash^\epsilon &:= M\otimes_R \mulg^v\hslash^\epsilon R.
\end{align*}
In general, any shifts by $\mulg$ and right shifts by $\hslash$ do not change the sign of $d_M$, while left shifts by $\hslash$ do. The matrix factorization with the same underlying module as $M$ but opposite sign differential is denoted by $-M$.

Using the notation just established, we have that $\mulg^v\hslash^\epsilon M \cong (-1)^\epsilon M\mulg^v\hslash^\epsilon $. In particular, there is an identification
\[
	M\hslash\otimes_R N=M\otimes_R(\hslash N)=M\otimes_R(-N\hslash),
\]
which may be used to move $\hslash$-grading shifts all the way to the right when manipulating longer tensor products:
\begin{align*}
	M_1\hslash^{\epsilon_1}\otimes_R M_2\hslash^{\epsilon_2}\otimes_R\cdots \otimes_R M_r\hslash^{\epsilon_r}
	\cong (M_1\otimes_R (-1)^{\epsilon_1} M_2\otimes_R\cdots\otimes_R (-1)^{\epsilon_1+\cdots+\epsilon_{r-1}} M_r) \hslash^{\epsilon_1+\cdots+\epsilon_r}.
\end{align*}

\subsection{Koszul matrix factorizations}

Let us introduce the basic building blocks for the matrix factorizations that appear in the rest of the paper. Let $a,b\in R$ be homogeneous elements with respect to the even grading $\mulg$. The associated \emph{elementary Koszul matrix factorization} with potential $W=ab$ is defined by
\[
	\kmf{R}{a & b }
	: =
	\left[\mfshort{\mulg^{\frac12 (\mulg(a)-\mulg(b))} \hslash R}{R}{a}{b}\right]
	.
\]
As an $R$-module, we think of $\kmf{R}{a & b }$ as the exterior algebra $R\otimes\Lambda(\vartheta)$ with $R$-basis $\left\{\vartheta,1\right\}$. The generator $\vartheta$ lives in degree $\mulg^{\frac12 (\mulg(a)-\mulg(b))} \hslash^1$, while $1$ lives in degree $\mulg^{0}\hslash^0$. In the case $a=0$, the expression $\mulg(a)$ appearing in the degree of $\vartheta$ is a priori undefined, but the intended value will be clear from context in subsequent sections, and similarly for $b=0$. The differential can be written as
\[
	d_{\skmf{R}{a & b }}=\begin{pmatrix}
		0 & b \\ a & 0
	\end{pmatrix}=a\partial + b\theta.
\]
In this expression, $\partial$ and $\theta$ represent, respectively, the usual contraction (i.e. odd variable differentiation) and multiplication operators with respect to $\vartheta$ in the Koszul algebra:
\[
	\partial(c_1 1+c_2\vartheta)=c_2, \qquad \theta(c_1 1+c_2\vartheta)=c_1\vartheta.
\]
More generally, given sequences $\mathbf{\overline a}=(a_1,\dots,a_r)^t$ and $ \mathbf{\overline b}=(b_1,\dots,b_r)^t$, we have an associated Koszul factorization
\[
	\kmf{R}{\mathbf{\overline a} & \mathbf{\overline b} }
	=
	\bigotimes_i \kmf{R}{a_i & b_i },
\]
with potential $\sum_i a_ib_i$. For any shift of a Koszul matrix factorization, say $\kmf{R}{\mathbf{\overline a} & \mathbf{\overline b} }\mulg^v\hslash^\epsilon$, we have a \emph{Koszul basis} $\left\{\widehat\vartheta_1\otimes\cdots\otimes\widehat\vartheta_r\mid \widehat\vartheta_i=\vartheta_i \mbox{ or } 1\right\}$ as an $R$-module. The degree of $1\otimes\cdots\otimes 1$ is $\mulg^v\hslash^\epsilon$, while the other basis elements are shifted according to which $\vartheta_i$ appear. In practice, it is more useful to think of elements in $\kmf{R}{\mathbf{\overline a} & \mathbf{\overline b}}$ as integral polynomials in the even variables in $R$ and the odd variables $\vartheta_i$. From this perspective, the $R$-basis can be rewritten as $\left\{\vartheta_{i_1}\cdots\vartheta_{i_k}\mid 1\leq i_1<\cdots<i_k\leq r\right\}$. This notation extends to the case when the indexing set is only a subset of $\{1,\dots,r\}$.

Given two possibly shifted Koszul factorizations of the same length, we have several left multiplication and contraction operators,
\[
	\theta_i, \partial_i
	:
	\kmf{R}{\mathbf{\overline a} & \mathbf{\overline b} }\mulg^v\hslash^\epsilon
	\longrightarrow
	\kmf{R}{\mathbf{\overline a}' & \mathbf{\overline b}' }\mulg^{v'}\hslash^{\epsilon'}.
\]
In the Koszul basis,
\begin{align*}
	\theta_i(\vartheta_{i_1}\cdots\vartheta_{i_k}) &=
	\begin{cases}
		0,
		&
		\mbox{if } i\in\{i_1,\dots,i_k\},
		\\
		(-1)^{j-1}\vartheta_{i_1}\cdots\vartheta_{i_{j-1}}\vartheta_i\vartheta_{i_{j}}\cdots \vartheta_{i_k},
		&
		\mbox{if } i_{j-1}<i<i_j, \mbox{ allowing } i_0=-\infty, i_{k+1}=\infty,
	\end{cases}
	\\
	\partial_i(\vartheta_{i_1}\cdots\vartheta_{i_k})
	&=
	\begin{cases}
		(-1)^{j-1}\vartheta_{i_1}\cdots\vartheta_{i_{j-1}}\vartheta_{i_{j+1}}\cdots \vartheta_{i_k},
		&
		\mbox{if } i=i_j
		\\
		0,
		&
		\mbox{if } i\notin\{i_1,\dots,i_k\}.
	\end{cases}
\end{align*}
In the case when the two matrix factorizations do not have the same indexing set, we define $\theta_i$ to vanish when the index $i$ does not appear in the target, and we define $\partial_i$ to vanish when the index $i$ does not appear in the source.

The differential of $M=\kmf{R}{\mathbf{\overline a} & \mathbf{\overline b} }\mulg^v\hslash^\epsilon $ can be written using this notation as
\[
	d_{M}
	=
	\sum_{i=1}^r a_i\partial_i+b_i\theta_i.
\]

The following properties readily follow from the definitions.
\begin{lemma}\label{lem:fermioniccalculusformulas} The following identities of operators hold.
	\begin{align*}
		\theta_i\theta_j
		&
		= -\theta_j\theta_i,
		&
		\partial_i\partial_j
		&
		= -\partial_j\partial_i,
		\\
		\partial_i\partial_i
		&
		= 0,
		&
		\theta_i\theta_i
		&
		= 0,
		\\
		\partial_i\theta_i+\theta_i\partial_i
		&
		= 1 ,
		&
		\partial_i\theta_j
		&
		= -\theta_j\partial_i \quad\mbox{if } i\neq j,
		\\
		\partial_i\theta_i\partial_i
		&
		=\partial_i,
		&
		\theta_i\partial_i\theta_i
		&
		=\theta_i .
	\end{align*}
\end{lemma}

When tensoring a Koszul factorization $K_{R}\begin{pmatrix} a & b \end{pmatrix}$ with some other matrix factorization $M$, we sometimes write
\begin{equation}\label{koszulexpand}
	\kmf{R}{a & b }\otimes_R M
	=
	\left[ \mfshort{\mulg^{\frac12 (\mulg(a)-\mulg(b))}\hslash M}{M}{a}{b} \right]
	.
\end{equation}
This indicates that the tensor product splits as a free $R$-module into two copies of $M$, the first one with shifted $\mulg$ and $\hslash$ degrees. In terms of this decomposition, the total differential takes the form
\[
	d_{\skmf{R}{a & b } \otimes M}
	=
	\begin{pmatrix}
		-d_M & b
		\\
		a & d_M
	\end{pmatrix}
	.
\]

Let us write down some useful isomorphisms between Koszul factorizations.

\begin{lemma}[Row operations]
	Let $\lambda\in R$ and denote the standard $i$-th unit vector by $\mathbf{\overline e}_i$. The following are isomorphisms of matrix factorizations.
	\begin{align*}
		l_{i,j}(\lambda) = 1+\lambda\partial_{i}\partial_{j}:\kmf{R}{\mathbf{\overline a} &\mathbf{\overline b} }
		& \longrightarrow
		\kmf{R}{\mathbf{\overline a}-\lambda b_i\mathbf{\overline e}_j+\lambda b_j\mathbf{\overline e}_i & \mathbf{\overline b} },
		\\
		r_{i,j}(\lambda) = 1+\lambda{\theta_i\theta_j}:\kmf{R}{\mathbf{\overline a} &\mathbf{\overline b} }
		& \longrightarrow
		\kmf{R}{\mathbf{\overline a} & \mathbf{\overline b} -\lambda a_i\mathbf{\overline e}_j+\lambda a_j\mathbf{\overline e}_i},
		\\
		m_{i,j}(\lambda) = 1+\lambda \theta_i\partial_{j}:\kmf{R}{\mathbf{\overline a} & \mathbf{\overline b} }
		& \longrightarrow
		\kmf{R}{\mathbf{\overline a}-\lambda a_i\mathbf{\overline e}_j & \mathbf{\overline b}+\lambda b_j\mathbf{\overline e}_i },
		\\
		s_i = \partial_i\theta_i - \theta_i\partial_i:\kmf{R}{\mathbf{\overline a} &\mathbf{\overline b} }
		& \longrightarrow
		\kmf{R}{\mathbf{\overline a}-2a_i\mathbf{\overline e}_i & \mathbf{\overline b} -2b_i\mathbf{\overline e}_i}.
	\end{align*}
	In the first three isomorphisms, we require $i\neq j$ and $\lambda$ to be homogeneous of suitable degree to make the isomorphism itself homogeneous.
\end{lemma}

\begin{proof}
	First, we verify that $l_{i,j}(\lambda)$ is a morphism of matrix factorizations by checking
	\[
		l_{i,j}(\lambda)\sum_k (a_k\partial_k+b_k \theta_k)
		=
		\left(\lambda(b_j\partial_i-b_i\partial_j)+\sum_k(a_k\partial_k+b_k \theta_k)\right)l_{i,j}(\lambda)
	\]
	Since $l_{i,j}(\lambda)=1+\lambda\partial_{i}\partial_{j}$ commutes with $\sum_{k\neq i,j} (a_k\partial_k+b_k \theta_k)$, we only need to check that
	\[
		(1+\lambda\partial_{i}\partial_{j})
		\left(a_i\partial_i+b_i\theta_i+a_j\partial_j +b_j\theta_j\right)
		=
		\left((a_i+\lambda b_j)\partial_i+b_i\theta_i+(a_j-\lambda b_i)\partial_j +b_j\theta_j\right)
		(1+\lambda\partial_{i}\partial_{j}).
	\]
	After distributing and canceling out terms appearing on both sides, as well as terms that involve $\partial_i^2$ or $\partial_j^2$, the equation boils down to
	\[
		b_i\lambda\partial_{i}\partial_{j}\theta_i+b_j\lambda\partial_{i}\partial_{j}\theta_j
		=
		b_j\lambda \partial_i - b_i\lambda \partial_j
		+b_i\lambda\theta_i\partial_{i}\partial_{j}+b_j\lambda\theta_j\partial_{i}\partial_{j}
		.
	\]
	This last identity is a consequence of the equation $1=\theta_k\partial_k+\partial_k\theta_k$. To prove that $l_{i,j}(\lambda)$ is an isomorphism, we observe that $l_{i,j}(\lambda)^{-1}=l_{i,j}(-\lambda)$. The cases of $m_{i,j}$ and $r_{i,j}$ are completely analogous.

	For $s_i$, we want to verify
	\[
		s_i\sum_k (a_k\partial_k+b_k \theta_k)
		=
		\left(-a_i\partial_i-b_i\theta_i+\sum_{k\neq i}(a_k\partial_k+b_k \theta_k)\right)s_i
	\]
	Since $s_i=\partial_i\theta_i - \theta_i\partial_i$ commutes with $\sum_{k\neq i} (a_k\partial_k+b_k \theta_k)$, the verification reduces to checking
	\[
		(\partial_i\theta_i - \theta_i\partial_i)
		\left(a_i\partial_i+b_i\theta_i\right)
		=
		\left(-a_i\partial_i-b_i\theta_i\right)
		(\partial_i\theta_i - \theta_i\partial_i).
	\]
	But from the identities in Lemma \ref{lem:fermioniccalculusformulas} it follows that both sides are equal to $a_i\partial_i-b_i\theta_i$.
\end{proof}

\begin{lemma}\label{lem:labelperm}
	Let $\mathbf{\overline a}'$ and $\mathbf{\overline b}'$ be obtained from $\mathbf{\overline a}$ and $\mathbf{\overline b}$ respectively by swapping the $i$-th and $j$-th entries: $a_i\leftrightarrow a_j$ and $b_i\leftrightarrow b_j$. Then we have a swap homomorphism $\tau_{i,j}:\kmf{R}{\mathbf{\overline a} &\mathbf{\overline b} } \to \kmf{R}{\mathbf{\overline a}' & \mathbf{\overline b}' }$ of the form
	\begin{align*}
		\tau_{i,j}
		&= \theta_i\partial_j+\theta_j\partial_i + \theta_j\theta_i\partial_j\partial_i + \partial_j\partial_i\theta_i\theta_j.
	\end{align*}
	It satisfies $\tau_{i,j}\tau_{i,j}=1$.
\end{lemma}

\begin{proof}
	It follows by using the identities in Lemma \ref{lem:fermioniccalculusformulas} to verify that $\tau_{i,j}$ intertwines $\partial_i$ with $\partial_j$ and $\theta_i$ with $\theta_j$, and so it must intertwine the differentials. The involution property is also a consequence of these identities.
\end{proof}

\subsection{Matrix factorizations associated to decorated diagrams}
In the same way that crossingless tangles in $\mathbb D^2$ define objects in the Bar-Natan category, we also want to associate objects in a category of matrix factorizations to crossingless tangles. For this purpose, however, we will have to extend our notion of crossingless tangle to include some decorations.

\begin{definition}
	A \emph{decorated diagram} is a crossingless tangle $\Gamma$ in the 2-dimensional disk $\mathbb D^2$ together with the following decorations.
	\begin{itemize}
		\item A checkerboard coloring $\sigma$ of the complement $\mathbb D^2\setminus\Gamma$.
		\item A finite indexed set $A=(\alpha_i)_{i\in I\subseteq \N}$ of non-intersecting directed arcs in the interior of the disk such that $\alpha_i\cap\Gamma=\partial \alpha_i$ transversally and $\Gamma\cup \bigcup_i\alpha_i$ is connected. We call $A$ an \emph{arc system}.
		\item A possible basepoint in $\partial \Gamma$, in which case we say that $\Gamma$ has a \emph{distinguished end}. We may write $\Gamma^*$ to emphasize the presence of a distinguished end.
	\end{itemize}
\end{definition}

In general, the ends of arcs split the crossingless tangle into segments, which are referred to as \emph{edges}. There is one variable $x_e$ associated to each edge $e$. The \emph{edge ring} $R_\Gamma$ associated to the decorated diagram $\Gamma$ is defined to be the quotient of the polynomial ring generated by $G$ and all $x_e$ with linear relations given by an alternating sum of the edge variables as one travels around each arc. If there is a distinguished end, with associated variable $x_0$, we also impose the relation $x_0=0$. The subring of $R_\Gamma$ generated by variables associated to edges intersecting $\partial \mathbb D^2$ (\emph{boundary variables}) is called the \emph{boundary ring}, denoted by $R^\partial_\Gamma$. 

A checkerboard coloring $\sigma$ of $\Gamma$ determines an alternating coloring $\partial \sigma$ of the components of $\partial \mathbb D^2 \setminus \partial \Gamma$. If the boundary variables are labeled counterclockwise by $x_0,\dots,x_{2k-1}$, we write $\sigma_1$ for the coloring in which $\partial\sigma_1$ shades the interval between $x_0$ and $x_{2k-1}$, and $\sigma_0$ for the opposite coloring. When $\Gamma=\Gamma^*$ has a distinguished end, we always associate the variable $x_0$ to the basepoint.

Given a decorated diagram $\Gamma$, we associate a matrix factorization $\mathcal M (\Gamma)$ over the edge ring $R_\Gamma$. Let us begin by indicating the elementary 4-ended decorated diagrams and their associated factorizations.
\begin{align*}
	\Gamma_i=\Elementary[extra={\node at (-0.8,0.8) {$x_0$};
		\node at (-0.8,-0.8) {$x_1$};
		\node at (0.8,-0.8) {$x_2$};
		\node at (0.8,0.8) {$x_3$};
		\node at (0,0.15) {\tiny $\alpha_i$};},]
	\longmapsto \mathcal M(\Gamma_i)&=\kmf{R}{x_0-x_1 &(x_0-x_3)(x_1+x_3+G)}
	\\
	\Gamma_i=\Elementary[extra={\node at (-0.8,0.8) {$x_0$};
		\node at (-0.8,-0.8) {$x_1$};
		\node at (0.8,-0.8) {$x_2$};
		\node at (0.8,0.8) {$x_3$};
		\node at (0,0.15) {\tiny $\alpha_i$};}, sha1=1]
	\longmapsto \mathcal M(\Gamma_i)&=\kmf{R}{x_1-x_0 &(x_0-x_3)(x_1+x_3+G)}
\end{align*}

Here, the base ring is $\Z[G]$, and the elementary matrix factorizations are over the edge ring $R_i:=R_{\Gamma_i}=R^\partial_{\Gamma_i}=\mathbf \Z[G,x_0,x_1,x_2,x_3]/(x_0-x_1+x_2-x_3)$. The variables $x_j$ and $G$ live in quantum degree $q^{-2}$. The degree of the Koszul generator $\vartheta_i$ is always $\hslash^1 q^1$ for these matrix factorizations.

In these definitions, reversing the direction of the arc has the effect of flipping the sign of both terms inside the Koszul matrix factorization, while inverting the coloring changes the sign of all the variables $x_i$ and $G$. Thus, a coloring change only flips the sign on the linear term.
\begin{lemma}[cf. \cite{kwzmutation}]
	Let $\Gamma$ be a decorated diagram with arc system $A$, and let $A'\subseteq A$ be a minimal connecting subset. Write $X_\alpha$ for the linear term in the matrix factorization associated to the arc $\alpha$, $x_c$ for the edge variable of a chosen edge in a loop $c \in \mathrm{circ}(\Gamma)$, and $x_0,\dots,x_{2k-1}$ for the boundary variables. If $\Gamma$ has no distinguished end, then
	\[
		R_{\Gamma} \cong \Z[G,X_\alpha,x_c, x_{i} \mid \alpha\in A\setminus A', c\in\mathrm{circ}(\Gamma), 0\leq i \leq 2k-2 ].
	\]
	If $\Gamma=\Gamma^*$ has a distinguished end, and $x_0$ is the distinguished variable, then
	\[
		R_{\Gamma^*} \cong \Z[G,X_\alpha,x_c, x_{i} \mid \alpha\in A\setminus A', c\in\mathrm{circ}(\Gamma), 1\leq i \leq 2k-2 ].
	\]
\end{lemma}

This lemma shows, in particular, that $R_\Gamma$ is still a polynomial ring for any decorated diagram $\Gamma$. Now, let $\Gamma$ be an arbitrary decorated diagram on the disk. Each arc has a neighborhood looking like either of the two forms denoted $\Gamma_i$ above. We associate a matrix factorization
\[
	\mathcal M(\Gamma) =\bigotimes_i\mathcal M(\Gamma_i)\otimes_{ R_i} R_\Gamma
\]
over the corresponding edge ring $R_\Gamma$. The tensor product runs over all arcs $\alpha_i$ of $\Gamma$ and depends, in principle, on their labeling. However, the swap maps $\tau_{i,j}$ can be used to define canonical isomorphisms between the matrix factorizations associated to different choices of arc labeling.

For convenience, we often represent Koszul matrix factorizations associated to decorated diagrams by the diagram itself with some additional labels, namely
\begin{itemize}
	\item the shift in $q$-degree indicated by a number on the top right,
	\item the shift in $h$-degree indicated by a number on the bottom right,
	\item a possible right shift by 1 in $\hslash$-degree indicated by $\hslash$ on the bottom left,
	\item labels indicating edge variables, and
	\item arc labels.
\end{itemize}
Some of these labels can be omitted, following a few implicit conventions. When no arc labels are explicitly indicated, their labels are $\alpha_1,\dots,\alpha_r$, assigned lexicographically from top to bottom, then left to right, according to the position of the arrow tips. The boundary edge variables start from $x_0$ and go counter-clockwise. If $x_0$ is not indicated, it corresponds to the edge that intersects the boundary circle at the leftmost position in the top half. If an internal edge is not labeled, it is because its variable can be written in terms of the boundary edges and the labeled internal edges.

\begin{example}\label{ex:labelingscheme}
	We illustrate the labeling scheme with the following diagram. Take the minimal connecting set $A'=\{\alpha_1,\alpha_2\}$, so that, in principle, we can express $R_\Gamma\cong \Z[G,x_0,x_1,x_2,x_3,x_4,X_{\alpha_3}]$.

	\begin{align*}
		\Gamma&=\text{\LRL[q=3, h=0, hslash, extra={\small \node at (-0.55,0.65) {\tiny $x_0$};
				\node at (-0.55,-0.65) {\tiny $x_1$};
				\node at (0.3,-0.8) {\tiny $x_2$};
				\node at (0.7,-0.5) {\tiny $x_3$};
				\node at (0.7,0.5) {\tiny $x_4$};
				\node at (0.3,0.8) {\tiny $x_5$};
				\node at (-0.65,0) {\tiny $u$};
				\node at (0.15,0.35) {\tiny $v$};
				\node at (0.15,-0.35) {\tiny $w$};}, lab1={\tiny {$\alpha_1$}}, lab2={\tiny $\alpha_2$}, lab3={\tiny $\alpha_3$}]}
		=\text{\LRL[q=3, hslash, extra={\small \node at (-0.65,0) {\tiny $u$};}]}\qquad
		\begin{cases}
			x_5=x_0+x_2+x_4-x_1-x_3 \\
			u=X_{\alpha_3}+x_1\\
			v=x_2+x_4-x_3+X_{\alpha_3}\\
			w=x_2+X_{\alpha_3}
		\end{cases}
	\end{align*}
	The first diagram is fully labeled, while the second uses all possible implicit conventions. In practice, we favor the use of the indicated variable $u$ rather than $X_{\alpha_3}$ when doing algebraic manipulations with the explicit terms of the matrix factorization. In other words, we use the parametrization $R_\Gamma\cong \Z[G,x_0,x_1,x_2,x_3,x_4,u]$. In these variables, the associated matrix factorization is
	\begin{align*}
		\mathcal M (\Gamma)&=\kmf{R_\Gamma}{x_0-u & (x_1+x_3-x_2-x_4)(u+x_0+x_2+x_4-x_3-x_1+G) \\
			x_3-x_4& (u+x_2-x_1-x_3)(u+x_2+x_4-x_1+G) \\
			u-x_1& (x_1-x_2)(u+x_2+G)
		} q^3\hslash
	\end{align*}
	with potential $W=\sum_{i=0}^{5}(-1)^i\left(\frac13 x_i^3+\frac G 2x_i^2\right).$
\end{example}
To describe the potential for general $\Gamma$, assume that we have a counterclockwise labeling of boundary variables by $x_0,\dots, x_{2k-1}$, and that the coloring is $\sigma_\epsilon$. Then, the potential of $\mathcal M (\Gamma)$ is
\[
	W_\Gamma = \sum_{i=0}^{2k-1}{(-1)^{i+\epsilon}\left(\frac{1}{3}x^3_{i}+\frac{G}{2}x^2_i\right)}.
\]
In the based case, set $x_0=0$ in the formula. Although this expression has rational coefficients, it lies in $R^\partial$: writing $x_{2k-1}$ in terms of the other boundary variables makes it integral. Notice that all the contributions from internal edges cancel out, so that $W_\Gamma$ is completely determined by the choice of boundary variables and the coloring $\partial\sigma_\epsilon$.

\begin{remark}\label{remarkalteration}
	Alterations of a labeled decorated diagram $\Gamma$, such as coloring swap, arrow orientation reversal or relabelings of edges or arcs, have easily described algebraic counterparts. Let us think of $\mathcal M(\Gamma)$ as a polynomial algebra $\mathcal M(\Gamma)=\Z[G,x_i,\vartheta_j]_{i,j}$, where $G$ and $x_i$ are even and $\vartheta_j$ are odd, and similarly $\mathcal M(\Gamma')=\Z[G,x'_k,\vartheta'_l]_{k,l}$ for the altered diagram $\Gamma'$. Each alteration induces an algebra isomorphism $\mathcal M(\Gamma)\xrightarrow{\sim}\mathcal M(\Gamma')$:
	\begin{itemize}
		\item a coloring swap flips the sign of all even variables, including $G$,
		\item reversing orientation of the arc $\alpha_i$ flips the sign of $\vartheta_i$, and
		\item relabeling the edges or arcs maps the corresponding $x_i$ or $\vartheta_j$ to their respective $x_k'$ or $\vartheta_l'$.
	\end{itemize}
\end{remark}

\subsection{Simplifying the matrix factorization of a decorated diagram}\label{subsec:simplify}
The goal of this subsection is to realize delooping at the level of matrix factorizations. The general statement is the following.

\begin{proposition}\label{prop:mfdelooping}
	Given a decorated diagram $\Gamma$, there exists an explicit special deformation retract
	\[
		\vsdr{\bigoplus_{X\subseteq \mathrm{circ}(\Gamma)} \mathcal M(\Gamma')q^{|\mathrm{circ}(\Gamma)|-2|X|}\hslash^{|\mathrm{circ}(\Gamma)|}}{\mathcal M(\Gamma)}{}{}{}
	\]
	over $R_\Gamma^\partial$ such that each $\Gamma'$ is the crossingless matching $o(\Gamma)$ with any minimal set of connecting arcs.
\end{proposition}

In particular, the matrix factorization associated to a decorated diagram $\Gamma$ does not depend on the choice of arcs $A$, at least up to homotopy equivalence. In the process, we obtain useful expressions for the homotopy equivalences that realize loop removal and some basic moves on the arc system $A$. The various examples in this section cover the different types of decorated diagrams that appear in the analysis of Reidemeister invariance.

\subsubsection{Arc slides}
Let us establish some isomorphisms that allow us to move arcs past each other in a decorated diagram.

\begin{lemma}\label{lem:arrowslide}
	The following gives a pair of mutually inverse isomorphisms of matrix factorizations.
	\begin{equation}\label{iso:longslide}
		\begin{tikzcd}[column sep=7cm]
			\ThreeHills[emp2, dir1=1, extra={\node at (-0.8,0.2) {\tiny$x_0$}; }]
			\arrow[r, shift left=0.5ex, "\psi_{1,2}:= m_{1,2}(1)r_{1,2}(x_1-x_2)=1+(\theta_1\partial_2+(x_1-x_2)\theta_1\theta_2)"]
			&
			\ThreeHills[dir1=1, emp1, extra={\node at (-0.8,0.2) {\tiny$x_0$}; }]
			\arrow[l, shift left=0.5ex, "\psi^{-1}_{1,2}:= r_{2,1}(x_1-x_2)m_{1,2}(-1)=1-(\theta_1\partial_2+(x_1-x_2)\theta_1\theta_2)"]
		\end{tikzcd},
	\end{equation}
	If the direction of the first (resp. second) arc on each diagram is flipped, then the formulas for the left-to-right and right-to-left isomorphisms are exchanged. The formulas remain the same if both arc directions are flipped in each diagram.

	The same holds for
	\begin{equation}\label{iso:tipslide}
		\begin{tikzcd}[column sep=6cm]
			\ThreeHills[emp2, dir3=1, dir1=1, extra={\node at (-0.8,0.2) {\tiny$x_0$}; }]
			\arrow[r, shift left=0.5ex, "r_{2,1}(x_0+x_3+G)=1-(x_0+x_3+G)\theta_1\theta_2"]
			&
			\ThreeHills[emp1, dir3=1, dir1=1, extra={\node at (-0.8,0.2) {\tiny$x_0$}; }]
			\arrow[l, shift left=0.5ex, "r_{1,2}(x_0+x_3+G)=1+(x_0+x_3+G)\theta_1\theta_2"]
		\end{tikzcd}.
	\end{equation}
\end{lemma}

\subsubsection{Diagram simplifications}
Now, we want to express matrix factorizations associated to more complex decorated diagrams in terms of simpler ones. For us, a simple diagram is one without closed components (the underlying crossingless tangle has no loops) and with a minimal set of connecting arcs. When $\Gamma$ is not simple, the edge ring $R_\Gamma$ is larger than the boundary ring $R_\Gamma^\partial$. Example \ref{ex:labelingscheme}, for instance, exhibits a matrix factorization $\mathcal M (\Gamma)$ over $R_\Gamma=R_\Gamma^\partial[u]$, where $u$ is the variable labeling one of the internal edges. Since the potential itself is an element of the boundary ring $R_\Gamma^\partial$ (all the terms involving $u$ cancel out), it is possible to think of $\mathcal M (\Gamma)$ as a matrix factorization over $R_\Gamma^\partial$, albeit one of infinite rank. The main lemma below provides a recipe that can be specialized to obtain a homotopy equivalent finite rank matrix factorization in this type of situation.

First, let us fix some notation. Let $S$ be an integral domain, graded by quantum degree, and let $u$ be a variable in degree $q^{-2}$. Let $f(u)\in S[u]$ be a fixed polynomial in $u$ with unit leading coefficient. Let $M$ be a free $S[u]$-module with a fixed basis. For any $g\in M$ there are a well defined quotient $\Quo_f^u(g)\in M$ and residue $\Res_f^u(g)\in M$ satisfying
\[
	g= f\Quo_f^u(g)+ \Res_f^u(g),
\]
with $\deg_u(\Res_f^u(g))<\deg_u(f)$. Note that these operators are only $S$-linear. The following is our version of a result in \cite{khovanov2007virtual}, which we adapt to produce explicit deformation retract data.

\begin{lemma}\label{sdrlemma} Let $R$ be a polynomial ring over $\Z$, let $u$ be a variable in quantum degree $q^{-2}$, and let $f\in R[u]$ be a polynomial of quantum degree $q^{-2n}$ with leading term $\pm u^n$. Let $g\in R[u]$ have quantum degree $q^{-2m}$. Let $M$ be a matrix factorization over $R[u]$ with potential of the form $W=\overline W-fg\in R[u]$, such that $\overline W\in R$. Denote
	\begin{align*}
		M'
		&
		:= M/fM
		,
		\\
		\overline M
		&
		:=
		\kmf{R[u]}{f & g}\otimes_{R[u]} M
		,
	\end{align*}
	thought of as matrix factorizations over $R$ with potential $\overline W$.
	Then there is a special deformation retract
	\[
		\vsdr{M'}{\overline M}{\iota}{\rho}{\kappa}
	\]
	of the form
	\[
		\begin{tikzcd}[column sep=large, ]
			&
			M/fM
			\arrow[dd, shift right=0.6ex, "\iota_v", swap]
			\arrow[dd, leftarrow, shift left=0.6ex, "\rho"]
			\arrow[ddl, "\iota_d", bend right=20, swap]
			\\
			\\
			q^{m-n}\hslash M
			\arrow[r, "f", shift left=1ex, bend left=20, start anchor=east, end anchor=west]
			&
			M\arrow[l, "\kappa"]
			\arrow[l, "g", shift left=1ex, bend left=20, start anchor=west, end anchor=east].
		\end{tikzcd}
	\]

	The arrows are as follows:
	\begin{itemize}
		\item $\rho$ is the usual projection.
		\item $\kappa=-\Quo^u_f$.
		\item $\iota$ consists of a vertical component $\iota_v=\Res^u_f$ and a diagonal component $\iota_d=-\Quo^u_f\circ d_{M}\circ \iota_v$.
	\end{itemize}
\end{lemma}

\begin{proof}
	The identities $\kappa\iota=0$, $\rho\kappa=0$, $\rho\iota=1$ and $\kappa^2=0$ are straightforward.
	The identity
	\[
		d_{\overline M}\kappa+\kappa d_{\overline M} +1_{\overline M}=\iota\rho
	\]
	is verified as follows. Take $(x,z)\in q^{m-n}\hslash M \oplus M=\overline M$. We compute directly, keeping in mind that $\Res^u_f(z)=z-f\Quo^u_f(z)$ and $\Quo^u_f(fz)=z$:
	\begin{align*}
		(d_{\overline M}\kappa+\kappa d_{\overline M} +1_{\overline M})(x,z)
		&
		=(-\Quo^u_f(fx)+x -\Quo^u_f(d_M(z))-d_M(-\Quo^u_f(z)), -f\Quo^u_f(z) +z)
		\\
		&
		=(-\Quo^u_f(d_M(z))+d_M(\Quo^u_f(z)), \Res^u_f(z))
		,
	\end{align*}
	while
	\begin{align*}
		\iota\rho(x,z)
		&
		=(\iota_d([z]), \iota_v([z]))
		=(-\Quo^u_f(d_M(\Res^u_f(z))),\Res^u_f(z))
		\\
		&
		=(-\Quo^u_f(d_M(z-f\Quo^u_f(z))),\Res^u_f(z))
		\\
		&
		=(-\Quo^u_f(d_M(z))+d_M(\Quo^u_f(z)),\Res^u_f(z))
		.
	\end{align*}
	It is straightforward to verify that $\rho$ is a morphism. In the case of $\iota$, we have
	\begin{align*}
		\iota(d_{M'}([x]))
		&
		=( -\Quo^u_f(d_M^2(x))+d_M(\Quo^u_f(d_{M}(x))), \Res^u_f(d_{M}(x)))
		\\
		&
		=(d_M(\Quo^u_f(d_{M}(x))) -\Quo^u_f((\overline W -fg)x), \Res^u_f(d_{M}(x)))
		,
	\end{align*}
	and
	\begin{align*}
		d_{\overline M}(\iota([x]))
		&
		=d_{\overline M}(-\Quo^u_f(d_M(x))+d_M(\Quo^u_f(x)),\Res^u_f(x))
		\\
		&
		=
		(-d_M(-\Quo^u_f(d_M(x)))-d^2_M(\Quo^u_f(x))+g\Res^u_f(x),
		\\
		&
		\qquad\qquad d_M(\Res^u_f(x))-f\Quo^u_f(d_M(x))+fd_M(\Quo^u_f(x)))
		\\
		&
		=
		(d_M(\Quo^u_f(d_M(x)))-(\overline W -fg)\Quo^u_f(x)+g(x-f\Quo^u_f(x)),
		\\
		&
		\qquad\qquad d_M(\Res^u_f(x)+f\Quo^u_f(x))-f\Quo^u_f(d_M(x)))
		\\
		&
		=
		(d_M(\Quo^u_f(d_M(x)))-\overline W\Quo^u_f(x)+gx, d_M(x)-f\Quo^u_f(d_M(x)))
		\\
		&
		=(d_M(\Quo^u_f(d_{M}(x))) -\Quo^u_f((\overline W -fg)x), \Res^u_f(d_{M}(x)))
		.
	\end{align*}
\end{proof}

\begin{remark}\label{sdrremark}
	There is a slight variation in which $f$ appears in the second column, i.e.
	\[
		\overline M
		=
		\kmf{R[u]}{g & f}\otimes_{R[u]} M.
	\]
	In this case, we have $M'=q^{n-m}\hslash M/fM$, and the special deformation retract has the form
	\[
		\begin{tikzcd}[column sep=large, ]
			q^{n-m}\hslash M/fM
			\arrow[dd, shift right=0.6ex, "\iota_v", swap]
			\arrow[dd, leftarrow, shift left=0.6ex, "\rho"]
			\arrow[ddr, "\iota_d", bend left=20, swap]
			&
			\\
			\\
			q^{n-m}\hslash M
			\arrow[r, "g", shift left=1ex, bend left=20, start anchor=east, end anchor=west]
			\arrow[r, "\kappa", swap]
			&
			M
			\arrow[l, "f", shift left=1ex, bend left=20, start anchor=west, end anchor=east],
		\end{tikzcd}
	\]
	The formulas for $\kappa$ and $\iota_v$ are the same as in the lemma, while the formula for $\iota_d$ switches sign.
\end{remark}

When all the matrix factorizations involved are Koszul, we obtain very concrete simplification formulas. Let us first deal with the linear case.

\begin{lemma}[Simplification of Koszul factorizations by a linear term]\label{lemma:sdrkoszullinear}
	Let $R$ be a polynomial ring over $\Z$ and let $\kmf{R[u]} {\mathbf{\overline a} &\mathbf{\overline b}}$ be a Koszul matrix factorization over $R[u]$ with $a_j=\pm u-c$, $\deg_u(c)=0$. Let $\mathbf{\overline a '}$ and $\mathbf{\overline b '}$ be the sequences obtained from $\mathbf{\overline a}$ and $\mathbf{\overline b}$ by deleting $a_j$ and $b_j$, respectively, and then setting $u=\pm c$. Then
	we have a special deformation retract
	\[
		\vsdr
		{\kmf{R}
			{\left.\mathbf{\overline a'}\right.
				&
				\left.\mathbf{\overline b'}\right.}}
		{\kmf{R[u]}
			{\mathbf{\overline a}
				&\mathbf{\overline b}}}
		{\sum_{i=1}^r\left(\Quo_{a_j}^u(a_i)\partial_i + \Quo^u_{a_j}(b_i)\theta_i\right)\theta_j}
		{\ev_{u=\pm c}\partial_j\theta_j}
		{-\Quo_{a_j}^u \theta_j}
		[5cm]
	\]
\end{lemma}

\begin{proof}
	First assume $j=1$, in which case Lemma \ref{sdrlemma} applies directly. The formulas $\kappa=-\Quo_{a_j}\theta_1$ and $\rho=\ev_{u=\pm c}\partial_1\theta_1$ follow easily, while the verification of $\iota=\sum_{i=1}^r\left(\Quo^u_{a_1}(a_i)\partial_i + \Quo^u_{a_1}(b_i)\theta_i\right)\theta_1$ is left to the reader. The case $j\neq 1$ follows from the case $j=1$ by conjugating everywhere by swap isomorphisms $\tau_{1,j}$, which has the effect of interchanging $\theta_1\leftrightarrow \theta_j$ and $\partial_1\leftrightarrow \partial_j$.
\end{proof}

In more general cases, we want an expression $M/fM=\bigoplus_{i=0}^{k}M^{(i)}$ in which $M^{(i)}$ are simpler factorizations, e.g. coming from simpler diagrams. The lemma below gives such decompositions for Koszul factorizations in which a suitable higher order monic polynomial appears.

We make use of the following additional notation. Given $f\in R[u]$ with leading term $\pm u^n$, and $\phi\in \Z[u,x_i,\vartheta_j]_{ i\in I, j\in J}$ (variables $x_i$ and $u$ are even, $\vartheta_j$ are odd), define $P_{f}^{u,k}(\phi)$ by the identity
\[
	\Res_{f}^u(\phi) = \sum_{k=0}^{n-1} u^k P_{f}^{u,k}(\phi).
\]
In other words, $P_{f}^{u,k}$ picks up the coefficient of $u^k$ in $\Res_{f}^u(\phi)$.

\begin{lemma}[Simplification of Koszul factorizations by a higher order polynomial]\label{lemma:sdrkoszul} Let $R$ be a polynomial ring over $\Z$ and let $\kmf{R[u]} {\mathbf{\overline a} &\mathbf{\overline b}}$ be a Koszul matrix factorization over $R[u]$ such that $b_j$ has leading term $\pm u^n$ with respect to $u$, and satisfies $\deg_u(\Res^u_{b_j}(a_i))=\deg_u(\Res^u_{b_j}(b_i))=0$ for all $i\neq j$. Let $\mathbf{\overline a '}$ and $\mathbf{\overline b '}$ be the sequences obtained from $\mathbf{\overline a}$ and $\mathbf{\overline b}$ by deleting $a_j$ and $b_j$ and then replacing $a_i$ and $b_i$, respectively, by $\Res^u_{b_j}(a_i)$ and $\Res^u_{b_j}(b_i)$ for $i\neq j$. Then we have a special deformation retract
	\[
		\vsdr{ \displaystyle\bigoplus_{l=0}^{n-1}\kmf{R}{\left.-\mathbf{\overline a'}\right.
				& \left.-\mathbf{\overline b'}\right.}q^{N-2l}\hslash}
		{\kmf{R[u]}{\mathbf{\overline a} &\mathbf{\overline b}}}
		{\left(u^l\sum_{i=1}^r\left(\Quo_{b_j}^u(a_i)\partial_i + \Quo_{b_j}^u(b_i)\theta_i\right)\right)_l}
		{\left(P^{u,l}_{b_j}\partial_j\right)_l}
		{-\Quo_{b_j}^u \partial_j}
		[5.5cm],
	\]
	where $N=q(\vartheta_j)$.
\end{lemma}

\begin{proof}
	Similarly to the linear case, we apply Remark \ref{sdrremark} to obtain
	\[
		\vsdr{q^{N}\hslash M/fM}
		{M}
		{\sum_{i=1}^r\left(\Quo_{b_j}^u(a_i)\partial_i + \Quo_{b_j}^u(b_i)\theta_i\right)}
		{\partial_j}
		{-\Quo_{b_j}^u \partial_j}
		[5cm]
	\]
	for $M:=\kmf{R[u]}{\mathbf{\overline a} &\mathbf{\overline b}}$. The map
	\[
		\displaystyle\bigoplus_{l=0}^{n-1}q^{-2l}\kmf{R}{\mathbf{\overline a'} & \mathbf{\overline b'}}
		\xrightarrow{(u^l)_l}
		M/fM
	\]
	is an isomorphism of $R$-modules; its inverse is given by $(P_{b_j}^{u,l})_l$, and the condition $\deg_u(\Res^u_{b_j}(a_i))=\deg_u(\Res^u_{b_j}(b_i))=0$ for all $i\neq j$ guarantees that it is in fact a morphism of matrix factorizations. The special deformation retract in the statement of the lemma is obtained by composing.
\end{proof}

\subsubsection{Arc removal}
The simplification by linear polynomials in Lemma \ref{lemma:sdrkoszullinear} can be directly applied in order to remove arcs that bound an interior region of a decorated diagram. Let $\Gamma$ be a decorated diagram with its $j$-th arc $\alpha_j$ bounding an interior region, i.e. $\Gamma\cup \bigcup_{i\neq j} \alpha_i$ remains connected. Let $\Gamma'$ be the decorated diagram obtained from $\Gamma$ by removal of $\alpha_j$, so that $A_{\Gamma'}=(\alpha_1,\dots,\alpha_{j-1},\alpha_{j+1},\dots,\alpha_r)$. In particular, $R_{\Gamma'}[X_{\alpha_j}]\cong R_\Gamma$. In this situation, there is a special deformation retract of matrix factorizations over $R_{\Gamma'}$
\begin{equation}\label{sdr:interiorfaceremoval}
	\vsdr{\mathcal M(\Gamma')}{\mathcal M(\Gamma)}
	{\Quo^u_{X_{\alpha_j}}\sum_{i=1}^r(a_i\partial_i+b_i\theta_i)\theta_j}
	{\ev_{X_{\alpha_j}=0}\partial_j\theta_j}
	{-\Quo^u_{X_{\alpha_j}}\theta_j}
	[3.5cm].
\end{equation}
Here $u$ is an internal edge variable such that $R_\Gamma \cong R_{\Gamma'}[u]$ and $X_{\alpha_j}=a_j(u)$ is linear in $u$.

Let us do a few explicit examples that are useful in the proof of Reidemeister invariance.

\begin{example} \label{ex:R1arcremoval}
	A first example is the removal of an arc connecting a component to itself
	\[
		\ThinEdgeAndCircle[dir1=1,extra={\node at (0.2,0.4) {\tiny $u$};}]= \kmf{R^\partial[u]}{u-x_0 & 0}
	\]
	The boundary ring in this case is simply $R^\partial=\Z[G,x_0]$. The special deformation retract is
	\begin{equation}\label{eq:sdrthickloop}
		\vsdr{\SingleEdge}
		{[1cm] \ThinEdgeAndCircle[dir1=1]}
		{\partial_1\theta_1}
		{ \ev_{u=x_0}\partial_1\theta_1}
		{-\Quo^u_{u-x_0}\theta_1}[2cm]
	\end{equation}
\end{example}

\begin{example} \label{ex:R2arcremoval}
	Similarly, take
	\[
		\Square[dir1=1, extra={\node at (0,0.4) {\tiny $u$};}]
		=
		\kmf{R^\partial[u]}{u-x_0 & (x_1-x_0)(u+x_1+G) \\
			x_0-x_1 & (u+x_1-x_0-x_2)(x_2+u+G)}
	\]
	with boundary ring $R^\partial=\Z[G,x_0,x_1,x_2]$. The formula gives
	\begin{equation}\label{eq:sdrthickthin}
		\vsdr{\scalebox{0.8}{\Elementary}}
		{\scalebox{0.8}{\Square[dir1=1]}}
		{(\partial_1 + (u+x_1+G)\theta_2)\theta_1}{\ev_{u=x_0}\partial_1\theta_1}{-\Quo^u_{u-x_0}\theta_1}[4cm]
	\end{equation}
\end{example}

\begin{example} \label{ex:R3arcremoval}
	A third example consists of the removal of $\alpha_1$ from the diagram
	\begin{align*}
		{\RLR[extra={\node at (0.65,0) {\tiny $u$};}]}
		& = K_{R^\partial[u]}
		\begin{pmatrix}
			u-x_4
			&
			(x_0+x_2-x_1-x_3)(u+x_1+x_3-x_0-x_2-x_4+G)
			\\
			x_0-x_1
			&
			(x_1+x_3-x_2-u)(u+x_0+x_2-x_3+G)
			\\
			x_3-u
			&
			(x_2-x_3)(u+x_2+G)
		\end{pmatrix}.
	\end{align*}
	In this case, the boundary ring is $R^\partial=\Z[G,x_0,x_1,x_2,x_3,x_4]$. The resulting special deformation retract is
	\[
		\begin{tikzcd}[column sep=7cm, ampersand replacement=\&]
			{\scalebox{0.7}{\LR[lab1={\tiny $\alpha_2$}, lab2={\tiny $\alpha_3$}]}}
			\arrow[r, shift left=0.5ex, color=violet, "(\partial_1 -\partial_3-(u+x_0+2x_2+x_4-x_1-2x_3+G)\theta_2+(x_2-x_3)\theta_3)\theta_1"]
			\arrow[r, leftarrow, shift right=0.5ex, swap, color=violet, "\ev_{u=x_4}\partial_1\theta_1"]
			\&
			{\scalebox{0.7}{\RLR}}
			\arrow[loop right, distance=1em, in=15, out=-15, ,swap, color=violet, "-\Quo^u_{u-x_4}\theta_1"]
		\end{tikzcd}
	\]
\end{example}

\subsubsection{Circle removal (delooping)}

Given a decorated diagram $\Gamma$ with an arc $\alpha_j$ connecting a loop of the underlying crossingless tangle to some other component, we have a special deformation retract of the form

\begin{equation}\label{sdr:delooping}
	\vsdr{ \mathcal M(\Gamma')q\hslash \oplus \mathcal M(\Gamma')q^{-1}\hslash}
	{\mathcal M(\Gamma)}
	{\begin{pmatrix}1 & u\end{pmatrix}\Quo_{b_j}^u\sum_{i=1}^r b_i\theta_i}
	{\begin{pmatrix} P_{b_j}^{u,0} \\ P_{b_j}^{u,1} \end{pmatrix}\partial_j }
	{-\Quo_{b_j}^u \partial_j}
	[5.5cm],
\end{equation}
as long as the condition $\deg_u(\Res_{b_j}^u(a_i))=\deg_u(\Res_{b_j}^u(b_i))=0$ is satisfied for all $i\neq j$. Here, the decorated diagram $\Gamma'$ is obtained from $\Gamma$ by removing the loop via surgery along the arc $\alpha_j$, and then reversing the orientation of all remaining arcs. Since all the $a_i$ are linear and all the $b_i$ are quadratic, the condition is equivalent to asking $a_i$ to be constant in $u$ and $P_{b_j}^{u,1}(b_i)=0$ for all $i\neq j$.

\begin{example} \label{ex:R1delooping}
	The simplest form of delooping consists of removing the loop in the local picture
	\[
		\ThinEdgeAndCircle[extra={\node at (0.2,0.4) {\tiny $u$};}]
		=
		\kmf{R^\partial[u]}{0 & (x_0-u)(u+x_0+G)}
		,
	\]
	where the boundary ring is $R^\partial=\Z[G,x_0]$. By simplifying with respect to $b_1(u)=(x_0-u)(u+x_0+G)$, we are left with the special deformation retract
	\begin{equation}\label{sdr:simpleloopremoval}
		\vsdr
		{\scalebox{0.8}{\SingleEdge[q=1, hslash]}\oplus \scalebox{0.8}{\SingleEdge[q=-1, hslash]}}
		{[4cm] \ThinEdgeAndCircle}
		{\begin{pmatrix} 1 &u \end{pmatrix}\theta_1 }
		{\begin{pmatrix} P_{b_1}^{u,0} \\ P_{b_1}^{u,1} \end{pmatrix}\partial_1 }
		{-\Quo^{u}_{b_1(u)}\partial_1}
	\end{equation}
\end{example}

\begin{example}\label{ex:thinthin}
	Loops may appear in more complicated shapes, and there are different ways to deal with them. The diagram
	\[
		\Square[extra={\node at (0,0.42) {\tiny $u$};}] =
		\kmf{R^\partial[u]}{ x_0-x_1 & (x_0-u)(u+x_1+G) \\
			x_0-x_1 & (u+x_1-x_0-x_2)(u+x_2+G) }
	\]
	can be simplified by reducing with respect to the quadratic polynomial $b_2(u)=(u+x_1-x_0-x_2)(u+x_2+G)$.
	\begin{equation} \label{sdr:thinthin}
		\vsdr{\scalebox{0.8}{\Elementary[q=1, hslash, dir1=2 ]} \oplus \scalebox{0.8}{\Elementary[q=-1, hslash, dir1=2]}}
		{\scalebox{0.7}{\Square}}
		{\begin{pmatrix}1 & u \end{pmatrix}(-\theta_1+\theta_2)}
		{\begin{pmatrix} P^{u,0}_{b_2(u)} \\ P^{u,1}_{b_2(u)} \end{pmatrix}\partial_2}
		{-\Quo^{u}_{b_2(u)}\partial_2}[5cm]
	\end{equation}
\end{example}

\begin{remark}
	Let $\Gamma$ be a decorated diagram, and let $\Gamma'$ be an alteration as in Remark \ref{remarkalteration}. If we have an explicit special deformation retract simplifying the matrix factorization associated to a diagram $\Gamma$, then we can also obtain formulas for a simplification of $\Gamma'$. Indeed, the algebraic transformation rules extend to terms appearing in maps such as $\partial_j$, $P_f^{u,l}$ and $\Quo_f^u$: arc alterations affect $\partial_j$ the same way they do $\theta_j$, while coloring changes or edge relabelings alter $P_f^{u,l}$ and $\Quo_f^u$ by applying the corresponding variable transformations to $u$ and $f$.
\end{remark}

\begin{example}\label{ex:thinthintransformed}
	To illustrate the preceding remark, here is a version of Example~\ref{ex:thinthin} after inverting the coloring and the orientation of the first arc:

	\begin{equation}
		\vsdr{\scalebox{0.8}{\Elementary[sha1=1,q=1, hslash ]} \oplus \scalebox{0.8}{\Elementary[sha1=1,q=-1, hslash]}}
		{\adjustbox{scale=0.7, rotate=90, raise=-0.8cm}{\Bigon[dir1=1,dir2=3]}}
		{\begin{pmatrix}1 & -u \end{pmatrix}(\theta_1+\theta_2)}
		{\begin{pmatrix} P^{u,0}_{b_2(u)} \\ -P^{u,1}_{b_2(u)} \end{pmatrix}\partial_2}
		{-\Quo^{u}_{b_2(u)}\partial_2}[5cm].
	\end{equation}
	Here, we are flipping the signs of $\theta_1$, $\partial_1$, $G$ and all the edge variables. The sign of the original $b_2(u)=(u+x_1-x_0-x_2)(u+x_2+G) $ did not change, and we used the fact that $P^{-u,0}_{b_2(u)}=P^{u,0}_{b_2(u)}$, $P^{-u,1}_{b_2(u)}=-P^{u,1}_{b_2(u)}$ and $\Quo^{-u}_{b_2(u)}=\Quo^{u}_{b_2(u)}$. A different way of obtaining such a special deformation retract is by applying \eqref{sdr:delooping} directly. These two possibilities differ by $u\mapsto-u$.
\end{example}

\begin{example}\label{ex:circleremovalreid3}
	In the matrix factorization
	\[
		\RLR[extra={\node at (0.65,0) {\tiny $u$};}, dir1=1, dir2=2, dir3=1] =
		\kmf{R^\partial[u]}{x_0+x_2-x_1-x_3 & (u-x_4)(u+x_0+x_2+x_4-x_1-x_3+G) \\
			x_1-x_0 & (u+x_2-x_1-x_3)(u+x_0+x_2-x_3+G) \\
			x_2-x_3 & (x_3-u)(u+x_2+G)}
		,
	\]
	there is no choice of arc $\alpha_j$ so that the terms $b_i$ in the second column fulfill the condition $P_{b_j}^{u,1}(b_i)=0$ for all $i\neq j$. Instead, we simplify this matrix factorization by first using an arc slide isomorphism
	\[
		\begin{tikzcd}[column sep=2cm]
			\RLR[dir1=1, dir2=2, dir3=1]
			\arrow[r, shift left=0.5ex, "\psi_{2,3}=1+(\theta_2\partial_3+(x_3-u)\theta_2\theta_3)"]
			&[1.7cm]
			\RLR[emp2, dir1=1, dir3=1, extra={\draw[dotted, dash pattern=on 1pt off 1pt, {<[scale=0.8]}-]
				(-0.65,-0.2)
				.. controls (-0.45,-0.3) and (-0.3,-0.7)
				.. (-0.1,-0.8);}]
			\arrow[l, shift left=0.5ex, "\psi_{2,3}^{-1}=1-(\theta_2\partial_3+(x_3-u)\theta_2\theta_3)"]
		\end{tikzcd}
	\]
	to obtain the right-hand side diagram, to which \eqref{sdr:delooping} applies similarly to \eqref{sdr:thinthin}:
	\[
		\vsdr{\scalebox{0.8}{\RL[dir1=3, hslash, q=1]}\oplus\scalebox{0.8}{\RL[dir1=3, hslash, q=-1]}}
		{
			\RLR[emp2, dir1=1, dir3=1, extra={\draw[dotted, dash pattern=on 1pt off 1pt, {<[scale=0.8]}-]
				(-0.65,-0.2)
				.. controls (-0.45,-0.3) and (-0.3,-0.7)
				.. (-0.1,-0.8);}]
		}
		{\begin{pmatrix} 1 &u \end{pmatrix}(-\theta_1+\theta_3) }
		{\begin{pmatrix} P_{(x_3-u)(u+x_0+x_2-x_1+G)}^{u,0} \\ P_{(x_3-u)(u+x_0+x_2-x_1+G)}^{u,1} \end{pmatrix}\partial_3 }
		{\Quo^{u}_{(u-x_3)(u+x_0+x_2-x_1+G)}\partial_3}
		[4cm]
	\]
	After suitable compositions and simplifications, we obtain
	\[
		\vsdr{\scalebox{0.8}{\RL[dir1=3, hslash, q=1]}\oplus\scalebox{0.8}{\RL[dir1=3, hslash, q=-1]}}
		{\RLR[dir1=1, dir2=2, dir3=1]}
		{\begin{pmatrix} 1 &u \end{pmatrix}(-\theta_1-\theta_2+\theta_3+(x_3-u)\theta_1\theta_2\theta_3) }
		{\begin{pmatrix} P_{(x_3-u)(u+x_0+x_2-x_1+G)}^{u,0} \\ P_{(x_3-u)(u+x_0+x_2-x_1+G)}^{u,1} \end{pmatrix}(\partial_3+(u-x_3)\theta_2\partial_3\theta_3) }
		{\kappa}
		[7cm],
	\]
	where $\kappa=\psi_{2,3}^{-1}\Quo^{u}_{(u-x_3)(u+x_0+x_2-x_1+G)}\partial_3 \psi_{2,3}$.
\end{example}

We can leverage these ideas to obtain explicit special deformation retracts delooping at the level of matrix factorizations for general decorated diagrams, as stated in Proposition \ref{prop:mfdelooping}.

\begin{proof}[Proof of Proposition \ref{prop:mfdelooping}]
	First, do a sequence of arc removals to $\Gamma$ to obtain a deformation retract into a decorated diagram $\Gamma''$ with the same underlying 1-manifold, but with a minimal set of arcs. Now, given any loop in $\Gamma''$ not enclosing any other loops, it is always possible to do a sequence of arc slides to put it in the form appearing in Example \ref{ex:R1delooping}, and then apply the special deformation retract appearing there. This will create two summands with identical decorated diagrams with shifts $q\hslash$ and $q^{-1}\hslash$, respectively. The desired result is obtained by doing this process inductively, one loop at a time, with the number of summands doubling at each step. Observe that arc slide isomorphisms, arc removal deformation retracts and delooping special deformation retracts all admit explicit formulas, according to Lemma \ref{lem:arrowslide} and formulas \eqref{sdr:interiorfaceremoval} and \eqref{sdr:delooping}, respectively.
\end{proof}

\section{The extended tangle invariant}\label{sec:multif}

Now that we have established how to associate matrix factorizations to decorated crossingless tangles, we proceed to associate a matrix multifactorization $\mathcal M(D)$, a sort of cube-of-resolutions complex in $MF(R,W)$, to a suitably decorated tangle diagram $D$, in analogy to the definition of the Bar-Natan bracket $[\![D]\!]$. At the end of this section we focus on the 4-ended case, where this analogy is taken further to describe how $\mathcal M(D)$ is used to obtain an extension of the Bar-Natan invariant $\De(T)\in Ch(\mathcal B(E^4_*))$ into a bifiltered twisted complex $\De^\infty (T)$ over $\mathcal B^\infty_{ds}$.

\begin{definition}
	Let $R$ be a polynomial ring over $\Z$ graded by quantum degree $q$ and homological degree $h$, so that all polynomial generators live in degree $q^{-2}h^0$. A \emph{matrix multifactorization} over $R$ with potential $W\in R$ is a graded free $R$-module $M$ equipped with an additional $\Z/2$-grading $\hslash$ and with an $R$-module homomorphism $d_M$ of the form $d_M=\sum_{i=0}^\infty d_i$, with $d_i$ homogeneous of degree $q^{3(i-1)}h^i\hslash$, satisfying $d_M^2=W\cdot\id$. We require $M$ to be bounded in $h$-degree. We call $d_0$ the \emph{vertical component} of $d_M$, and we define the \emph{underlying} or \emph{vertical matrix factorization} of $M$ to be $\underline M=(M,d_0)$.
\end{definition}

If we overlook momentarily the lack of an integral $A_\infty$-grading on $MF(R,W)$, then a matrix multifactorization $M$ can be thought of, in the language of Subsection \ref{sec:prelimainfinity}, as a bifiltered twisted complex over $MF(R,W)$, where $h$ is the homological degree in the sense of twisted complexes. From this perspective, the vertical matrix factorization $\underline M$ decomposes into a direct sum of components concentrated in fixed homological degree, and the remaining components of $d_M$ form the twisted differential $\delta=d_1+d_2+\cdots$.

The notions of $n$-(homo)morphism, $n$-homotopy and $n$-homotopy equivalence are defined as for twisted complexes. Additionally, we define a \emph{vertical isomorphism} to be an invertible $0$-homomorphism that is homogeneous of $h$-degree $h^0$, i.e. only the vertical component is possibly nonzero. We reserve the notation $M\cong M'$ to indicate vertically isomorphic multifactorizations.

\begin{definition}
	A \emph{decorated diagram} for an oriented tangle $T$ is an oriented tangle diagram $D$ in the 2-dimensional disk $\mathbb D^2$ together with
	\begin{itemize}
		\item a checkerboard coloring of the complement $\mathbb D^2\setminus D$,
		\item a finite indexed set $(\alpha_j)_{j\in I_{a}\subset \N}$ of non-intersecting directed arcs in the interior of the disk and away from the crossings, such that $\alpha_j\cap D=\partial \alpha_j$ transversally and $D\cup \bigcup_j\alpha_j$ is connected,
		\item a finite set $I_{c}\subseteq \N\setminus I_{a}$ indexing the crossings in $D$, and
		\item a possible choice of distinguished end in $\partial D$, in which case we may write $D^*$ to indicate the presence of such end. An asterisk $*$ may be placed near the distinguished end to graphically indicate it.
	\end{itemize}
\end{definition}

We denote the total indexing set by $I:=I_a\cup I_c$. Given a decorated tangle diagram, we associate an elementary matrix multifactorization $\mathcal M(D_i)$ to each local picture $D_i$ around a crossing with label $i\in I_c$:

\[
	D_i=\Elementary[dir1=1, cr1, extra={\node at (0.25,0) {\tiny $i$};}]
	\longmapsto
	\mathcal M (D_i)=\left[
	\begin{tikzcd}[column sep=2cm]
		\scalebox{0.8}{\Elementary[q=1, hslash, extra={\node at (0,0.1) {\tiny $\alpha_i$};} ] }
		\arrow[r, "\partial_i+(x_1+x_3+G)\theta_i"]
		&
		\scalebox{0.8}{\Elementary[dir1=1, q=2, h=1, hslash, extra={\node at (0.2,-0.05) {\tiny $\alpha_i$};} ]}
	\end{tikzcd}
	\right]_,
\]

\[
	D_i=\Elementary[dir1=1,cr1=-1, extra={\node at (0.25,0) {\tiny $i$};}]
	\longmapsto
	\mathcal M (D_i)=\left[
	\begin{tikzcd}[column sep=2cm]
		\scalebox{0.8}{\Elementary[q=-2, h=-1, dir1=1, hslash, extra={\node at (0.2,-0.05) {\tiny $\alpha_i$};}] }
		\arrow[r, "\partial_i+(x_1+x_3+G)\theta_i"]
		&
		\scalebox{0.8}{\Elementary[q=-1, hslash, extra={\node at (0,0.1) {\tiny $\alpha_i$};} ]}
	\end{tikzcd}
	\right]_.
\]
One may also find the inverted coloring at a crossing. As before, a coloring inversion entails a sign flip on the even variables. The associated multifactorizations are
\[
	D_i=\Elementary[dir1=1, cr1, sha1=1, extra={\node at (0.25,0) {\tiny $i$};}]
	\longmapsto
	\mathcal M (D_i)=\left[
	\begin{tikzcd}[column sep=2cm]
		\scalebox{0.8}{\Elementary[q=1, sha1=1, hslash, extra={\node at (0,0.1) {\tiny $\alpha_i$};} ] }
		\arrow[r, "\partial_i-(x_1+x_3+G)\theta_i"]
		&
		\scalebox{0.8}{\Elementary[dir1=1, q=2, h=1, sha1=1, hslash, extra={\node at (0.2,-0.05) {\tiny $\alpha_i$};}]}
	\end{tikzcd}
	\right]_,
\]

\[
	D_i=\Elementary[dir1=1,cr1=-1, sha1=1, extra={\node at (0.25,0) {\tiny $i$};}]
	\longmapsto
	\mathcal M (D_i)=\left[
	\begin{tikzcd}[column sep=2cm]
		\scalebox{0.8}{\Elementary[q=-2, h=-1, dir1=1, sha1=1, hslash, extra={\node at (0.2,-0.05) {\tiny $\alpha_i$};}] }
		\arrow[r, "\partial_i-(x_1+x_3+G)\theta_i"]
		& \scalebox{0.8}{\Elementary[q=-1, sha1=1,hslash, extra={\node at (0,0.1) {\tiny $\alpha_i$};} ]}
	\end{tikzcd}
	\right]_.
\]
In all four cases, the horizontal arrow is called a \emph{saddle map}, and will be denoted by $S_i$.

Now, given a decorated tangle diagram $D$, we associate to it a matrix multifactorization $\mathcal M(D)$ over $R^\partial_D$ given by the tensor product
\begin{equation}\label{eq:defmultifact}
	\mathcal M (D) := \bigotimes_{i\in I} \left(\mathcal M (D_i)\otimes_{R_{i}} R_D\right),
\end{equation}
where $R_i$, $R_D$ and $R^\partial_D$ are defined as for decorated planar diagrams. Notice that $R_D= R_{\Gamma}$ for any resolution $\Gamma$ of $D$. Setting $E= D\cap\partial \mathbb D^2$, the coloring $\sigma$ of $D$ induces an alternating coloring $\partial\sigma$ of the components in $\partial \mathbb D^2 \setminus E$.

\begin{remark}
	Let us compare our conventions to the multifactorizations $M(-)$ associated to tangle diagrams in \cite{kwzmutation}. In \cite{kwzmutation}, tangles admit resolutions into webs, in which arcs have thickness rather than orientation, and edges are oriented and signed in compatible ways around each arc. The edge sign assignment should be thought of as a counterpart to our checkerboard colorings. Here is the correspondence of the matrix factorizations associated to elementary local pictures:

	\[
		\mathcal M\left(\scalebox{0.8}{$\Elementary$} \right)= M\left(\scalebox{0.8}{$\Elementary[thin1, extra={\node at (-0.28,0.58) {\tiny +};}]$}\right) = M\left(\scalebox{0.8}{$\Elementary[thick1, extra={\node at (0.3,0.56) {\tiny --};}]$}\right)\hslash,
	\]

	\[
		\mathcal M\left(\scalebox{0.8}{$\Elementary[dir1=1]$} \right)= M\left(\scalebox{0.8}{$\Elementary[thin1, dir1=1, extra={\node at (-0.3,-0.56) {\tiny --};}]$}\right) = M\left(\scalebox{0.8}{$\Elementary[thick1, dir1=1, extra={\node at (-0.3,0.56) {\tiny +};}]$}\right)\hslash.
	\]
	The elementary multifactorizations associated to positive crossings satisfy

	\[
		\mathcal M\left(\scalebox{0.8}{$\Elementary[cr1, dir1=1]$} \right)= M\left(\scalebox{0.8}{$\Elementary[cr1=10, dir1=1, sha1=1, extra={\node at (-0.28,0.58) {\tiny +};}]$}\right)_,
		\qquad
		\mathcal M\left(\scalebox{0.8}{$\Elementary[cr1, sha1=1, dir1=1]$} \right)\cong M\left(\scalebox{0.8}{$\Elementary[cr1=10, dir1=1, extra={\node at (-0.28,0.58) {\tiny --};}]$}\right)_,
	\]
	and similarly for negative crossings. As a consequence of these identifications, if $D^*$ is a connected diagram for a tangle $T$ with boundary $E_*^4$, and $D^*_{\sigma_0}$ is the decorated tangle diagram obtained by coloring by $\sigma_0$, it follows that there is a vertical isomorphism $\mathcal M (D^*_{\sigma_0})\cong M(D^*)$ of matrix multifactorizations.

\end{remark}

Now, recall that, in the Bar-Natan setting, we defined the subcategory $\mathcal B(E)$ whose objects are representatives of each isotopy class of crossingless matchings. The following definition is the counterpart to $\mathcal B(E)$ in the setting of matrix factorizations.

\begin{definition}
	Let $E\subseteq \partial \mathbb D^2$ be finite and possibly based, and let $\partial\sigma$ be an alternating coloring of $\partial \mathbb D^2\setminus E$. Write $R^\partial$ and $W$ for the associated boundary ring and potential, respectively. For each $\Gamma\in Ob(\mathcal B(E))$, consider the checkerboard coloring $\sigma(\Gamma)$ compatible with $\partial\sigma$, and fix a minimal set of arcs, thus obtaining a decorated diagram $\widetilde \Gamma$. Denote the number of shaded components in $\mathbb D^2\setminus \Gamma$ by $|\sigma(\Gamma)|$. We define $\mathcal A(E,\partial\sigma)$ to be the full subcategory of $MF(R^\partial, W)$ generated by the matrix factorizations of the form $\mathcal M(\widetilde \Gamma)\hslash^{|\sigma(\Gamma)|}$, with $\Gamma$ running through all objects in $\mathcal B(E)$.
\end{definition}

The definition does not depend on the choice of arc system, up to isomorphism of $\Z/2$-graded dg-algebras: different arc systems for the same underlying crossingless tangle are related by sequences of slide isomorphisms. The algebra $\mathcal A(E,\partial\sigma)$ is a more manageable setting than $MF(R^\partial,W)$ to define integral lifts of the $\hslash$-grading, thus getting an actual dg-algebra structure. In the case of $E=E_*^4$ with the standard coloring $\partial\sigma_0$, we have the generating idempotents
\[
	\mathcal M_\circ:=\Elementary[extra={\node at (-0.62,0.6) {*};}] \qquad \mbox{and} \qquad \mathcal M_\bullet :=\Elementary[dir1=1, hslash, extra={\node at (-0.62,0.6) {*};}]
\]
for $\mathcal A (E^4_*,\partial \sigma_0)$. These are matrix factorizations over the ring $\Z[x,y,z]$ with potential $W=xyz$, where $x=x_1$, $y=x_3$ and $z=x_1+x_3+G$. The integral $A_\infty$-grading on $\mathcal A (E^4_*,\partial \sigma_0)$ we work with is defined as follows. Declare the generators of the coefficient ring to have degrees $|y|=2$ and $|x|=|z|=0$, and the elements in the bases $\{\vartheta_\circ,1_\circ\}\subseteq \mathcal M_\circ$ and $\{\vartheta_\bullet,1_\bullet\}\subseteq \mathcal M_\bullet$ to have degrees $|\vartheta_\circ|=|1_\bullet|=-1$ and $|1_\circ|=|\vartheta_\bullet|=0$.

Once $\mathcal A(E,\partial \sigma)$ is endowed with an $A_\infty$-grading lifting $\hslash$, any multifactorization $(M,d_0+d_1+\cdots)$ satisfying $\underline M\cong \bigoplus_{i\in I} \mulg^{v_i}\hslash^{\epsilon_i} M_i$, with $M_i\in\mathcal A(E,\partial \sigma)$, can be seen as a bifiltered twisted complex in $Tw^h\mathcal A(E,\partial \sigma)$. The twisted differential is $\delta=d_1+d_2+\cdots$. Notions such as $n$-homomorphism, $n$-homotopy and $n$-homotopy equivalence coincide across the two perspectives. With this in mind, we want to deloop the multifactorization $\mathcal M(D)$ associated to a tangle diagram $D$ by replacing it by a 0-homotopy equivalent bifiltered twisted complex over $\mathcal A(E,\partial \sigma)$. The equivalence in question is of a special form.

\begin{definition}
	Let $(M,d)$ and $(M',d')$ be matrix multifactorizations. A \emph{special} $n$\emph{-deformation retract} from $M$ to $M'$ consists of 0-homomorphisms $\rho:M\to M'$ and $\iota:M'\to M$ such that $\rho\iota=1$, together with an $n$-homotopy $\kappa:M\to M$ between $\iota\rho$ and $1$ such that $\kappa\iota=0$, $\rho\kappa=0$ and $\kappa^2=0$. We often represent this data by a triple $(\iota,\rho,\kappa)$ or, more suggestively, by a diagram
	\[
		\sdr{(M',d')}{(M,d)}{\iota}{\rho}{\kappa}
		.
	\]
\end{definition}

\begin{proposition}\label{prop:deloopmulti}
	Let $D$ be a decorated tangle diagram with boundary $E$ and coloring $\sigma$. Fix an integral lift of $\hslash$ enhancing $\mathcal A(E,\partial\sigma)$ into a dg-algebra. Then the multifactorization $\mathcal M (D)$ has an explicit special 0-deformation retract into an object $\mathcal M'(D)\in Tw^{h}\mathcal A(E,\partial\sigma)$.
\end{proposition}

\begin{proof}
	The multifactorization associated to the decorated tangle diagram can be written as a cube of resolutions
	\[
		\mathcal M (D) = \bigoplus_{s\subseteq I_c} \mathcal M (\Gamma_s),
	\]
	where each summand is a matrix factorization associated to a decorated crossingless diagram. Then Proposition \ref{prop:mfdelooping} can be used to decompose each summand into shifted copies of the objects in $\mathcal A(E,\partial\sigma)$ up to special $0$-deformation retract, thus defining a special deformation retract of matrix factorizations
	\[
		\sdr{\underline{\mathcal M}'}{\underline{\mathcal M (D)}}{}{}{}.
	\]
	The desired multifactorization deformation retract $\mathcal M'(D)$ is obtained by altering $d_{\underline{\mathcal M}'}$ in the way indicated by the following lemma and its proof.
\end{proof}

\begin{lemma} [Perturbation Lemma]\label{lem:specialDefRetract}
	Suppose
	\[
		\sdr{(M',d')}{(M,d)}{i}{p}{k}
	\]
	is a special 1-deformation retract between matrix multifactorizations $M'$ and $M$, and let $\tilde d$ be another multifactorization differential in $M$ such that its vertical part coincides with $d$ (i.e. $\tilde d_0=d_0$) and $(\tilde d-d)k$ is nilpotent. Then there exists a special 1-deformation retract
	\[
		\sdr{(M',\tilde d')}{(M,\tilde d)}{\iota}{\rho}{\kappa}.
	\]
	such that $\tilde d'_0=d'_0$.
	If, furthermore, $k$ is a $0$-homotopy, making the original data a special 0-deformation retract, then the result is also a special 0-deformation retract, with the additional properties $\rho_0=p_0$, $\iota_0=i_0$ and $\kappa_0=k_0$.
\end{lemma}

\begin{proof} Let
	\[
		A=(1-(\tilde d-d)k)^{-1}(\tilde d-d)=\sum_{j=0}^{\infty} ((\tilde d-d)k)^j(\tilde d-d).
	\]
	Notice that $\tilde d-d$ has filtered degree $\geq 1$, while $k$ has filtered degree $\geq -1$. In total, $A$ has filtered degree $\geq 1$.
	Take
	\begin{equation}\label{eq:deformedsdrdata}
		\begin{aligned}
			\tilde d'&=d'+pAi
			\\
			\rho&=p+pAk
			\\
			\iota&=i+kAi
			\\
			\kappa&=k+kAk
		\end{aligned}
	\end{equation}
	These define special deformation retract data: see \cite{crainic} for calculations, which readily adapt to non-zero potential. Notice that the term $pAi$, i.e. the deformation of the differential, has filtered degree $\geq 1$, which guarantees $\tilde d_0'=d_0'$. In the case where $k$ is a 1-homotopy, then $pAk$ and $kAi$ preserve filtration, while $kAk$ has filtered degree $\geq -1$, which are necessary conditions for the deformed data $(\iota,\rho,\kappa)$ to give a special 1-deformation retract. If $k$ is additionally a 0-homotopy, then all these deformation terms have now filtered degree $\geq 1$, so that $(\iota,\rho,\kappa)$ is a special 0-deformation retract with $\rho_0=p_0$, $\iota_0=i_0$ and $\kappa_0=k_0$.
\end{proof}

We conclude this section by describing how $\mathcal M(-)$ extends the Bar-Natan invariant in the case of 4-ended oriented tangles. Let us invoke the following theorem, which is itself a particular version of the homological mirror symmetry results in \cite{aaeko}. Recall that a functor $\mathcal F$ between $q$-filtered $A_\infty$-categories is called $q$-filtered if each component is $q$-homogeneous of degree $q(\mathcal F_n)=3(n-1)$.

\begin{theorem}[cf. \cite{kwzmutation}]
	Let $\mathbf k$ be a field. There is a $q$-filtered quasi-isomorphism
	\[
		\mathcal F: \mathcal A( E_*^4,\partial\sigma_0)\otimes \mathbf k \to \mathcal B^\infty_{ds}\otimes \mathbf k.
	\]
\end{theorem}

Since the delooping $\mathcal M'(D)$ of $\mathcal M(D)$ can be seen as an object in $Tw^{h}\mathcal A(E^4_*,\partial\sigma_0)$, we may use the induced functor to obtain a bifiltered twisted complex over $\mathcal B^\infty_{ds}\otimes\mathbf k$. We observe that $\mathcal F$ is unital.

\begin{definition}
	Let $T$ be a tangle with boundary $E^4_*$, and let $D$ be a decorated tangle diagram for $T$ with coloring $\sigma_0$. Let $\mathcal M'(D)\in Tw^{h}\mathcal A(E^4_*,\partial\sigma_0)$ be a special 0-deformation retract of $\mathcal M(D)$. Define
	\[
		\De^\infty(D) \in Tw^h (\mathcal B^\infty_{ds}\otimes\mathbf k)
	\]
	to be the image of $\mathcal M'(D)\otimes\mathbf k$ under the map induced by the quasi-isomorphism $\mathcal F$ at the level of bifiltered twisted complexes.
\end{definition}

\section{Invariance}\label{sec:inv}

Here are the main invariance results.

\begin{theorem}\label{thm:invariance}
	Let $T$ be an oriented tangle with boundary $E$ and let $D$ and $D'$ be decorated tangle diagrams for $T$ with fixed boundary coloring $\partial\sigma$. Then $\mathcal M(D)$ and $\mathcal M(D')$ are 1-homotopy equivalent matrix multifactorizations.
\end{theorem}

\begin{corollary}
	Let $T$ be a tangle with boundary $E^4_*$, and let $D$ be a decorated tangle diagram for $T$ with coloring $\sigma_0$. Then the invariant $\De^\infty(T):=\De^\infty(D)\in Tw^h (\mathcal B^\infty_{ds}\otimes\mathbf k)$ is well defined up to 1-homotopy equivalence of bifiltered twisted complexes.
\end{corollary}

\begin{proof}
	Given a decorated diagram $D$ for $T$, we have $\De^\infty(T)=\mathcal F(\mathcal M'( D)\otimes\mathbf k)$, where $\mathcal M' (D)\in Tw^h (\mathcal A( E_*^4,\partial\sigma_0))$ is a special 0-deformation retract of $\mathcal M (D)$. A different choice of diagram $D'$ and special 0-deformation retract $\mathcal M' (D')$ will be 1-homotopy equivalent to $\mathcal M' (D)$, by Theorem \ref{thm:invariance}. Since $\mathcal F$ preserves 1-homotopy equivalence, the corollary follows.
\end{proof}

\begin{remark}
	By the same arguments, there is another invariant of based 4-ended tangles corresponding to the reverse coloring $\sigma_1$. It contains the same information as $\De^\infty(-)$.
\end{remark}

The rest of this section is devoted to the proof of Theorem \ref{thm:invariance}. On the one hand, we want $\mathcal M (D)$ to be independent of the choice of diagram $D$, up to 1-homotopy. On the other hand, we also want independence from the choice of arc system $A$ and indexing. Let us begin with the precise statement of the latter.

\begin{proposition}\label{prop:arcinvariance}
	Let $D$ and $D'$ be decorated tangle diagrams with the same underlying tangle diagram, coloring and distinguished end, if there is one.
	\begin{enumerate}
		\item Suppose $D$ and $D'$ differ only in their respective sets of arcs $A$ and $A'$, and $A\subset A'$. Then there is a
		special $0$-deformation retract of multifactorizations
		\[
			\sdr{\mathcal M(D)}{\mathcal M(D')}{}{}{}.
		\]
		\item Suppose $D$ and $D'$ differ only by a permutation of the
		indexing set $I=I_a\cup I_c$. Then there is a vertical isomorphism of matrix multifactorizations
		$\mathcal M(D)\cong\mathcal M(D')$.
		\item Suppose $D$ and $D'$ differ only in their respective sets of arcs $A$ and $A'$, but these are arbitrary. Then $\mathcal M(D)$ and $\mathcal M(D')$ are 0-homotopy equivalent.
	\end{enumerate}
\end{proposition}

Observe that all three parts yield 0-homotopy equivalences, which is stronger than strictly needed for Theorem \ref{thm:invariance}. Before we present a proof of the proposition, we state a lemma that will also be useful in the rest of this section. Its proof follows directly from the definitions.

\begin{lemma}[Concatenating special deformation retracts]\label{lem:concat}
	Suppose we have special $n$-deformation retracts
	\[
		\sdr{(A,d_A)}{(B,d_B)}{i_1}{p_1}{k_1}
		,\qquad
		\sdr{(B,d_B)}{(C,d_C)}{i_2}{p_2}{k_2}
		.
	\]
	Then the following is a special $n$-deformation retract
	\[
		\sdr{(A,d_A)}{(C,d_C)}{i_2i_1}{p_1p_2}{k_2+i_2k_1p_2}
		.
	\]
\end{lemma}

\begin{proof}[Proof of Proposition \ref{prop:arcinvariance}]
	For part 1, assume first that $A'=A\cup \{\alpha\}$. If we define vertical deformation retract data consisting of arc removal \eqref{sdr:interiorfaceremoval} at each resolution, the Perturbation Lemma \ref{lem:specialDefRetract} yields a special 0-deformation retract of the desired form. Such special 0-deformation retracts can be concatenated by Lemma \ref{lem:concat} to obtain the general case $A\subset A'$. Regarding part 2, we note that relabeling isomorphisms may be used at each resolution to define a vertical isomorphism of matrix multifactorizations. For the final part, assume first that there are no intersections between arcs in $A$ and $A'$. Then there is a third diagram $D''$ with arc system $A\cup A'$, up to some relabeling, so that $\mathcal M(D'')$ has both $\mathcal M(D)$ and $\mathcal M(D')$ as special 0-deformation retracts, making them 0-homotopy equivalent. For the general case, we can assume that the arcs in $A$ and $A'$ intersect transversally. Consider arcs $\alpha\in A$ and $\alpha'\in A'$ intersecting at $p\in\mathbb D^2$, and such that $\alpha'$ meets no other arcs in $A$ in the segment $\beta$ between $p$ and the tip of $\alpha'$. Consider arcs $\alpha_i$ and $\alpha_j$ which run parallel to $\alpha$ except in a neighborhood of the segment $\beta$, where they take the form
	\[
		\begin{tikzpicture}[scale=1.1, baseline={(0,0)}]
			\fill[black!10] (-35:1) .. controls (0.48,-0.42) and (0.48,0.42) .. (35:1)
				arc[start angle=35, end angle=-35, radius=1] -- cycle;
			\draw[line width=0.9pt] (-35:1) .. controls (0.48,-0.42) and (0.48,0.42) .. (35:1);
			\draw[dotted, line width=0.6pt] (0,0) circle (1);
			\draw[dotted, dash pattern=on 1pt off 1pt, line width=0.6pt] (-0.3,0.954) -- (-0.3,-0.954);
			\draw[dash pattern=on 2.5pt off 1.5pt, line width=0.5pt, black!55, -{>[scale=0.8]}] (-1,0) -- (0.565,0);
			\draw[dotted, dash pattern=on 1pt off 1pt, line width=0.6pt, -{>[scale=0.8]}] (-0.19,0.982) -- (-0.19,0.27)
				.. controls (-0.19,0.21) and (-0.17,0.19) .. (-0.11,0.19) -- (0.583,0.19);
			\draw[dotted, dash pattern=on 1pt off 1pt, line width=0.6pt, -{>[scale=0.8]}] (-0.19,-0.982) -- (-0.19,-0.27)
				.. controls (-0.19,-0.21) and (-0.17,-0.19) .. (-0.11,-0.19) -- (0.583,-0.19);
			\fill (-0.3,0) circle (0.7pt);
			\node[font=\tiny, anchor=north east, inner sep=1pt] at (-0.32,-0.02) {$p$};
			\node[font=\tiny, anchor=east, inner sep=1pt] at (-0.32,0.6) {$\alpha$};
			\node[font=\tiny, anchor=south, inner sep=1pt] at (-0.68,0.03) {$\alpha'$};
			\node[font=\tiny, anchor=west, inner sep=1pt] at (-0.17,0.62) {$\alpha_i$};
			\node[font=\tiny, anchor=west, inner sep=1pt] at (-0.17,-0.62) {$\alpha_j$};
		\end{tikzpicture}_.
	\]
	By part 1, $\mathcal M(D)$ is a special 0-deformation retract of the multifactorization $\mathcal M(\tilde D)$ associated to the diagram $\tilde D$ with arc system $A\cup\{\alpha_i,\alpha_j\}$. The latter, in turn, has another special 0-deformation retract $\mathcal M(\overline D)$ associated to the diagram $\overline D$ with arc system $\overline A:=A\cup\{\alpha_i,\alpha_j\}\setminus \{\alpha\}$. But now there is one fewer intersection point between arcs in $\overline A$ and $A'$ than between arcs in $A$ and $A'$. We can iterate this argument to reduce to the intersection-free case.
\end{proof}

Verification of invariance of $\mathcal M(D)$ under Reidemeister moves will proceed by looking at the explicit matrix multifactorization associated to the tangle at one side of each Reidemeister move and simplifying it in a few steps. At its core, the argument consists of decomposing each decorated resolution into simple decorated diagrams by means of arc removal and delooping special 0-deformation retracts, which then extend into a special 0-deformation retract of the entire multifactorization by the Perturbation Lemma \ref{lem:specialDefRetract}. There is enough control of the deformation data to compute explicit differentials in the resulting multifactorization, and in particular to identify some isomorphism components in $d_1$. One then removes the contractible summands along these isomorphisms to obtain a special 1-deformation retract of a desirable simplified form. Removal of 1-contractible summands is a particular case of Lemma \ref{lem:specialDefRetract}, which we make explicit now for convenience.

\begin{lemma}\label{lem:SummandLemma}
	Let $(M,d)$ be a matrix multifactorization of the form $M=C\oplus M'$, with differential
	$d_M=\begin{pmatrix}d_C & \beta \\ \gamma & d_{M'} \end{pmatrix}$, and such that $\beta_0$ and $\gamma_0$ vanish. Suppose that we have a special 1-deformation retract
	\[
		\sdr{0}{(C,d_C)}{}{}{k}.
	\]
	Then there exists a special 1-deformation retract
	\[
		\gsdr{(M',d_{M'}+\gamma k\beta)}{(M,d_M)}{\iota}{\rho}{\kappa}.
	\]
\end{lemma}

Here, $k$ is called a 1-\emph{nullhomotopy}. We reserve the green color to indicate the arrow data of a special 1-deformation retract obtained by 1-contractible summand removal as in the lemma.

The following theorem establishes the invariance of $\mathcal M(-)$ under Reidemeister moves up to 1-homotopy equivalence. See \cite{polyak} for a proof that the list below is a generating set of oriented Reidemeister moves.

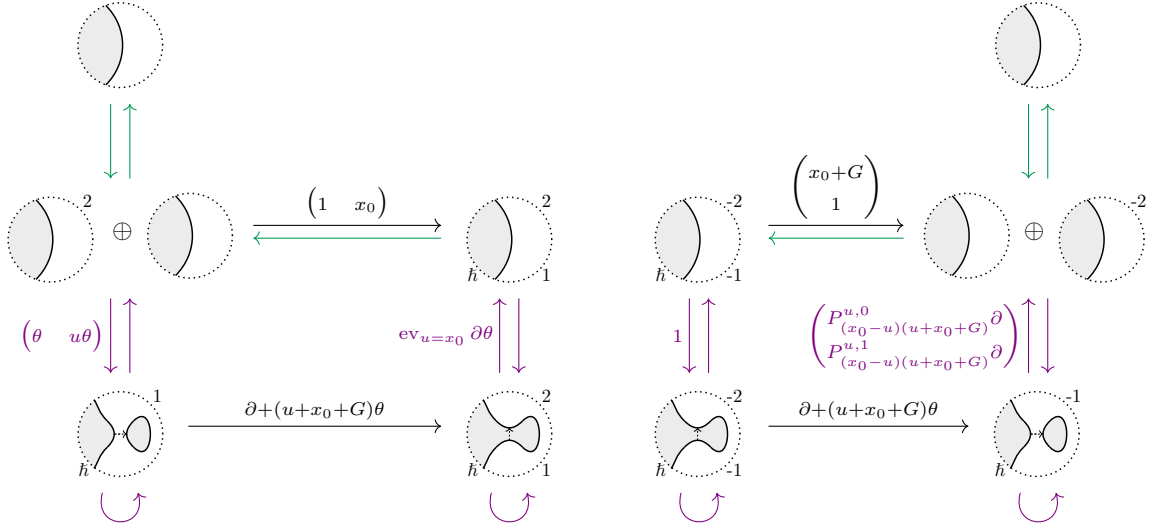
\begin{figure}
	\centering
	\[
		\begin{tikzcd}[column sep=2.5cm, row sep=1cm, ampersand replacement=\&]
			\scalebox{0.7}{\SingleEdge}
			\arrow[d, color=ForestGreen, shift right=1.5]
			\& \\
			\scalebox{0.7}{\SingleEdge[q=2]} \oplus \scalebox{0.7}{\SingleEdge}
			\arrow[u, color=ForestGreen, shift right=1.5]
			\arrow[d, color=violet, shift right=1.5, "{\begin{pmatrix} \theta& u\theta
			\end{pmatrix}}", swap]
			\arrow[r, "{\begin{pmatrix}1 & x_0\end{pmatrix}}", shift left=1, start anchor=east, end anchor=west]
			\&
			\scalebox{0.7}{\SingleEdge[q=2, h=1, hslash]}
			\arrow[d, color=violet, shift left=1.5]
			\arrow[l, color=ForestGreen, shift left=1]
			\\
			\scalebox{0.7}{\ThinEdgeAndCircle[q=1, hslash]}
			\arrow[r, "\partial+(u+x_0+G)\theta"]
			\arrow[u, color=violet, shift right=1.5]
			\arrow[loop right, distance=1.5em, in=-75, out=-105, color=violet]
			\&
			\scalebox{0.7}{\ThinEdgeAndCircle[dir1=1, q=2, h=1, hslash]}
			\arrow[u, color=violet, shift left=1.5, "{\ev_{u=x_0}\partial\theta}"]
			\arrow[loop right, distance=1.5em, in=-75, out=-105, color=violet]
		\end{tikzcd}
		\qquad
		\begin{tikzcd}[row sep=1cm, column sep=1.8cm]
			&
			\scalebox{0.7}{\SingleEdge}
			\arrow[d, color=ForestGreen, shift right=1.5]
			\\
			\scalebox{0.7}{\SingleEdge[ q=-2, h=-1, hslash]}
			\arrow[d, color=violet, shift right=1.5, "1", swap]
			\arrow[r, "{\begin{pmatrix}x_0+G \\ 1
			\end{pmatrix}}", shift left=1]
			&
			\scalebox{0.7}{\SingleEdge}\oplus \scalebox{0.7}{\SingleEdge[q=-2]}
			\arrow[d, color=violet, shift left=1.5, ]
			\arrow[u, color=ForestGreen, shift right=1.5]
			\arrow[l, color=ForestGreen, shift left=1]
			\\
			\scalebox{0.7}{\ThinEdgeAndCircle[dir1=1, q=-2, h=-1, hslash]}
			\arrow[r, "\partial+(u+x_0+G)\theta"]
			\arrow[u, color=violet, shift right=1.5]
			\arrow[loop right, distance=1.5em, in=-75, out=-105, color=violet]
			&
			\scalebox{0.7}{\ThinEdgeAndCircle[q=-1, hslash]}
			\arrow[u, color=violet, shift left=1.5, "{\begin{pmatrix} P^{u,0}_{(x_0-u)(u+x_0+G)}\partial \\
					P^{u,1}_{(x_0-u)(u+x_0+G)}\partial \end{pmatrix}}"]
			\arrow[loop right, distance=1.5em, in=-75, out=-105, color=violet]
		\end{tikzcd}
	\]
	\caption{Invariance under R1a (left) and R1b (right) moves.}
	\label{fig:R1}
\end{figure}

\begin{theorem}
	There are special 1-deformation retracts
	\begin{align*}
		\emph{(R1a)} &\sdr{\scalebox{0.9}{\SingleEdge}}{\ThinEdgeAndCircle[cr1, dir1=1]}{}{}{}
		&\emph{(R1b)}& \sdr{\scalebox{0.9}{\SingleEdge}}{\ThinEdgeAndCircle[cr1=-1, dir1=1]}{}{}{}
		\\
		\emph{(R2a)}&\sdr{\scalebox{0.9}{\Elementary[dir1=1]}}{\scalebox{0.8}{\Square[dir1=1, dir2=3, cr1, cr2=-1]}}{}{}{}
		& \emph{(R2b)} & \sdr{\scalebox{0.9}{\Elementary[dir1=3]}}{\scalebox{0.8}{\Square[dir1=3, dir2=1, cr1=1, cr2=-1]}}{}{}{}
	\end{align*}
	Moreover, there is a matrix multifactorization $M'''_{R3}$ which is a common special 1-deformation retract of both
	\[
		\scalebox{0.8}{\RLR[cr1, cr2, cr3, dir1=1, dir2=1, dir3=1]} \quad\mbox{and}\quad \scalebox{0.8}{\LRL[cr1, cr2, cr3, dir1=1, dir2=1, dir3=1]}.
	\]
	There are also analogous special 1-deformation retracts if all colorings are inverted.
\end{theorem}

\begin{proof}
	We build explicit special 1-deformation retracts for each move, using the methods for diagram simplification in Subsection~\ref{subsec:simplify}.\footnote{All calculations in this proof can be verified by hand. For the reader's convenience, a Python notebook automating some of our constructions and carrying out these computations is available at \url{https://github.com/tme93/bar-natan-multifactorizations}.} Reidemeister 1 deformation retracts are illustrated in Figure \ref{fig:R1}. Let us first focus on the R1a move, which corresponds to the left side of the figure. The multifactorization initially associated to the tangle
	\[
		\scalebox{0.8}{\ThinEdgeAndCircle[cr1, dir1=1, extra={\node at (0.45,0) {\tiny $u$};}]}
	\]
	is at the bottom row of Figure \ref{fig:R1}. With the application of the Perturbation Lemma \ref{lem:specialDefRetract} in mind, let us denote this bottom row by $(C,D)$ and the middle row by $(C',D')$. The violet arrows represent the special deformation retracts of matrix factorizations appearing in Examples \ref{ex:R1arcremoval} and \ref{ex:R1delooping}. Together, these violet
	arrows define a special 0-deformation retract between the vertical factorizations $\underline C$ and $\underline C'$. In this case, the perturbation $D-d$ is the horizontal part $D_1$ consisting of just a saddle map. By Formula \eqref{eq:deformedsdrdata}, the only non-vertical component in $(C',D')$ is calculated as the composite of the other three labeled arrows.

	Now, in the middle row $(C',D')$ we have a 1-contractible summand
	\[
		\left[\begin{tikzcd}
			\scalebox{0.6}{\SingleEdge[q=2]}
			\arrow[r,"1"] &
			\scalebox{0.6}{\SingleEdge[q=2, h=1, hslash]}
		\end{tikzcd}\right]_,
	\]
	which we cancel according to Lemma \ref{lem:SummandLemma}. The R1b move is completely analogous.\\

	\begin{figure}[t!]
		\centering
		\[
			\begin{tikzcd}[row sep=-0.4cm, column sep=4cm, ampersand replacement=\&]
				\&
				\scalebox{0.7}{\Elementary[dir1=1]}
				\\[4em]\&
				\scalebox{0.7}{\Elementary[dir1=1]}
				\arrow[dr, "-\partial_1-(x_0+x_2+G)\theta_2", pos=0.3]
				\arrow[ddd, bend left=40, color=violet, shift right=1.5, dash pattern=on 30pt off 38pt, "(-\partial_1+\partial_2+(x_0-x_1)\theta_1)\theta_2", swap, pos=0.1]
				\arrow[leftarrow, ddd, bend left=40, color=violet, shift left=1.5, dash pattern=on 32pt off 36pt, "\ev_{u=x_0+x_2-x_1}\partial_2\theta_2", pos=0.9]
				\arrow[u, color=ForestGreen, shift right=1.5]
				\&
				\\
				\scalebox{0.7}{\Elementary[q=-1,h=-1]}
				\arrow[dr, "-\begin{pmatrix}
					x_2+G \\ 1
				\end{pmatrix}", swap, pos=0.7, , shift right=1]
				\arrow[ur, "\partial_1+(x_0+x_2+G)\theta_1", pos=0.7]
				\arrow[ddd, color=violet, shift left=1.5, "(\partial_2+(u+x_2+G)\theta_1)\theta_2", pos=0.6 ]
				\&
				\&
				\scalebox{0.7}{\Elementary[q=1,h=1, extra={\node at (0,0.1) {\tiny $\alpha_2$};}]}
				\arrow[ddd, color=violet, shift left=1.5]
				\arrow[dl, shift right=1.5, color=ForestGreen]
				\\
				\&
				\scalebox{0.7}{\Elementary[q=1, dir1=2, hslash]}\oplus \scalebox{0.7}{\Elementary[q=-1, dir1=2, hslash]}
				\arrow[ur, crossing over, "{-\begin{pmatrix} 1 & x_0 \end{pmatrix}{(\partial_1\theta_1+\theta_2\partial_1)}}", swap, pos=0.2, shift right=1]
				\arrow[ul, shift right=1, color=ForestGreen]
				\&
				\\[2em]
				\&
				\scalebox{0.7}{\Square[dir1=1,dir2=1]}
				\arrow[dr, "-\partial_2-(u+x_2+G)\theta_2", pos=0.3]
				\&
				\\
				\scalebox{0.7}{\Square[dir2=1,q=-1,h=-1, lab1={\tiny $\alpha_1$}, lab2={\tiny $\alpha_2$}]}
				\arrow[ur, "\partial_1+(u+x_1+G)\theta_1", pos=0.7]
				\arrow[dr,"-\partial_2-(u+x_2+G)\theta_2", swap, pos=0.7]
				\arrow[uuu, color=violet, shift left=1.5]
				\&\&
				\scalebox{0.7}{\Square[dir1=1,q=1,h=1]}
				\arrow[uuu, color=violet, shift left=1.5, "\ev_{u=x_0}\partial_1\theta_1"]
				\\
				\&
				\scalebox{0.7}{\Square}
				\arrow[ur, "\partial_1+(u+x_1+G)\theta_1", swap, pos=0.3]
				\&
				\arrow[ddd, from=4-2, to=7-2, bend right=40, crossing over, color=violet, shift right=1.5, "{\begin{pmatrix}
						1 & u
				\end{pmatrix}}(-\theta_1+\theta_2)", swap, pos=0.68]
				\arrow[uuu, from=7-2, to=4-2, bend left=40, crossing over, color=violet, shift right=1.5, "{\begin{pmatrix}P^{u,0}_{(u+x_1-x_0-x_2)(u+x_2+G)} \\ P^{u,1}_{(u+x_1-x_0-x_2)(u+x_2+G)}\end{pmatrix}\partial_2}", swap, pos=0]
			\end{tikzcd}
		\]
		\caption{Invariance under R2a move.}
		\label{fig:r2}
	\end{figure}
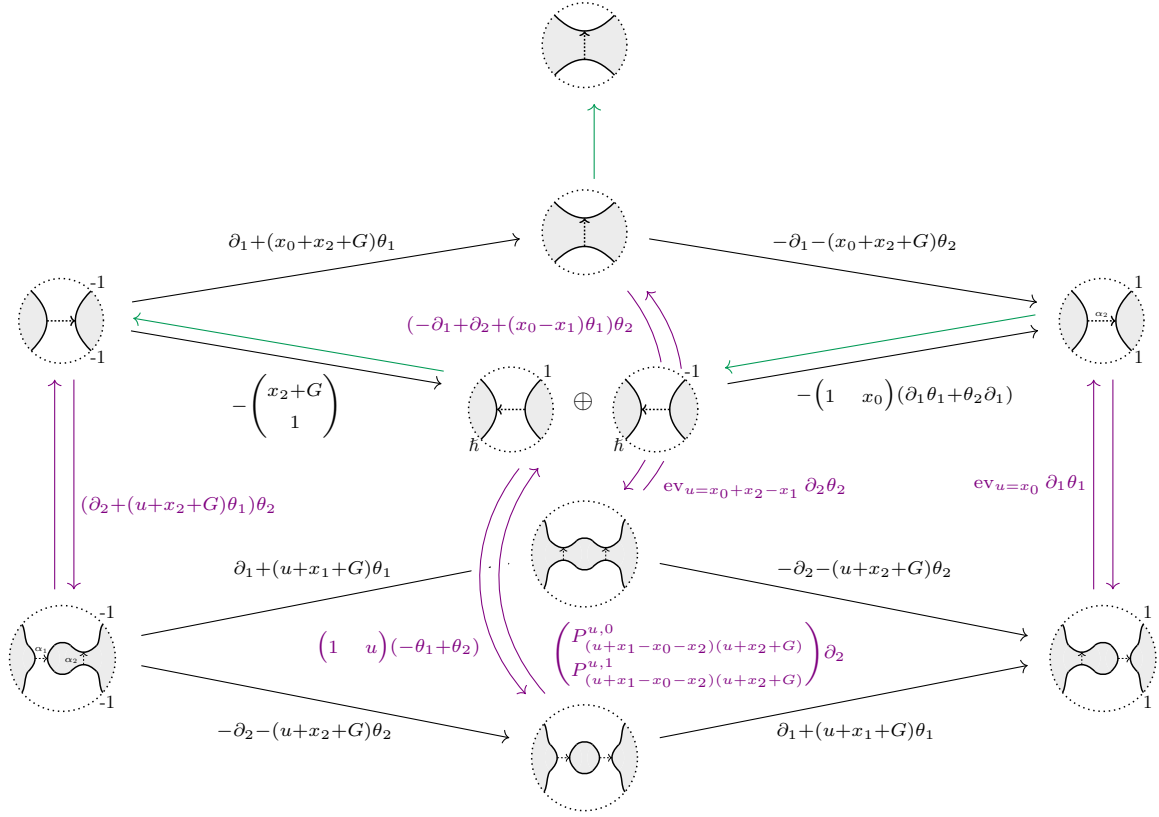

	The argument for Reidemeister 2 deformation retracts follows along the same lines. In Figure \ref{fig:r2}, we represent the special 1-deformation retract corresponding to the R2a move. The cube of resolutions of the tangle diagram
	\[
		\scalebox{0.7}{\Square[dir1=1, dir2=3, cr1, cr2=-1, extra={\node at (0,0.4) {\tiny $u$};}]}
	\]

	appears at the bottom layer. There is, again, a special deformation retract simplifying the decorated diagram at each vertex, depicted in violet. A couple of these deformation retracts were explicitly computed in Examples \ref{ex:R2arcremoval} and \ref{ex:thinthin}, while the others are computed analogously. Following the same perturbation argument that we used for R1a, we obtain all the arrows in the middle layer from the given vertical special deformation retract data via the Perturbation Lemma \ref{lem:specialDefRetract}. Observe that, in principle, Formula \eqref{eq:deformedsdrdata} could give an arrow of length (i.e. $h$-degree) 2, but it vanishes by direct calculation. We then use Lemma \ref{lem:SummandLemma} twice to remove first the 1-contractible piece corresponding to the isomorphism $-(\partial_1\theta_1+\theta_2\partial_1)$ with target at the rightmost vertex, and then the 1-contractible piece corresponding to the isomorphism $-1$ with source the leftmost vertex.

	The R2b move corresponds to an orientation reversal on both strands of R2a. To obtain the desired special 1-deformation retract in that case, we simply reverse all the arcs in Figure \ref{fig:r2}.\\

	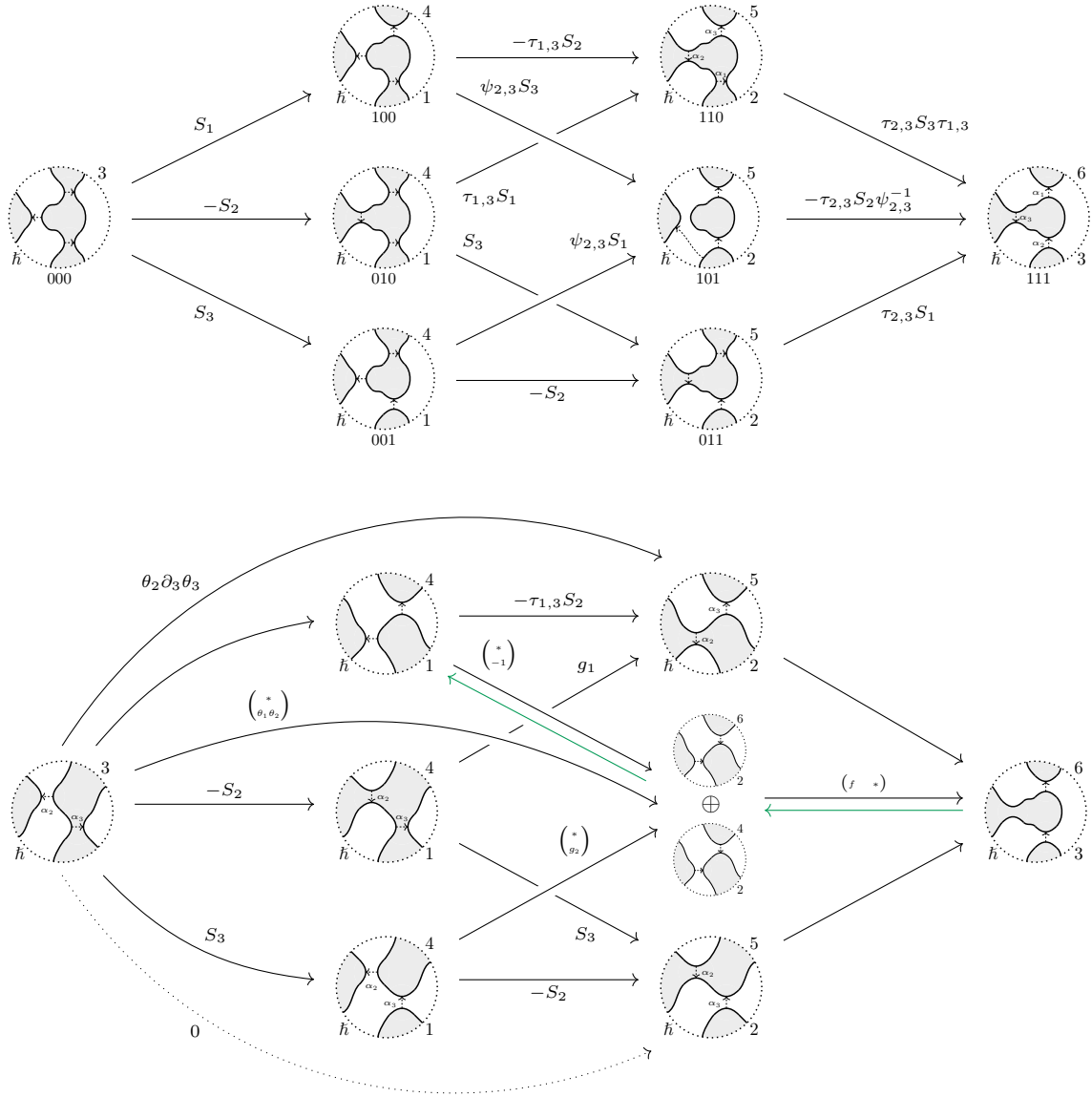
\begin{figure}[h!]
		\centering
		\begin{tikzcd}[row sep=0.2cm, column sep=2.5cm]
			&
			\scalebox{0.7}{\RLR[dir1=1, dir2=2, q=4, h=1, hslash, extra={\node at (0,-1.3) {\small 100};}]}
			\arrow[r, "-\tau_{1,3}S_2"]
			&
			\scalebox{0.7}{\RLR[dir2=3, dir1=1, q=5, h=2, hslash, lab1={\tiny $\alpha_3$}, lab2={\tiny $\alpha_2$}, lab3={\tiny $\alpha_1$}, extra={\node at (0,-1.3) {\small 110};}]}
			\arrow[dr, "\tau_{2,3}S_3\tau_{1,3}"]
			&
			\\
			\scalebox{0.7}{\RLR[q=3,dir2=2, hslash, extra={\node at (0,-1.3) {\small 000};}]}
			\arrow[ur, "S_1"]
			\arrow[r, "-S_2"]
			\arrow[dr, "S_3", swap]
			&
			\scalebox{0.7}{\RLR[dir2=3, q=4, h=1, hslash, extra={\node at (0,-1.3) {\small 010};}]}
			\arrow[ur, "\tau_{1,3}S_1", swap, pos=0]
			\arrow[dr,"S_3", pos=0]
			&
			\scalebox{0.7}{\RLR[dir1=1, dir2=2, emp2, dir3=1, q=5, h=2, hslash, extra={\draw[dotted, dash pattern=on 1pt off 1pt, {<[scale=0.8]}-]
					(-0.65,-0.2)
					.. controls (-0.45,-0.3) and (-0.3,-0.7)
					.. (-0.1,-0.8);
					\node at (0,-1.3) {\small 101};}]

			}
			\arrow[r, "-\tau_{2,3}S_2\psi^{-1}_{2,3}", pos=0.4]
			\arrow[ul, leftarrow, crossing over, "\psi_{2,3} S_3", pos=0.9, swap]
			&
			\scalebox{0.7}{\RLR[dir1=1,dir2=3,dir3=1, q=6, h=3, hslash, lab1={\tiny $\alpha_1$}, lab2={\tiny $\alpha_3$}, lab3={\tiny $\alpha_2$}, extra={\node at (0,-1.3) {\small 111};}]}
			\\
			&
			\scalebox{0.7}{\RLR[dir3=1, dir2=2, q=4, h=1, hslash, extra={\node at (0,-1.3) {\small 001};}]}
			\arrow[ur, crossing over, "\psi_{2,3} S_1", pos=1]
			\arrow[r, "-S_2", swap]
			&
			\scalebox{0.7}{\RLR[dir2=3,dir3=1, q=5, h=2, hslash, extra={\node at (0,-1.3) {\small 011};}]}
			\arrow[ur, "\tau_{2,3}S_1", swap]
			&
		\end{tikzcd}
		\begin{tikzcd}[row sep=0.2cm, column sep=2.5cm]
			&
			\scalebox{0.7}{\RL[dir1=1,dir2=2, q=4, h=1, hslash]}
			\arrow[r, "-\tau_{1,3}S_2"]
			&
			\scalebox{0.7}{\RL[dir2=3, dir1=1, q=5, h=2, hslash, lab1={\tiny $\alpha_3$}, lab2={\tiny $\alpha_2$}]}
			\arrow[dr]
			&
			\\
			\scalebox{0.7}{\LR[dir1=2,q=3, hslash, lab1={\tiny $\alpha_2$}, lab2={\tiny $\alpha_3$}]}
			\arrow[ur, bend left=20, start anchor=60,
			"
			", pos=0.8]
			\arrow[r, "-S_2"]
			\arrow[dr, bend right=20, start anchor=-60,
			"S_3"]
			\arrow[urr, bend left=40 , start anchor=90, end anchor=130, "\theta_2\partial_3\theta_3", pos=0.3]
			\arrow[drr, bend right=40 , start anchor=-90, end anchor=-130, "0", pos=0.3, dotted, swap]
			&
			\scalebox{0.7}{\LR[dir1=3, q=4, h=1, hslash, lab1={\tiny $\alpha_2$}, lab2={\tiny $\alpha_3$}]}
			\arrow[ur, "g_1", pos=0.8]
			\arrow[dr, "S_3", pos=0.8, swap]
			&
			\begin{matrix}
				\scalebox{0.5}{\RL[dir1=3, q=6, h=2]}
				\\
				\oplus
				\\
				\scalebox{0.5}{\RL[dir1=3, q=4, h=2]}
			\end{matrix}
			\arrow[r, shift left=1, "\tiny \begin{pmatrix}
				f \quad *
			\end{pmatrix}"]
			\arrow[r, leftarrow, shift right=1, color=ForestGreen]
			\arrow[ul, leftarrow, crossing over, start anchor=140, , "\tiny \begin{pmatrix}
				* \\ -1
			\end{pmatrix}", swap, pos=0.9, shift left=2]
			\arrow[ul, crossing over, start anchor=140, color=ForestGreen, shift left=4]
			\arrow[ll,leftarrow, bend right=25, "\tiny \begin{pmatrix} * \\ \theta_1\theta_2\end{pmatrix}", crossing over, pos=0.7, swap, start anchor=180]
			&
			\scalebox{0.7}{\RLR[dir1=1,dir2=1, dir3=1, emp2, q=6, h=3, hslash]}
			\\
			&
			\scalebox{0.7}{\LR[dir1=2, dir2=1, q=4, h=1, hslash, lab1={\tiny $\alpha_2$}, lab2={\tiny $\alpha_3$}]}
			\arrow[ur, crossing over, "\tiny \begin{pmatrix} * \\ g_2 \end{pmatrix}", pos=0.7, end anchor=-155]
			\arrow[r, "-S_2", swap]
			&
			\scalebox{0.7}{\LR[dir1=3,dir2=1, q=5, h=2, hslash, lab1={\tiny $\alpha_2$}, lab2={\tiny $\alpha_3$}]}
			\arrow[ur, swap]
			&
		\end{tikzcd}
		\caption{Intermediate simplification steps in the proof of Reidemeister 3: $M_{R3}'$ (top) and $M_{R3}''$ (bottom). In $M_{R3}''$ we have $f:=1+(\partial_2 +(x_1-x_2)\theta_2)\theta_1$, $g_2:=1+(x_4-x_3)\theta_1\theta_2+(\theta_1+\theta_2)\partial_3$ and $g_1:=\partial_3(1-\partial_2\theta_3)+(x_1-x_0)\theta_2\partial_3\theta_3-(x_0+x_2+x_4-x_1+G)\theta_3$.}
		\label{fig:r3}
	\end{figure}

	For the Reidemeister 3 move we want to prove an equivalence
	\begin{equation}\label{eq:R3}
		\scalebox{0.7}{\RLR[cr1, cr2, cr3, dir1=1, dir2=1, dir3=1, extra={\node at (0.65,0) {\tiny $u$};}]} 
		\simeq 
		\scalebox{0.7}{\LRL[cr1, cr2, cr3, dir1=1, dir2=1, dir3=1, extra={\node at (0.23,0) {\tiny 2}; \node at (0.05,0.52) {\tiny 3};\node at (0.05,-0.52) {\tiny 1}; \node at (-0.65,0) {\tiny $u$};}]}_.
	\end{equation}
	Observe that, by Proposition \ref{prop:arcinvariance}, no generality is lost after the relabeling on the right-hand side. Let $M_{R3}$ be the multifactorization associated to the left-hand side diagram. We simplify in three steps.

	\emph{Step 1.} Consider the multifactorization $M_{R3}'$ at the top of Figure \ref{fig:r3}, which is obtained from $M_{R3}$ by relabeling arcs at the vertices $110$ and $111$ of the cube, while also sliding an arc at $101$. We have corresponding explicit matrix factorization isomorphisms $\tau_{1,3}$, $\tau_{2,3}$ and $\psi_{2,3}$ for the transformations at each of these three vertices, afforded by Lemmas \ref{lem:labelperm} and \ref{lem:arrowslide}. The components of $(d_{M'_{R3}})_1$ are defined by compositions of the original saddles $S_i$ with the corresponding isomorphisms and their inverses as needed. There is a vertical isomorphism $M_{R3}\cong M'_{R3}$ given by the chosen isomorphisms at the vertices $110$, $111$ and $101$, and identities at all other vertices. These particular transformations are chosen to help obtain a conveniently symmetric form in the fully simplified expression $M_{R3}'''$ below.

	\emph{Step 2.} Next, we simplify all diagrams by removing an arc at each resolution other than the one sitting at the vertex $101$, in which we do a delooping. Explicit formulas for the special deformation retracts at $000$ and $101$ can be found in Examples \ref{ex:R3arcremoval} and \ref{ex:circleremovalreid3}, respectively. At all other vertices, the special deformation retract is computed just as in $000$. This defines a vertical special deformation retract $\underline{M''_{R3}}$ of $\underline{M'_{R3}}$. Through the Perturbation Lemma \ref{lem:specialDefRetract} and the formulas in its proof, we obtain the special 0-deformation retract $M''_{R3}$ of $M'_{R3}$. All the explicit information we need from $M''_{R3}$ appears at the bottom of Figure~\ref{fig:r3}. There are possible length 2 and 3 arrows ending at $111$, but they are not depicted as they disappear without contribution in the next step. The same goes for the depicted but unlabeled arrows and arrow components.

	\begin{figure}[h!]
		\centering
		\begin{tikzcd}[column sep=2.5cm, row sep=0.2cm]
			&
			\scalebox{0.7}{\LR[dir1=3, q=4, h=1, hslash, lab1={\tiny $\alpha_2$}, lab2={\tiny $\alpha_3$}]}
			\arrow[rdd, "\partial_3-(x_2+x_4+G)\theta_3", pos=0.79, bend left=5]
			\arrow[r, "\partial_3(1-\partial_2\theta_3)+(x_1-x_0)\theta_2\partial_3\theta_3-(x_0+x_2+x_4-x_1+G)\theta_3"]
			&[4.8cm]
			\scalebox{0.7}{\RL[dir2=3, dir1=1, q=5, h=2, hslash, lab1={\tiny $\alpha_3$}, lab2={\tiny $\alpha_2$}]}
			\\
			\scalebox{0.7}{\LR[dir1=2, hslash, q=3, lab1={\tiny $\alpha_2$}, lab2={\tiny $\alpha_3$}]}
			\arrow[ur, "-\partial_2-(x_0+x_2+x_4-x_3+G)\theta_2", pos=0.8, bend left=20]
			\arrow[dr, "\partial_3-(x_2+x_4+G)\theta_3", swap, pos=0.7, bend right=20]
			\arrow[urr, "\partial_2\theta_2\theta_3+\theta_2\partial_3\theta_3", pos=0.2, crossing over, swap]
			&
			&
			\\
			&
			\scalebox{0.7}{\LR[dir1=2, dir2=1, q=4, h=1, hslash, lab1={\tiny $\alpha_2$}, lab2={\tiny $\alpha_3$}]}
			\arrow[uur, crossing over, shift left=0.5ex, bend right=5, "\partial_2(\partial_3\theta_2-1)+(x_4-x_3)\theta_3\partial_2\theta_2-(x_0+x_2+G)\theta_2", pos=0.3]
			\arrow[r, "-\partial_2-(x_0+x_2+x_4-x_3+G)\theta_2", swap]
			&
			\scalebox{0.7}{\LR[dir2=1,dir1=3, hslash, q=5, h=2, lab1={\tiny $\alpha_2$}, lab2={\tiny $\alpha_3$}]}
		\end{tikzcd}\vspace{0.5em}
		\begin{tikzcd}[column sep=2.5cm, row sep=0.2cm]
			&
			\scalebox{0.7}{\LR[dir1=0, dir2=3, q=4, h=1, hslash, lab1={\tiny $\alpha_3$}, lab2={\tiny $\alpha_2$}]}
			\arrow[rdd, "\partial_3+(x_0+x_2+x_4-x_3+G)\theta_3", pos=0.79, bend left=5]
			\arrow[r, "\partial_3(1-\partial_2\theta_3)+(x_3-x_4)\theta_2\partial_3\theta_3+(x_0+x_2+G)\theta_3"]
			&[4.8cm]
			\scalebox{0.7}{\RL[dir2=1, dir1=3, q=5, h=2, hslash, lab1={\tiny $\alpha_2$}, lab2={\tiny $\alpha_3$}]}
			\\
			\scalebox{0.7}{\LR[dir1=0,dir2=2, hslash, q=3, lab1={\tiny $\alpha_3$}, lab2={\tiny $\alpha_2$}]}
			\arrow[ur, "-\partial_2+(x_2+x_4+G)\theta_2", pos=0.8, bend left=20]
			\arrow[dr, "\partial_3+(x_0+x_2+x_4-x_3+G)\theta_3", swap, pos=0.7, bend right=20]
			\arrow[urr, "-\partial_2\theta_2\theta_3-\theta_2\partial_3\theta_3", pos=0.2, crossing over, swap]
			&
			&
			\\
			&
			\scalebox{0.7}{\LR[dir1=1, dir2=2, q=4, h=1, hslash, lab1={\tiny $\alpha_3$}, lab2={\tiny $\alpha_2$}]}
			\arrow[uur, crossing over, shift left=0.5ex, bend right=5, "\partial_2(\partial_3\theta_2-1)+(x_0-x_1)\theta_3\partial_2\theta_2+(x_0+x_2+x_4-x_1+G)\theta_2", pos=0.3]
			\arrow[r, "-\partial_2+(x_2+x_4+G)\theta_2", swap]
			&
			\scalebox{0.7}{\LR[dir2=3, dir1=1, hslash, q=5, h=2, lab1={\tiny $\alpha_3$}, lab2={\tiny $\alpha_2$}]}
		\end{tikzcd}
		\caption{Simplified matrix factorizations $M_{R3}'''$ (top) and $N_{R3}'''$ (bottom) for the two sides of Reidemeister 3. }
		\label{fig:r3simp}
	\end{figure}
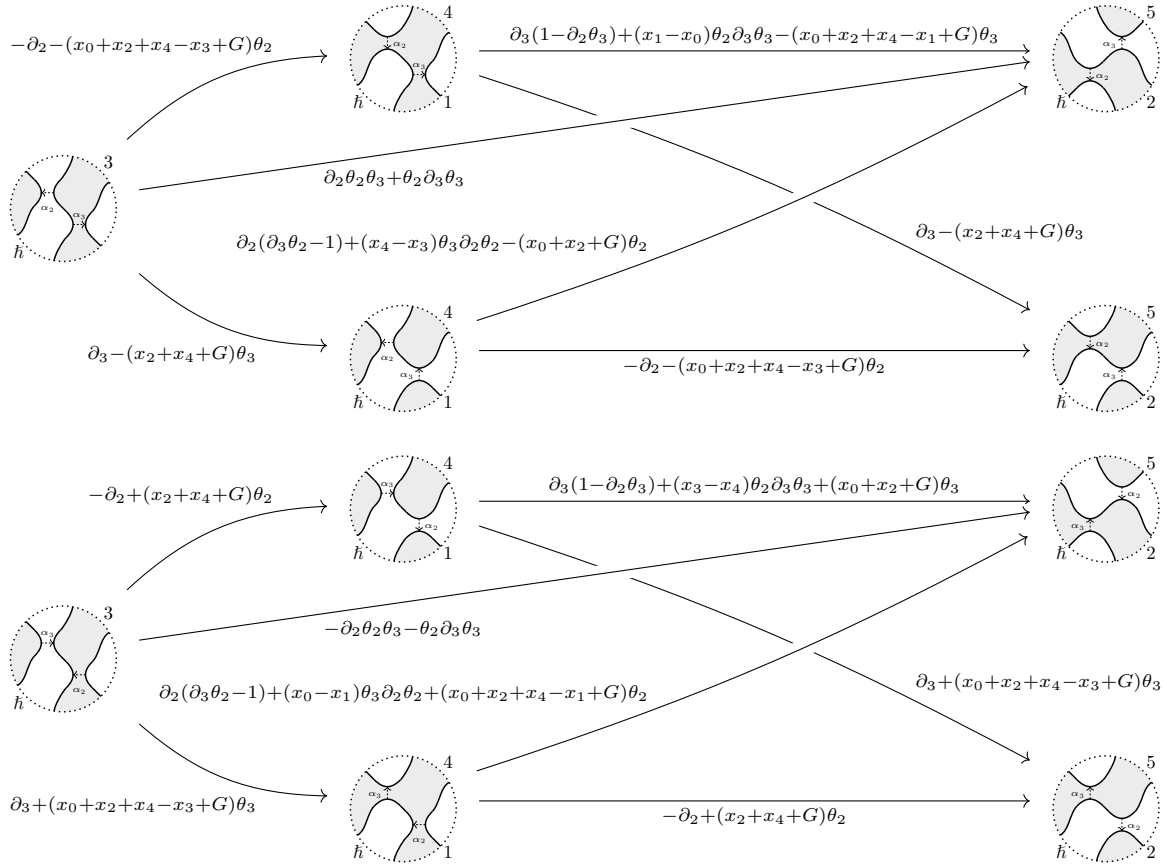

	\emph{Step 3.} Finally, we remove 1-contractible pieces using 1-nullhomotopies indicated by the green arrows in Figure~\ref{fig:r3}. The component $f=1+(\partial_2 +(x_1-x_2)\theta_2)\theta_1$ of the arrow $101\to 111$ in $M''_{R3}$ has inverse $f^{-1}:=1-(\partial_2 +(x_1-x_2)\theta_2)\theta_1$. Applying Lemma \ref{lem:SummandLemma}, we obtain a special 1-deformation retract of $M''_{R3}$ in which the first summand at 101 and the entire 111 piece are removed, while all the differentials between the remaining pieces remain untouched. The same lemma can be applied again to the $-1$ component of $100\to 101$, this time creating new arrows along $000\to 110$ and $001\to 110$. The multifactorization $M'''_{R3}$ obtained after these two cancellations appears at the top of Figure \ref{fig:r3simp}.

	Now, the two decorated tangle diagrams in \eqref{eq:R3} are related by a $180^\circ$ rotation, followed by a coloring inversion and a swap of all strand orientations. Let us denote by $N_{R3}$ the multifactorization associated to the right-hand side in \eqref{eq:R3}. The vertical part $\underline N_{R3}$ is obtained from $\underline M_{R3}$ by doing $180^\circ$ rotation, coloring inversion and reversal of arc directions at each resolved diagram. On the terms defining length 1 and 2 arrows, the total transformation has the following effect:
	\[
		x_0\leftrightarrow -x_3, \quad x_1\leftrightarrow -x_4, \quad x_2\leftrightarrow -x_5, \quad u\leftrightarrow -u, \quad G\leftrightarrow -G.
	\]
	These transformations can be carried throughout the special deformation retracts, using the algebraic recipe, to transform arrows $d_k$ of length $k>0$. Applying this total transformation to $M_{R3}'''$, one obtains the multifactorization $N_{R3}'''$ at the bottom of Figure \ref{fig:r3simp}. These two simplified multifactorizations are related by exchanging labels everywhere and then flipping the sign of all differentials. More precisely, $M'''_{R3}\cong N_{R3}'''$ by a vertical isomorphism which has a component of the form $s_2s_3\tau_{2,3}$ between each decorated diagram in $M'''_{R3}$ and the corresponding decorated diagram in $N'''_{R3}$ with the same underlying crossingless matching.

\end{proof}

\bibliographystyle{alpha}
\bibliography{biblio}

\end{document}